\documentclass{daj}

\usepackage{amsmath}
\usepackage{amssymb}
\usepackage{amsthm}
\usepackage{eucal}
\usepackage{epstopdf}
\usepackage{mathrsfs}
\usepackage{enumerate}
\usepackage{fullpage} 
\usepackage{setspace}
\usepackage{xcolor}
\usepackage{dsfont}

\newcommand{\e}{\varepsilon}

\newcommand{\cA}{\mathcal{A}}
\newcommand{\cC}{\mathcal{C}}
\newcommand{\R}{\mathbb{R}}
\newcommand{\C}{\mathbb{C}}
\newcommand{\Z}{\mathbb{ Z}}

\newcommand{\I}{\mathcal{I}}
\newcommand{\J}{\mathcal{J}}

\newcommand{\tf}{\tfrac{1}{2}}
\newcommand{\X}{X}
\newcommand{\bG}{{\bf G}}
\newcommand{\bZ}{{\bf Z}}

\newcommand{\oh}{\mbox{$\frac{1}{2}$}}

\newtheorem{prop}{Proposition}[section]
\newtheorem{thm}[prop]{Theorem}
\newtheorem{cor}[prop]{Corollary}
\newtheorem{lem}[prop]{Lemma}
\newtheorem{defn}{Definition}

\newtheorem{conj}{Conjecture}

\numberwithin{equation}{section}

\dajAUTHORdetails{%
  title = {The sixth moment of the Riemann zeta function and  ternary additive divisor sums}, 
  author = {Nathan Ng},
  plaintextauthor = {Nathan Ng},
  runningtitle = {Sixth moment of zeta and ternary additive divisor sums}, 
  keywords = {Moments of the Riemann zeta function, additive divisor sums},
}   

\dajEDITORdetails{%
   year={2021},
   number={6},
   received={4 June 2020},   
   published={6 July 2021},  
   doi={10.19086/da.22057},       
}   

\begin{document}

\begin{frontmatter}[classification=text]

\title{The sixth moment of the Riemann zeta function and  ternary additive divisor sums} 

\author[ng]{Nathan Ng\thanks{The research for this article was supported by an
NSERC Discovery Grant.}}

\begin{abstract}
Hardy and Littlewood initiated the study of the $2k$-th moments of the Riemann zeta function on the critical line. 
In 1918 Hardy and Littlewood established an asymptotic formula for the second moment and in 1926 
Ingham established an asymptotic formula for the fourth moment.  Since then no other moments have been 
asymptotically evaluated.  In this article we study the sixth moment of the zeta function on the critical line. 
We show that a conjectural formula for a certain family of ternary additive divisor sums implies an asymptotic formula
with power savings error term
for the sixth moment of the Riemann zeta function on the critical line.  This provides a rigorous proof for a heuristic argument of Conrey and Gonek \cite{CG}.  Furthermore, this gives some evidence towards
a conjecture of Conrey, Keating, Farmer, Rubinstein, and Snaith \cite{CFKRS} on shifted moments of the Riemann
zeta function.  In addition, this improves
on a theorem of Ivic \cite{Iv2}, who obtained an upper bound for the the sixth moment of the zeta function, based on the assumption of a conjectural formula for correlation sums of the triple divisor function. 
\end{abstract}
\end{frontmatter}


\section{Introduction}

The $2k$-th moment of the Riemann zeta function is 
\begin{equation}
   \label{IkT}
   I_{k}(T) = \int_{0}^{T} |\zeta(\tfrac{1}{2}+it)|^{2k} dt,
\end{equation}
where $\zeta$ denotes the Riemann zeta function and $k > 0$.  
This article concerns the behaviour of \eqref{IkT} in the case $k=3$. 
Hardy and Littlewood initiated the study of the moments \eqref{IkT}.  Their interest in these moments
arose from the relationship of the $I_k(T)$ to the Lindel\"{o}f hypothesis, which asserts that for any $\varepsilon >0$
one has 
$|\zeta(\tfrac{1}{2}+it)| \ll_{\varepsilon} t^{\varepsilon}$.  In fact, the Lindel\"{o}f hypothesis 
is equivalent to the statement, for any $\varepsilon >0$, $I_k(T) \ll_{\varepsilon} T^{1+\varepsilon}$
for all $k \in \mathbb{N}$.  
The motivation for studying the moments $I_k(T)$ is that it seems that 
it might be easier to obtain an average bound of $\zeta(\frac{1}{2}+it)$ rather than a pointwise bound. 
In 1918, Hardy and Littlewood \cite{HL} proved that 
\begin{equation}
 \label{I1}
  \int_{0}^{T} |\zeta(\tfrac{1}{2}+it)|^{2} dt \sim T(\log T) 
\end{equation}
and in 1926 Ingham \cite{In} proved that 
\begin{equation}
\label{I2}
  \int_{0}^{T} |\zeta(\tfrac{1}{2}+it)|^{4} dt \sim \frac{T}{2 \pi^2} (\log T)^4. 
\end{equation}
To date these are the only asymptotic results established for $I_k(T)$. 
In 1996, Conrey and Ghosh \cite{CGh} conjectured that 
\begin{equation}
  \label{I3}
  \int_{0}^{T} |\zeta(\tfrac{1}{2}+it)|^{6} dt \sim \frac{42 a_3}{9!}T(\log T)^9
\end{equation}
and in 1998,  Conrey and Gonek  \cite{CG} conjectured that
\begin{equation}
  \label{I4}
  \int_{0}^{T} |\zeta(\tfrac{1}{2}+it)|^{8} dt \sim \frac{24024 a_4}{16!}T(\log T)^{16}
\end{equation}
for certain specific constants $a_3$ and $a_4$ (see \eqref{ak} below).
In 1998, 
Keating and Snaith \cite{KS} conjectured  for all positive integers $k$ that
\begin{equation}
  \label{ksasymp}
I_k(T) \sim \frac{g_k a_k}{(k^2)!} T (\log T)^{k^2}
\end{equation}
where 
\begin{equation*}
  \label{gk}
  g_k =  k^2! \prod_{j=0}^{k-1} \frac{ j!}{(k+j)!}  
\end{equation*}
and
\begin{equation}
  \label{ak}
  a_k = \prod_{p} \Big( 1-\frac{1}{p} \Big)^{k^2} \sum_{m=0}^{\infty} \Big( \frac{\Gamma(m+k)}{m! \Gamma(k)} \Big)^2 p^{-m}.
\end{equation}
Note that \eqref{ksasymp} agrees with \eqref{I1}, \eqref{I2}, \eqref{I3}, and \eqref{I4}.  

The formulae \eqref{I1} and \eqref{I2} have been refined to asymptotic formulae with 
error terms admitting power savings.  
In 1926, Ingham \cite{In} showed that 
\begin{equation} \label{I1P}
   I_1(T) = T\mathcal{P}_1(\log T) + O(T^{\frac{1}{2}+\varepsilon}), \text{ for any } \varepsilon >0, 
\end{equation}
and in 1979, Heath-Brown  \cite{HB} showed that 
\begin{equation} \label{I2P}
  I_2(T) = T\mathcal{P}_4(\log T) + O(T^{\frac{7}{8}+\varepsilon}), \text{ for any } \varepsilon >0, 
\end{equation}
where $\mathcal{P}_1,\mathcal{P}_4$ are polynomials of degrees 1 and 4 respectively.  
The error term in \eqref{I1P} has been improved numerous times and the current record is due 
to Bourgain and Watt \cite{BW} who showed $O(T^{\frac{1515}{4816}+\varepsilon})$. 
Zavorotny\u{\i} \cite{Zav}  improved the error term in   \eqref{I2P} to $O(T^{\frac{2}{3}+\varepsilon})$ and this was improved to 
$O(T^{\frac{2}{3}} (\log T)^{8})$  (see \cite{Mo}, \cite{IM}).  

Although the asymptotic \eqref{ksasymp} remains open for $k \ge 3$, there are a number of results providing upper and lower bounds.
Ramachandra \cite{Ra1},\cite{Ra2}  established that $I_{k}(T) \gg_{k} T (\log T)^{k^2}$ for positive integers $2k$ and this was extended to rational $k$ by 
Heath-Brown \cite{HB2}.  Further, assuming the Riemann hypothesis, this was established for all real $k > 0$ by Ramachandra \cite{Ra2},\cite{Ra3} and 
independently Heath-Brown \cite{HB2}. 
Recently, Radziwi\l\l  \, and Soundararjan \cite{RS} obtained the correct lower bound unconditionally for all $k >0$.  In 2008, Soundararajan \cite{S2} showed 
that on the Riemann hypothesis that $I_{k}(T) \ll_k T (\log T)^{k^2+\varepsilon}$, for any $\varepsilon >0$.   Building on this work
and introducing a number of new ideas, Harper \cite{Ha} showed that the Riemann hypothesis implies $I_{k}(T) \ll_k T (\log T)^{k^2}$.
Prior to Harper's work, Radziwi\l\l  \, \cite{R} obtained this bound for $k < 2.18$, assuming the Riemann hypothesis. 

In order to study the mean value $I_k(T)$ it is convenient to consider the shifted moments
\begin{equation}
   \label{IIJT}
    I_{\I,\J}(\omega) = 
    \int_{-\infty}^{\infty}
   \Big( \prod_{j=1}^{k}
     \zeta(\oh+a_j +it)
    \zeta(\oh+b_j -it) \Big)
  \omega(t)dt
\end{equation}
where $\I = \{ a_1, \ldots , a_k \}$ and  $\J = \{ b_1, \ldots, b_k \}$ are $k$-tuples of complex numbers,
while $\omega$ is a suitable smooth function with support $[\alpha,\beta]$
satisfying $T \ll \alpha \ll \beta-\alpha \ll T$. 
The elements of $\I$ and $\J$ are called shifts.  Observe that 
$I_k(T) = I_{\{0, \ldots, 0 \}, \{ 0, \ldots, 0 \}}(\omega) 
 \text{ where } \omega = \mathds{1}_{[0,T]}(t)$
 \footnote{ $\mathds{1}_B(t)$ denotes the indicator function corresponding to $B \subset \mathbb{R}$.}.  It turns out that it can be technically simpler to study $ I_{\I,\J}(\omega)$
rather than $I_k(T)$.   It appears that this approach was first used by Ingham \cite{In} in the case $|\I|=|\J|=1$.
What the shifts $\I$ and $\J$ do is that they make the main term easier to compute.   
The main term of $ I_{\I,\J}(\omega)$ may be written as a certain contour integral of a meromorphic function.
If the shifts in $\I$ and $\J$ are distinct, then the poles of this meromorphic function are simple and thus are simpler to compute.  On the other hand, the shifts introduce a certain combinatorial problem.

 Conrey et al. \cite{CFKRS}  developed a conjecture for shifted moments
of the Riemann zeta function. This shall be described shortly.    
We now explain the conjecture of \cite{CFKRS} for this mean value, but we shall follow
the notation of \cite{HY}. In order to do this, we first define several functions. 

\begin{defn}  \label{zetaX} Let $X$ be a finite multiset of complex numbers.  We define the arithmetic function
$\sigma_{X}(n)$ to be the coefficient of $n^{-s}$ in the Dirichlet series $\zeta_X(s)$, defined by 
$\zeta_X(s):=\prod_{x \in X} \zeta(s+x)$.
In other words, if $X = \{x_1, \ldots, x_k \}$ then $
  \sigma_{X}(n) = \sum_{n_1 \cdots n_k=n} n_{1}^{-x_1} \cdots n_{k}^{-x_k}$.   
\end{defn}
Observe that 
\begin{equation*}
\begin{split}
   \zeta_{\{0, \ldots, 0\}}(s) & =\zeta(s)^k \text{ and } \\
   \sigma_{\{0, \ldots, 0\}}(n) & = d_k(n) = \sum_{d_1 \cdots d_k=n} 1 \end{split}
\end{equation*}
is the $k$-th divisor function 
where $k=    \# \{0,\ldots, 0 \}$.
Thus if, the elements of $X$ are close
to zero and $k=\#X$, then $\zeta_X(s)$ is approximately  $\zeta^k(s)$
and $\sigma_{X}(n)$ is approximately $d_k(n)$. 
\begin{defn}  \label{ZXYs} Given finite multisets $X,Y$ of complex numbers we define the Dirichlet series 
\[
  \mathcal{Z}_{X,Y}(s) := \sum_{n=1}^{\infty} \frac{\sigma_X(n) \sigma_Y(n)}{n^{1+s}}.
\]
\end{defn}
The series $\mathcal{Z}_{X,Y}(s)$  plays an important role in the study of $I_{\I,\J}(\omega)$ and
will occur frequently in this article.  It should be noted that $\mathcal{Z}_{X,Y}(s)$ has an analytic continuation to the left of $\Re(s)=0$. 
In fact, 
\[
   \mathcal{Z}_{X,Y}(s) =  \Big( \prod_{x \in X, y \in Y} \zeta(1+s+x+y) \Big) \cA_{X,Y}(s)
\]
where $\cA_{X,Y}(s)$ is holomorphic in a half-plane containing $s=0$.  
Precise formulae for $\cA_{X,Y}(s)$, in the case $|X|=|Y|=3$,  are given in Lemma \ref{ZIJ} which follows. 
For $|X|=|Y| \le 2$, there are simple formulae for $ \mathcal{Z}_{X,Y}(s)$.
\\

\noindent {\bf Examples}. (i) Let $X=\{x_1\}$ and $Y=\{y_1\}$.  Then 
$\mathcal{Z}_{X,Y}(s) = \zeta(1+s+x_1+y_1).$ \\
(ii)   Let $X=\{x_1,x_2\}$ and $Y=\{y_1,y_2\}$. Then a calculation, using a formula of Ramanujan, establishes that 
\[
   \mathcal{Z}_{X,Y}(s) = \frac{\zeta(1+s+x_1+y_1)\zeta(1+s+x_1+y_2)\zeta(1+s+x_2+y_1)\zeta(1+s+x_2+y_2)}{\zeta(2+2s+x_1+x_2+y_1+y_2)}.
\]
In order to formulate the conjecture on the size of $I_{\I,\J}(\omega)$, we require a definition
concerning set operations on $\I$ and $\J$.  
\begin{defn}[$j$-swaps] \label{IJST}
Let  $\I=\{a_1, \ldots, a_k\}$, let $\J = \{b_1, \ldots, b_k \}$, and let $0 \le j \le k$.
Let $\mathcal{S} \subset \I$ and $\mathcal{T} \subset \J$ be such that $|\mathcal{S}|=|\mathcal{T}|=j$.  
We write
$$ \mathcal{S} = \{ a_{i_1}, \ldots, a_{i_j} \}
  \text{ and }
\mathcal{T}= \{ b_{l_1}, \ldots, b_{l_j} \}$$
 where 
$$ i_1 < \cdots < i_j \text{ and } 
l_1 < \cdots < l_j.$$ 
We shall write $\I_{\mathcal{S}}$ to denote the $k$-tuple obtained from $\{ a_1, \ldots, a_k \}$
by replacing $a_{i_r}$ with $-b_{l_r}$ for $1 \le r \le j$.  Similarly, 
we write $\J_{\mathcal{T}}$  to denote the $k$-tuple obtained from $\{ b_1, \ldots, b_k \}$
by replacing $b_{l_r}$ with $-a_{i_r}$ for $1 \le r \le j$. 
\end{defn}
These are called $j$-swaps since $j$ members of $\I$ and $\J$ are swapped with the negative signs inserted.
This terminology is introduced in the series of articles \cite{CK1}, \cite{CK2}, \cite{CK3}, and \cite{CK4}. 
In order to explain this we give some simple examples. 

\noindent {\bf Examples}.
Let  $\I=\{a_1,a_2, a_3\}$, $\J = \{b_1,b_2, b_3 \}$.  If $\mathcal{S}=\emptyset$ and $\mathcal{T}=\emptyset$, then 
$\I_\mathcal{S} = \I$ and $\J=\J_\mathcal{T}$. If $\mathcal{S}= \{a_1 \}$ and $\mathcal{T}=\{b_3 \}$, then 
$\I_\mathcal{S}=(-b_3,a_2,a_3)$ and $\J_\mathcal{T}=(b_1,b_2,-a_1)$.
If $\mathcal{S}=\{ a_1, a_3\}$ and $\mathcal{T}=\{ b_2, b_3 \}$, then 
$\I_\mathcal{S}=(-b_2, a_2,-b_3)$  and $\J_\mathcal{T}= (b_1, -a_1, -a_3)$. 
If $\mathcal{S}=\I$ and $\mathcal{T}=\J$, then 
$\I_\mathcal{S}=-\J$ and $ \J_\mathcal{T}=-\I$.   The cases of $|\mathcal{S}|=|\mathcal{T}|=0$ are  called 0-swaps, the 
cases of $|\mathcal{S}|=|\mathcal{T}|=1$ are  called 1-swaps, and in general the cases of $|\mathcal{S}|
=|\mathcal{T}|=j$ are  called $j$-swaps.  

We are now prepared to state the  conjecture of Conrey, Farmer, Keating, Rubinstein, and Snaith \cite{CFKRS} for  $I_{\I,\J}(\omega)$ when $\omega$ is a  convenient weight function, which we now define.  
Let $\varepsilon_0 \in (0,\frac{1}{4}]$ be a positive absolute constant, and let
$\omega$ be a function from $\mathbb{R}$ to $\mathbb{C}$  that satisfies the following:
\begin{align}
 & \label{cond1}  \omega \text{ is smooth}, \\
& \label{cond2} \text{the support of }  \omega  \text{ lies in } [c_1T,c_2T]  \text{ where } 0 < c_1 < c_2
\text{ are positive absolute constants}, \\
& \label{cond3} \text{there exists } T_0  \ge T^{\frac{3}{4}+\varepsilon_0} 
\text{ such that } T_0 \ll T \text{ and }
  \omega^{(j)}(t) \ll T_{0}^{-j} \text{ for } j \in \mathbb{N} \cup \{ 0 \} \text{ and all }  t \in \R. 
\end{align}
\begin{conj}  \label{cfkrsconj}
Let $T >3$ and suppose that $\omega$ satisfies \eqref{cond1}, \eqref{cond2}, and \eqref{cond3}. 
Assume that  $\I=\{a_1, \ldots, a_k\} \subset \mathbb{C}$
and $\J= \{b_1, \ldots, b_k \} \subset \mathbb{C}$ satisfy $|a_i|, |b_i| \ll (\log T)^{-1}$ for 
$i=1, \ldots, k$ (where the implicit constants are absolute).  Then, in those cases where $T$ is sufficiently 
large (in absolute terms), one has
\begin{equation}
   \label{shiftedconjecture}
   I_{\I,\J}(\omega) = 
   \int_{-\infty}^{\infty} 
    \Big(
    \sum_{\substack{\mathcal{S} \subset \I, \mathcal{T} \subset \J \\ |\mathcal{S}|=|\mathcal{T}|  }  }
   \mathcal{Z}_{\I_{\mathcal{S}},\J_{\mathcal{T}}}(0) 
   \Big( \frac{t}{2 \pi} \Big)^{-\mathcal{S}-\mathcal{T}} + o(1) \Big) \omega(t) dt
\end{equation}
where we have defined
\begin{equation*}
  \label{tSTdefn}
  \Big( \frac{t}{2 \pi} \Big)^{-\mathcal{S}-\mathcal{T}} :=  \Big( \frac{t}{2 \pi} \Big)^{-\sum_{x \in \mathcal{S}} x - \sum_{y \in \mathcal{T}} y}
\end{equation*}
for $\mathcal{S}  \subset \I$ and $\mathcal{T} \subset \J$. 
\end{conj}

\noindent {\bf Remarks}  
\begin{enumerate}
\item The works of Ingham \cite{In} and Motohashi \cite{Mo} establish this conjecture in the cases $|\I|=|\J|=1$
and $|\I|=|\J|=2$, respectively.
\item Conrey, Farmer, Keating, Rubinstein, and Snaith \cite{CFKRS} made this conjecture with the $o(1)$ 
replaced by $O(T^{-\frac{1}{2}+\varepsilon})$. That is,  the total error with the weight included is $O(T^{\frac{1}{2}+\varepsilon})$.   They gave a heuristic argument based on a ``recipe" (see \cite[section 2.2, pp. 53-56]{CFKRS}).
\item It should be observed that if $a_1=a_2$, or if $a_1=-b_1$ then some of the terms $ \mathcal{Z}_{\I_{\mathcal{S}},\J_{\mathcal{T}}}(0)$ are undefined due to the pole of $\zeta(s)$  at $s=1$.   However, these polar 
terms cancel out as can be seen by Lemma 2.5.1 of \cite{CFKRS}, which shows that the right hand side 
of   \eqref{shiftedconjecture} is holomorphic in the $a_i$'s and $b_i$'s, as long as $|a_i|, |b_i| \le \eta$
for a sufficiently small positive $\eta$. 
\item The authors of \cite{CFKRS} make the conjecture conditional on the weight function $\omega$
being ``suitable",   without actually defining what they mean by that.  We have stated the conjecture 
for functions $\omega$ satisfying \eqref{cond1}, \eqref{cond2}, and \eqref{cond3}.  
However,  Conjecture \ref{cfkrsconj} should hold for a much wider class
of functions.  For instance, it should also be true when $\omega(t) = \phi(t/T)$ where $\phi$ 
is a Schwarz class function. 
\item There is some debate on the size of the error term in this conjecture.
In the case $k=3$ and all shifts $a_i=b_j=0$, namely $I_3(T)$, 
Motohashi \cite[p.218, eq. (5.4.10)]{Mo} has conjectured that the error term is $\Omega(T^{\frac{3}{4}-\delta})$ for any fixed $\delta >0$.
Similarly, Ivic \cite[p. 171]{Iv1} has conjectured that the error term in this case is $O(T^{\frac{3}{4}+\varepsilon})$ and $\Omega( T^{\frac{3}{4}} )$.
Zhang \cite{Z} studied a related mean value (the cubic moment of quadratic Dirichlet $L$-functions at the central point)
and found a main term plus secondary term of size
$T^{\frac{3}{4}}$,  assuming certain meromorphicity and polynomial growth assumptions for the corresponding multiple 
Dirichlet series.   Recently,  Diaconu \cite{Di}  and Diaconu and Whitehead \cite{DW} have given a proof of this 
in both the function field and number field cases. 
\item Recent numerical calculations of these moments (see \cite{RuYa}) currently do not seem give to conclusive evidence of what is the correct size 
for the error term.  
\end{enumerate}

In this article we shall prove that a certain asymptotic formula for ternary additive divisor sums implies an asymptotic formula 
for $ I_{\I,\J}(\omega)$ in the case $|\I|=|\J|=3$.   In the remainder of this article we consider $\I, \J$  where 
\begin{equation*}
  \label{IJ}
  \I = \{ a_1, a_2, a_3 \}  \text{ and } 
   \J = \{ b_1, b_2, b_3 \}.
\end{equation*}
We shall assume that the numbers $a_i,b_i$ $(i=1,2,3)$ satisfy a certain condition of the form
\begin{equation}
    \label{sizerestrictiondelta}
   |a_i|, |b_i| \le \delta \text{ for } i =1,2,3,
\end{equation} 
where $\delta$ denotes a positive constant satisfying $\delta < \frac{1}{2}$. 
At times we shall require the more restrictive size restriction
\begin{equation}
  \label{sizerestriction}
  |a_i|, |b_i| \ll \frac{1}{\log T} \text{ for } i =1,2,3.
\end{equation}
In our proof of Theorem \ref{mainthm} and in our proofs of various
intermediate results we shall assume that $T$ is sufficiently large.  With this assumption, 
observe that if $a_i$ and $b_i$ satisfy \eqref{sizerestriction}, then they will
automatically satisfy \eqref{sizerestrictiondelta}. 
The family of additive divisor sums we are concerned with are sums
\begin{equation}
  \label{DfIJ}
  D_{f;\I,\J}(r) = \sum_{m-n=r} \sigma_{\I}(m) \sigma_{\J}(n) f(m,n)
\end{equation}
where $r \in \mathbb{Z} \setminus \{ 0 \}$ and $f$ is a smooth function. 
We require that, the partial derivatives of $f$ satisfy growth conditions. 
That is,  there must exist $X,Y,$ and $P$ positive such that 
\begin{equation}
  \label{fsupport}
  \text{support}(f) \subset [X,2X] \times [Y,2Y]
\end{equation}
and the partial derivatives of $f$ satisfy
\begin{equation} \label{fcond} 
 x^{m} y^{n} f^{(m,n)}(x,y) \ll_{m,n}
P^{m+n}. 
\end{equation}
In order to state a conjecture for the size of $ D_{f;\I,\J}(r)$, we must introduce several multiplicative functions. 
\begin{defn}
 Let $X= \{ x_1, \ldots, x_k\}$ be a finite multiset of complex numbers and  let $s \in \mathbb{C}$
 satisfy $\Re(s) > -\min_{j=1, \ldots, k} \Re(x_j)$. 
The multiplicative function $n \to g_X(s,n)$ is given by
 \begin{equation}
  \label{gXdefn}
   g_{X}(s,n)
=
\prod_{p^{\alpha} \mid \mid n} 
  \frac{\sum_{j=0}^{\infty} \frac{\sigma_{X}(p^{j+\alpha})}{p^{js}}}{ \sum_{j=0}^{\infty} \frac{\sigma_{X}(p^j)}{p^{js}}}.
\end{equation}
In other words, for $n \in \mathbb{N}$ we have 
$
  \sum_{m=1}^{\infty} \frac{\sigma_{X}(nm) }{m^s}
  =  g_X(s,n)
  \zeta(s+x_1) \cdots \zeta(s+x_k)$.
  
The multiplicative function $n \to G_X(s,n)$ is given by 
\begin{equation}
  \label{GXdefn}
  G_{X}(s,n) =  \sum_{d \mid n} \frac{\mu(d) d^s}{\phi(d)}    \sum_{e \mid d} \frac{\mu(e)}{e^s} g_{X} \Big(s,\frac{ne}{d} \Big).
\end{equation}
\end{defn}
Finally, we shall use the standard notation $e(\theta) := e^{2 \pi i \theta}$ and  we recall that 
 the Ramanujan sum is defined by $ c_{\ell}(r)= \sum_{\substack{a=1 \\ (a,\ell)=1}}^{\ell}
e ( \frac{ar}{\ell})$.
With these definitions in hand, we may state the additive divisor conjecture.

\begin{conj} \label{divconj}
There exists a pair $(\vartheta, C) \in [\frac{1}{2}, \frac{2}{3}) \times [0,\infty)$ for which 
the following (henceforth to be referred to as $\mathcal{AD}(\vartheta,C)$, or the 
`additive divisor hypothesis') holds.  Let $\varepsilon$ and $\varepsilon_2$ be positive
absolute constants.  Let $P > 1$, and let $X,Y > \frac{1}{2}$ satisfy $Y \asymp X$.  Let $f$
be a smooth function satisfying \eqref{fsupport} and \eqref{fcond} and suppose 
$\I = \{ a_1, a_2, a_3 \}$ and $\J = \{ b_1, b_2, b_3 \}$ are sets of distinct complex numbers 
satisfying $|a_i|, |b_i| \ll (\log X)^{-1}$ for $i=1,2,3$ (where the implicit constants are absolute).
Then, in those cases where $X$ is sufficiently large (in absolute terms), one has
  \begin{align*}
   &   D_{f;\I,\J}(r)  =   \sum_{i_1=1}^{3} \sum_{i_2=1}^{3} 
    \prod_{j_1 \ne i_1} \zeta(1-a_{i_1}+a_{j_1})   \prod_{j_2 \ne i_2} \zeta(1-b_{i_2}+b_{j_2}) \cdot  \\
  & \sum_{\ell=1}^{\infty} \frac{c_{\ell}(r)G_{\I}(1-a_{i_1},\ell)G_{\J}(1-b_{i_2},\ell)  }{\ell^{2-a_{i_1}-b_{i_2}}}
  \int_{\max(0,r)}^{\infty}  f(x,x-r) x^{-a_{i_1}}(x-r)^{-b_{i_2}} dx 
+ O ( P^{C} X^{\vartheta+\varepsilon} ), 
 \end{align*}
 uniformly for $1 \le |r| \ll X^{\frac{1}{2}-\varepsilon_2}$.  
\end{conj}
\noindent {\bf Remarks}. 
\begin{enumerate}
\item The main term in the above conjecture can be derived by following Duke, Friedlander, and Iwaniec's $\delta$-method
\cite{DFI}.   
\item  The function  $D_{f;\I,\J}(r)$ is a smoothed and shifted variant of the classical additive divisor sum
\begin{equation*} 
   \label{Dkxr}
     D_k(x,r) =\sum_{n \le x} d_k(n)d_k(n+r). 
\end{equation*}
The leading term in the conjecture for $D_k(x,r)$ can be worked out with a heuristic
probabilistic calculation.  This was recently done independently by  Tao \cite{Tao} and Ng and Thom \cite{NT}.
\item   In the case that $\sigma_{\I}(n)=\sigma_{\J}(n)=d(n)$, the divisor function,
Duke, Friedlander, and Iwaniec \cite{DFI} have shown that an analogous result is available
with an error term having $\vartheta =\frac{3}{4}$ and $C= \frac{5}{4}$.   Furthermore, they mention that improvements to their argument would reduce $C$ to $\frac{3}{4}$ and more elaborate arguments may lead to $C=\frac{1}{2}$.  
\item  Conrey and Gonek \cite{CG} have conjectured that in the case of $D_3(x,r)$ (an unsmoothed version of $D_{f;\I,\J}(r)$) that 
$\vartheta=\frac{1}{2}$ is valid for $1 \le r \le \sqrt{x}$.     
Moreover, Conrey and Keating \cite{CK3} have suggested that  $\vartheta=\frac{1}{2}$ is valid for $1 \le r \le x^{1-\varepsilon}$. 
 This is discussed extensively in  Ng and Thom \cite{NT} where a probabilistic argument has been given
 which suggests the error term for $D_k(x,r)$ is uniform in the range $1 \le r \le x^{1-\varepsilon}$.   
 Hence it is likely that the above conjecture holds in the wider range $r \le X^{1-\varepsilon_2}$. 
\item  Blomer \cite{Bl} has shown that there exists $C>0$ such that 
\[
   \sum_{\ell_1m-\ell_2n=h} a(m) \overline{a(n)} f(m,n) \ll P^{C}  X^{\frac{1}{2}+\Theta+\varepsilon}
\]
where $g(z) = \sum_{m=1}^{\infty} a(m) m^{\frac{k-1}{2}} e(mz) \in S_k(N,\chi)$ is  a primitive cusp form
(holomorphic newform) and $\Theta$ is a non-negative constant such that $|\lambda(n)| \ll n^{\Theta}$ 
 for eigenvalues $\lambda(n)$ of the Hecke operator $T_n$ acting on the space of weight 0 Maass cusp forms of level $N$.
\item Recently, Aryan \cite{Ar} has shown in the case that $\sigma_{\I}(n)=\sigma_{\J}(n)=d(n)$, $X=Y$, and $P=1$, 
that the corresponding error term is $O(X^{\frac{1}{2}+\Theta+\varepsilon})$.  
\item Unfortunately, for $k \ge 3$ an asymptotic evaluation of $D_k(x,r)$
 currently remains an open problem.  In the case of the unsmoothed sum
$D_k(x,r)$  uniform upper and lower bounds for $r \le x^{A}$, for $A>0$, of the correct order
of magnitude are known. Ng and Thom \cite{NT}  established lower bounds
and Henriot \cite{He} established upper bounds. 
Recently, K. Matom\"aki, M.  Radziwi\l\l, and T. Tao \cite{MRT}, \cite{MRT2} have established the
expected asymptotic for $D_{k}(x,r)$ for almost all $r$ in certain ranges.
\end{enumerate}
The main goal of this article is to show that Conjecture \ref{divconj} implies the case $k=3$ of Conjecture \ref{cfkrsconj}
for those choices of $\omega$ which satisfy  \eqref{cond1}, \eqref{cond2},  \eqref{cond3}, 
$T_0 \gg T^{b}$ where $b > \frac{C+3 \vartheta/2}{C+1}$, and $\int_{0}^{\infty} |\omega(t)| \, dt \gg T$. 
\begin{thm} \label{mainthm}
Let $\I = \{ a_1,a_2,a_3 \}$, $\J=\{ b_1,b_2,b_3 \}$, and assume the elements of $\I$ and $\J$ satisfy \eqref{sizerestriction}.   Let $\omega$ satisfy \eqref{cond1}, \eqref{cond2}, and 
\eqref{cond3}. 
Suppose $\e$ is a positive constant, and that $\vartheta \in [\frac{1}{2},\frac{2}{3})$ and $C >0$ are such that the hypothesis
$\mathcal{AD}(\vartheta,C)$ holds.  Then, provided that $T$ is sufficiently large, one has
\begin{equation}
\begin{split}
   \label{mainthmformula}
   I_{\I,\J}(\omega)   & = 
   \int_{-\infty}^{\infty} 
   \Big(
     \sum_{\substack{\mathcal{S} \subset \I, \mathcal{T} \subset \J \\ |\mathcal{S}|=|\mathcal{T}|  }  }
   \mathcal{Z}_{\I_{\mathcal{S}},\J_{\mathcal{T}}}(0) 
   \Big( \frac{t}{2 \pi} \Big)^{-\mathcal{S}-\mathcal{T}}  \Big) \omega(t) dt  
   +O \Big( T^{\frac{3 \vartheta}{2}+\varepsilon} \Big( \frac{T}{T_0} \Big)^{1+C}   
      \Big).
\end{split}
\end{equation}  
\end{thm}
From this theorem, we deduce an asymptotic formula with power savings error term for the sixth moment of the 
Riemann zeta function.  
\begin{cor}  \label{maincor}
If Conjecture \ref{divconj}, $\mathcal{AD}(\vartheta,C)$, is true for some constants
$\vartheta \in [\frac{1}{2},\frac{2}{3})$ and $C \ge 0$, 
then there exists a polynomial $\mathcal{P}_9(x)$ of degree 9 such that, for any $\varepsilon >0$, 
\begin{equation*}
  \label{I3Tasymptotic}
   I_3(T) = \int_{0}^{T} |\zeta(\tf+it)|^6dt = T\mathcal{P}_9(\log T) + O(T^{\frac{\frac{3\vartheta}{2}+1+C}{2+C}+\varepsilon}),
\end{equation*}
as long as $T$ is sufficiently large. 
\end{cor}
\noindent {\bf Remarks}. 
\begin{enumerate}
\item Conditionally, this confirms Conjecture \ref{cfkrsconj} and provides an asymptotic formula
as long as $\vartheta < \frac{2}{3}$. 
\item  Assuming a stronger version of Conjecture \ref{divconj} in which the conditions $|a_i|,|b_i| \le (\log X)^{-1}$ 
$(i=1,2,3)$
are replaced by $|a_i|,|b_i| \le \delta_0$ $(i=1,2,3)$, with $\delta_0$ chosen to be some (sufficiently small) positive
constant,  we can deduce a correspondingly strengthened version of Theorem \ref{mainthm}, in which the condition
\eqref{sizerestriction} is replaced by a weaker condition of the form \eqref{sizerestrictiondelta}. 
\item This result makes rigorous the argument of Conrey and Gonek in \cite{CG}. 
In their work they argued that the $I_3(T)$ is asymptotic to 
the sum of mean values of  the shape $\int_{T}^{2T} |\mathbb{D}_{T^{\theta_i}}(\tfrac{1}{2}+it)|^2 dt$  $(i=1,2)$
where $\mathbb{D}_{T^{\theta_i}}(s) = \sum_{n \le T^{\theta_i}} d_3(n) n^{-s}$ and $\theta_1+\theta_2=3$.  They then invoked a Theorem 
of Goldston and Gonek \cite{GG} to asymptotically evaluate these expressions.     
This required a certain conjectural formula for 
$D_3(x,r)$
with sharp error terms, uniform for $r \le \sqrt{x}$.  
\item In a sense, this improves work of Ivi\'{c} \cite{Iv2}, who showed 
that certain asymptotic formula  for $D_3(x,r)$  implies $I_3(T) \ll T^{1+\varepsilon}$ for any 
$\varepsilon >0$. 
A slight difference in our treatment is that we have chosen to deal with smooth additive divisor sums corresponding
to $\sigma_{\I}(n)$ and $\sigma_{\J}(n)$ where the elements of $\I$ and $\J$ are $\ll (\log T)^{-1}$.   This is a mild assumption and in fact, a number of proofs for unsmoothed additive divisor sums 
are deduced from their smooth and shifted versions.
\item In our proof we follow an argument of Hughes and Young \cite{HY} who evaluated the twisted fourth moment
\[
  \int_{-\infty}^{\infty} \Big( \frac{h}{k} \Big)^{-it}  |\zeta(\tf+it)|^4 \omega(t) dt.
\]
for coprime natural numbers $h,k$ satisfying $hk \le T^{\frac{2}{11}-\varepsilon}$. 
Recently, this was improved by Bettin, Bui, Li, and Radziwi\l\l  \ \cite{BBLR} \ who extended the range to $\max(h,k) \le T^{\frac{1}{4}-\varepsilon}$. 
\item If the hypothesis $\mathcal{AD}(\vartheta,C)$ is true for the best possible exponent $\vartheta=\frac{1}{2}$
(and some $C>0$), then the error term in \eqref{mainthmformula} will be of size $O(T^{\frac{3}{4}+\varepsilon})$
when $T_0 =T^{1-\varepsilon}$.
Note this matches with the (already mentioned) speculations of Ivic and Motohashi  on the error term  for $I_3(T)$. 
Furthermore, under the hypothesis  $\mathcal{AD}(\tfrac{1}{2},C)$ the exponent in the error term in Corollary 
\ref{maincor} is $1- \frac{1}{8+4C} +\varepsilon$.  Note that this is greater than $ \tfrac{7}{8} +\varepsilon$, 
the exponent in \eqref{I2P}. 

\item Formulae and numerical values for the coefficients of $\mathcal{P}_9$ may be found in 
\cite{CFKRS}, \cite{CFKRS2008}. 
\item From Theorem \ref{mainthm} we can also deduce formulae for the integrals
$$\int_{-\infty}^{\infty} \zeta^{(j_1)}(\tfrac{1}{2}+it) \zeta^{(j_2)}(\tfrac{1}{2}+it)\zeta^{(j_3)}(\tfrac{1}{2}+it)
\zeta^{(j_4)}(\tfrac{1}{2}-it)\zeta^{(j_5)}(\tfrac{1}{2}-it)\zeta^{(j_6)}(\tfrac{1}{2}-it) \omega(t) dt,$$
where $(j_1, \ldots, j_6) \in \mathbb{Z}_{\ge 0}^6$,
assuming Conjecture \ref{divconj}.  Such integrals can be used in detecting large gaps between 
the zeros of the Riemann zeta function.  For instance see Hall \cite{Hall}. 
\end{enumerate}
\begin{proof}[Proof of Corollary \ref{maincor}]
In this proof $\I = \{ a_1,a_2,a_3 \}$ and  $\J= \{b_1, b_2,b_3 \}$ are each triples of complex numbers.   We also write
$\vec{a} = (a_1,a_2,a_3)$, 
and $\vec{b}=(b_1,b_2,b_3)$.   Set $f({\vec a};{\vec b}) = I_{\I,\J}(\omega)$  and
\begin{equation*}
   g({\vec a};{\vec b}) =  \int_{-\infty}^{\infty} 
   \omega(t) \Big(
     \sum_{\substack{\mathcal{S} \subset \I, \mathcal{T} \subset \J \\ |\mathcal{S}|=|\mathcal{T}|  }  }
   \mathcal{Z}_{\I_{\mathcal{S}},\J_{\mathcal{T}}}(0) 
   \Big( \frac{t}{2 \pi} \Big)^{-\mathcal{S}-\mathcal{T}}  \Big)dt. 
\end{equation*}
Note that $f({\vec a};{\vec b})$ is holomorphic in $a_i$ and $b_i$ as long as $|a_i| < \tfrac{1}{2}$ and $|b_i| < \tfrac{1}{2}$.  
Also by Lemma 2.5.1 of \cite{CFKRS} and \cite[Sections 4.4,4.5]{Br}  $g({\vec a};{\vec b})$ is holomorphic in $a_i$ and $b_i$ as long as $|a_i| < \eta$ and $|b_i| < \eta$ for a sufficiently small fixed $\eta>0$.   
It shall be convenient to  set $a_4=-b_1,a_5=-b_2$, and $a_6=-b_3$.  We have
(see \cite[p. 371]{CFKRS2003})
\begin{equation*}
  g({\vec a};{\vec b}) = \int_{-\infty}^{\infty} \omega(t) 
  \Big( \frac{t}{2 \pi} \Big)^{-\frac{1}{2} \sum_{j=1}^{3} (a_j+b_j)  }
    P(\log \tfrac{t}{2 \pi}, \vec{a},\vec{b}) 
   dt,
\end{equation*}
where 
\begin{equation*}
  P(x,\vec{a},\vec{b}) =   
  -\frac{e^{(-a_1-a_2-a_3-b_1-b_2-b_3) \tfrac{x}{2}  } }{(3!)^2 (2 \pi i)^6} \oint 
  \cdots \oint \frac{G(z_1, \ldots, z_6) \Delta^2(z_1, \ldots, z_6)}{\prod_{j=1}^{6} \prod_{i=1}^{6} (z_j-a_i)} 
   e^{(x/2) \sum_{j=1}^{3} (z_j-z_{3+j})  } dz_1 \ldots dz_6,
\end{equation*}
such that the integrals $\oint$ are over small, positively oriented circles, inside of each of which lie
all of the points $a_1, \ldots, a_6$, 
\begin{align*}
  \Delta(z_1, \ldots, z_6)  & =  \prod_{1 \le i < j \le 6 } (z_j-z_i), \\
    G(z_1, \ldots, z_6) & =   A(z_1, \ldots, z_6) \prod_{i=1}^{3} \prod_{j=1}^{3} \zeta( 1 + z_i-z_{3+j}), \\ 
    A(z_1, \ldots, z_6) & =  \prod_{p} \prod_{i=1}^{3} \prod_{j=1}^{3}
    \Big(1-\frac{1}{p^{1+z_i-z_{3+j}}} \Big) \int_{0}^{1} 
    \prod_{j=1}^{3}  \Big(1-\frac{e(\theta)}{p^{\frac{1}{2}+z_j}} \Big)^{-1}  
    \Big(1-\frac{e(-\theta)}{p^{\frac{1}{2}-z_{3+j}}} \Big)^{-1}. 
\end{align*} 
It follows that for $|a_i|, |b_i| < \eta$ that $F({\vec a};{\vec b}) = f({\vec a};{\vec b}) -g({\vec a};{\vec b})$ is holomorphic in 
each of the variables.  
It follows from Theorem \ref{mainthm} that
\[
   |F(\vec{0};\vec{0})| \ll   T^{\frac{3 \vartheta}{2}+\varepsilon} \Big( \frac{T}{T_0} \Big)^{1+C}    
\]
and so, when $\omega$ satisfies \eqref{cond1}, \eqref{cond2}, and \eqref{cond3}, one has
\begin{equation*}
\begin{split}
  \int_{-\infty}^{\infty} |\zeta(\tfrac{1}{2}+it)|^6 \omega(t) dt 
  & =  \int_{-\infty}^{\infty} P_9 \Big( \log \frac{t}{2 \pi} \Big)  \omega(t) dt 
   + O \Big( 
  T^{\frac{3 \vartheta}{2}+\varepsilon} \Big( \frac{T}{T_0} \Big)^{1+C}
  \Big)
\end{split}
\end{equation*}
where
\begin{equation*}
\begin{split}
  P_9(x) = -  \frac{1}{(3!)^2 (2 \pi i)^6} \oint \cdots \oint \frac{G(z_1, \ldots, z_6) \Delta^2(z_1, \ldots, z_6)}{\prod_{j=1}^{6} z_j^6} 
   e^{(x/2) \sum_{j=1}^{3} (z_j-z_{3+j})  } dz_1 \ldots dz_6.
\end{split}
\end{equation*}
It is known that $P_9(x)$ is a polynomial of degree 9 (see \cite[Theorem 1.2]{CFKRS2008}).
Now choose $\omega^{+}(t)$ to be a smooth majorant of the function $\mathds{1}_{[T,2T]}(t)$ with $\omega^{+}=1$ in $[T,2T]$, $\text{support}(\omega^{+}) \subseteq [T-T_0, 2T+T_0]$, and satisfying $(\omega^{+})^{(j)} \ll T_0^{-j}$. 
It follows that 
\begin{equation*}
  \label{I3T2T}
  I_3(2T)-I_3(T) \le  \int_{T}^{2T} P_9 \Big( \log \frac{t}{2 \pi} \Big) dt +  O \Big(T_0 (\log T)^9+
  T^{\frac{3 \vartheta}{2}+\varepsilon} \Big( \frac{T}{T_0} \Big)^{1+C} 
  \Big).
\end{equation*}
The term $O(T_0 (\log T)^9)$ arises from estimating the portions of the integral corresponding to the intervals 
$[T-T_0,T]$ and $[2T,2T+T_0]$. Now choose $T_0$ so that the first
and second error terms are equal.   Solving for $T_0$ we find that $T_0= T^{\frac{\frac{3\vartheta}{2}+1+C}{2+C}}$ and note that this satisfies \eqref{cond3} since 
$\frac{7}{8} \le  \frac{ \frac{3\vartheta}{2}+1+C}{2+C}  < 1$ for $\vartheta \in [\frac{1}{2},\frac{2}{3})$ and $C \ge 0$.  Thus
\begin{equation}
   \label{I3Tub}
   I_3(2T)-I_3(T) \le  \int_{T}^{2T}  P_9 \Big( \log \frac{t}{2 \pi} \Big) dt + O \Big(
   T^{\frac{\frac{3\vartheta}{2}+1+C}{2+C}+\varepsilon}
   \Big).
\end{equation}
By a similar argument with a smooth minorant $\omega^{-}(t)$ of $\mathds{1}_{[T,2T]}(t)$ one
may substitute ``$\ge$" for ``$\le$" in \eqref{I3Tub}, and so may conclude that
\begin{equation}
  \label{I3dyadic}
    I_3(2T)-I_3(T)  = \int_{T}^{2T} P_9 \Big( \log \frac{t}{2 \pi} \Big) dt + O \Big(
   T^{\frac{\frac{3\vartheta}{2}+1+C}{2+C}+\varepsilon}
   \Big)
\end{equation}
for $T \ge T_1$, where $T_1$ is a sufficiently large fixed constant.  We now assume that $T \ge T_1^5$ so that $T^{\frac{1}{5}} \ge T_1$.  Let $J \ge 1$, such that  $\frac{T}{2^J} \ge T^{\frac{1}{5}} > \frac{T}{2^{J+1}}$.
Substituting $\frac{T}{2^j}$ in \eqref{I3dyadic}  for $j=1, \ldots, J$ and summing, we have 
\[
  I_3(T) - I_3(\tfrac{T}{2^J}) = \int_{\frac{T}{2^J}}^{T} P_9 \Big( \log \frac{t}{2 \pi} \Big) dt 
  + O \Big(  \sum_{j=1}^{J}  \Big( \frac{T}{2^j}  \Big)^{\frac{\frac{3\vartheta}{2}+1+C}{2+C}+\varepsilon}
 \Big) 
 = \int_{\frac{T}{2^J}}^{T} P_9 \Big( \log \frac{t}{2 \pi} \Big) dt 
  + O \Big(  T^{\frac{\frac{3\vartheta}{2}+1+C}{2+C}+\varepsilon} \Big).
\]
Now observe that  \footnote{ The constants in \eqref{I3Tinitial}  are effective.
For instance, we can use Lehman's bound 
\cite[Lemma 2]{L}
$|\zeta(\tfrac{1}{2}+it)| \le (4/(2 \pi)^{\frac{1}{4}}) t^{\frac{1}{4}}$
for $t \ge 128 \pi$.} 
\begin{equation}
  \label{I3Tinitial}
  I_3(\tfrac{T}{2^J}) \le I_3(2T^{\frac{1}{5}}) \le \int_{0}^{2T^{\frac{1}{5}}} 
  |\zeta(\tfrac{1}{2}+it)|^6 \, dt
  \ll   \int_{0}^{2T^{\frac{1}{5}}}  ( t^{\frac{1}{4}} +1)^6  \, dt \ll 
   T^{\frac{1}{2}}.
\end{equation}
Letting $P_9(u) = \sum_{j=0}^{9} \alpha_j u^j$, 
we find 
\[
 \Bigg| \int_{0}^{\tfrac{T}{2^J}} P_9 \Big( \log \frac{t}{2 \pi} \Big) dt  \Bigg|
 \le  \Big(
  \max_{0 \le j \le 9} |\alpha_j| \Big)
 \int_{2 \pi e}^{2 T^{\frac{1}{5}}}  \sum_{j=0}^{9}  \log^j( \tfrac{t}{2 \pi}) 
 \, dt  + O(1)
 \ll T^{\frac{1}{5}} \log^9(T). 
\]
From the last three displayed equations we obtain 
\[
    I_3(T) = \int_{0}^{T} P_9 \Big( \log \frac{t}{2 \pi} \Big) dt 
    +O \Big(  T^{\frac{\frac{3\vartheta}{2}+1+C}{2+C}+\varepsilon} + T^{\frac{1}{2}} \Big)
    =T \mathcal{P}_9(\log T) + O \Big(
   T^{\frac{\frac{3\vartheta}{2}+1+C}{2+C}+\varepsilon}
   \Big)
\]
since $\frac{\frac{3\vartheta}{2}+1+C}{2+C} > \frac{1}{2}$. Here
 $\mathcal{P}_9$ is a polynomial of degree 9, since $P_9$ is a polynomial of degree 9.
\end{proof}

\subsection{Conventions and Notation} \label{convnot}
Given two functions $f(x)$ and $g(x)$, we shall interchangeably use the notation  $f(x)=O(g(x))$, $f(x) \ll g(x)$, and $g(x) \gg f(x)$  to mean there exists $M >0$ such that $|f(x)| \le M |g(x)|$ for all sufficiently large $x$. 
We write $f(x) \asymp g(x)$ to mean that the estimates $f(x) \ll g(x)$ and $g(x) \ll f(x)$ simultaneously hold.  
If we write $f(x)=O_{a_1, \ldots, a_k}(g(x))$, $f(x) \ll_{a_1, \ldots, a_k} g(x)$, or $f(x) \asymp_{a_1, \ldots, a_k} g(x)$ for real numbers $a_1, \ldots, a_k$, then we mean that the 
corresponding constants depend on $a_1, \ldots, a_k$.  
In this article we shall use the convention that $\varepsilon$ denotes an arbitrarily small positive constant which may vary from instance to instance.  
In addition, $B$ shall denote a positive constant, which may be taken arbitrarily large and which may change from line to line. 
The letter $p$ will always be used to denote a prime number.
For a function $\varphi : \mathbb{R}^{+} \times \mathbb{R}^{+} \to \mathbb{C}$, $\varphi^{(m,n)}(x,y) = \frac{\partial^m}{\partial x^m}
\frac{\partial^n}{\partial y^n} \varphi(x,y)$.  
The integral notation $\int_{(c)} f(s) ds$ for a complex function $f(s)$ and $c \in \mathbb{R}$
 will be used frequently and is defined by  the following contour integral
\begin{equation*}
   \label{intc}
  \int_{(c)}f(s) ds = \int_{c-i\infty}^{c + i \infty} f(s) \, ds. 
\end{equation*}
In this article we shall consider $s \in \mathbb{C}$ and  usually we shall write its real part as $\sigma =\Re(s)$. 
Throughout this article we often use the fact that $\omega(t)$ has support in
$[c_1 T, c_2T]$ so that $t \asymp T$.  \\
  
\section{The approximate functional equation and the Dirichlet series $\mathcal{Z}_{\I,\J}(s)$}

One of the difficulties in evaluating mean values of the type \eqref{IkT} and \eqref{IIJT} is that the
integration is on the line $\Re(s)=\frac{1}{2}$ where $\zeta$ does not possess an absolutely convergent
Dirichlet series.  Instead, in the critical strip a standard tool is the approximate functional equation. 
The approximate functional equation for $\zeta(s)^k$ for $k=1,2$ was derived by Hardy and Littlewood.  Their result is of the form
\begin{equation*}
  \label{afeclassical}
  \zeta(s)^k = \sum_{n \le x} \frac{d_k(n)}{n^s} + \chi(s)^k \sum_{n \le y} \frac{d_k(n)}{n^{1-s}} +  \text{error}(s,x,y)
\end{equation*}
where $x$ and $y$ are an arbitrary pair of real numbers satisfying $x, y \gg 1$ and 
$xy =( \frac{|t|}{2 \pi} )^k$ (it being assumed that $s=\sigma+it$ with $-\frac{1}{2} \le \sigma \le \frac{3}{2}$ 
and $t^2 \gg 1$).  
There are several problems with this version of the approximate functional equation. 
First, each of these sums have sharp cutoffs, that is, the sum over $n$ does not decay smoothly. 
 In practice, it is convenient to sum over all integers with 
a weight which is smooth.   The sharp cutoff functions lead to poor error terms error$(s,x,y)$.  
Another problem is the presence of the factor $\chi(s)^k$.   Modern versions of the approximate functional 
equation (see \cite[p.92, eq. (4.20.1)]{Ti}) have the shape
\begin{equation}
   \label{afemodern}
   \zeta(s)^k =\sum_{n=1}^{\infty} \frac{d_k(n)    \nu_{x}(n)}{n^s}
    +
    \chi(s)^k \sum_{n =1}^{\infty} \frac{d_k(n) \tilde{\nu}_{y}(n)}{n^{1-s}}  
    +O(\exp(-c t^2))
\end{equation}
where $c >0$ and 
$\nu_{x}(m)$ and $\tilde{\nu}_{y}(n)$ 
are certain smooth weights essentially supported in $[0,x]$ and $[0,y]$, while
 $xy = ( \frac{|t|}{2 \pi} )^k$ and $s=\sigma+it$ with $0 \le \sigma \le 1$ and $t^2 \gg 1$.  
A classical approach to evaluating \eqref{IkT} 
is to use the identity 
$|\zeta(\tfrac{1}{2}+it)|^{2k} = \zeta(\tfrac{1}{2}+it)^k \zeta(\tfrac{1}{2}-it)^k$ and then
to apply \eqref{afemodern} with $s=\tfrac{1}{2} \pm it$ and then multiply out to give nine terms. 
Some of the terms contain factors of the form $\chi(\tf \pm it)^k$.   These  have to be treated with stationary phase and
lead to unappealing oscillatory integrals. 
One way to circumvent this problem is to develop an approximate functional equation 
for $|\zeta(\tf+it)|^{2k}$ instead of $\zeta(\tf+it)^k$.  This idea is due to Heath-Brown 
\cite[Lemma 1]{HB} who showed that 
\begin{equation}
  \label{HBafe}
  |\zeta(\tf+it)|^{2k} = 2\sum_{mn \le c T^{k}}^{\infty} \frac{d_k(m)d_k(n)}{
  \sqrt{mn}}
  \Big( \frac{m}{n} \Big)^{-it}K(mn,t)  + O(T^{-2}) \text{ for }  t \in [T,2T]
\end{equation}
where both $c$ (a positive constant) and $K(u,t)$ (a certain smooth weight function) depend on $k$. 

In this section, we prove an approximate functional equation for $\zeta_{\I}(\tf+it) \zeta_{\J}(\tf-it)$
analogous to \eqref{HBafe} in the case $\I = \{a_1,a_2,a_3\}$ and $\J= \{b_1,b_2,b_3\}$.  
Recall that $\zeta_{\I}$ and $\zeta_{\J}$ are defined 
in Definition \ref{zetaX}.  The following proposition is a straightforward generalization of \cite[Proposition 2.1, p. 209]{HY} which handles the case $\I = \{a_1,a_2\}$ and $\J= \{b_1,b_2\}$.

\begin{prop} \label{afe}
Let $\delta \in (0,\frac{1}{2})$ and $|t| \ge 1$. 
Let $\I = \{a_1,a_2,a_3\}$ and $\J= \{b_1,b_2,b_3\}$ be  complex numbers
which satisfy  $|a_i|, |b_i| \le \delta$ for $i=1,2,3$. 
Let $G(s)$ be an even, entire function of rapid decay
\footnote{$G$ is of rapid decay if for every $B >0$, we have $|G(s)| \le |s|^{-B}$ for $|\Re(s)| \le A$
and $|\Im(s)|$ sufficiently large.
An admissible $G$ is 
$G(s) =  \exp(s^2)$.  Observe that $A$ may be chosen to be any positive constant.
}
 as $|s| \to \infty$ in any fixed strip $| \Re(s)| \le A$ with $G(0)=1$.
 Let 
\begin{equation}
  \label{VIJ}
   V_{\I,\J;t}(x) = \frac{1}{2 \pi i} \int_{(1)} \frac{G(s)}{s} g_{\I,\J}(s,t) x^{-s}ds,
\end{equation}
where 
\begin{equation}
  \label{gIJ}
  g_{\I,\J}(s,t) = 
   \prod_{j=1}^{3} 
  \frac{  \Gamma \Big(\frac{\tfrac{1}{2}+a_j+s+it}{2} \Big) \Gamma \Big(\frac{\tfrac{1}{2}+b_j+s-it}{2} \Big)}{ \Gamma  \Big(\frac{\tfrac{1}{2}+a_j+it}{2} \Big) \Gamma \Big(\frac{\tfrac{1}{2}+b_j-it}{2} \Big)}.
\end{equation}
Furthermore, set 
\begin{equation}
  \label{XIJ}
  X_{\I,\J;t} = \pi^{\sum_{j=1}^3 a_j+b_j}
  \prod_{j=1}^{3} 
  \frac{  \Gamma \Big(\frac{\tfrac{1}{2}-a_j-it}{2} \Big) \Gamma \Big(\frac{\tfrac{1}{2}-b_j+it}{2} \Big)}{ \Gamma \Big(\frac{\tfrac{1}{2}+a_j+it}{2} \Big) \Gamma \Big(\frac{\tfrac{1}{2}+b_j-it}{2} \Big)}.
\end{equation}
Then for any constant $A'>0$, we have 
\begin{multline}
  \label{smoothafe}
  \zeta_{\I}(\tf+it) \zeta_{\J}(\tf-it)
  = \sum_{m,n=1}^{\infty} \frac{\sigma_{\I}(m) \sigma_{\J}(n)}{(mn)^{\frac{1}{2}}}
\Big( \frac{m}{n} \Big)^{-it}
V_{\I,\J;t}( \pi^3 mn)  \\
+ X_{\I,\J;t} \sum_{m,n=1}^{\infty} \frac{\sigma_{-\J}(m) \sigma_{-\I}(n)}{(mn)^{\frac{1}{2}}} 
\Big( \frac{m}{n} \Big)^{-it}V_{-\J,-\I;t}(\pi^3 mn)  
 + O( (1+|t|)^{-A'} ).
\end{multline}
\end{prop}
\noindent {\bf Remark}. This proposition can be generalized to the case $\I=\{a_1, \ldots, a_k \}$ and $\J=\{b_1,\ldots, b_k\}$.
\begin{proof} 
We begin by assuming that $\I$ and $\J$  satisfy 
\begin{equation}
  \label{ajbjspecial}
  |a_j| = |b_j| = 2^{j-1} \delta^j \text{ for } j=1,2,3,
\end{equation}
where $\delta < \frac{1}{2}$.  At the end of the proof we remove this condition. 
This ensures that the elements of $\I$ are distinct from one another and also the elements of $\J$ are distinct from one another. 
Indeed, this implies 
\begin{equation}
   \label{ajbjdiff}
   |a_i-a_j|  \ge 4 \delta^2 (\tfrac{1}{2}-\delta) 
   \text{ and }   |b_i-b_j|  \ge 4 \delta^2 (\tfrac{1}{2}-\delta)  \text{ for } 
   1 \le i < j \le 3. 
\end{equation}
Throughout this proof we let $
  \Lambda(s) = \pi^{-\frac{s}{2}} \Gamma(\tfrac{s}{2}) \zeta(s)$ and we make use of the functional equation $\Lambda(s)= \Lambda(1-s)$.
Set
\begin{equation}
  \label{LambdaIJ}
\Lambda_{\I,\J}(s) 
  = \prod_{j=1}^{3} \Lambda(\tf+s+a_j+it) \Lambda(\tf+s+b_{j}-it),
\end{equation}
and
\begin{equation}
  I_1 = \frac{1}{2 \pi i} \int_{(1)} \Lambda_{\I,\J}(s) \frac{G(s)}{s} ds.
\end{equation}
We shall move the contour left to the line $\Re(s)=-1$ and apply the residue theorem. 
The integrand has poles at $s=0$ and at $s=\pm \frac{1}{2}-a_j-it$ and $s=\pm \frac{1}{2}-b_j+it$ with $j=1,2,3$.
Observe that by the condition \eqref{sizerestrictiondelta}
 the poles  all lie within the region $-1 < \Re(s) < 1$.  Furthermore, since $|a_j|, |b_j|  <  \frac{1}{2}$ for $j=1,2,3$ and $|t| \ge 1$, 
 it follows that the poles are distinct. 
The residue at $s=0$ is  
\[
  \Lambda_{\I,\J}(0) =\prod_{j=1}^{3} \Lambda(\tf+a_j+it) \Lambda(\tf+b_{j}-it).
\]
It may be checked that the residue at the pole $s = \frac{1}{2}-a_1-it$ is 
\begin{align*}
&  \Big( \prod_{j=2}^{3} \Lambda(1+a_j-a_1) \Big)
 \prod_{j=1}^{3} \Big( |\Gamma( \tfrac{1+b_j-a_1-2it}{2} )  \zeta(1+b_j-a_1-2it)|  \Big)
  \frac{|G( \frac{1}{2}-a_1-it )|}{| \frac{1}{2}-a_1-it| } \\
  & \ll  \Big( \prod_{j=2}^{3} \frac{1}{|a_j-a_1|} \Big) 
  \Big(  t^{\delta} e^{-\frac{\pi}{2}t} \cdot   t^{\frac{1}{2}} \Big)^3 
  (1+|t|)^{-A-6}
   \ll_{\delta}
    t^{3} e^{-\frac{3\pi t}{2}}(1+|t|)^{-A-6}, 
\end{align*}
where we have used $\delta < \frac{1}{2}$, 
$G$ is of rapid decrease of $G(s)$ when $|\Im(s)|$ is large, and \eqref{ajbjdiff}. 
We get the same bound at the other poles. 
Let
\begin{align*}
  I_2 & =  \frac{1}{2 \pi i} \int_{(-1)} \Lambda_{\I,\J}(s) \frac{G(s)}{s} ds. 
\end{align*} 
By the residue theorem it follows that 
\[
    I_{1} -I_{2} =  \Lambda_{\I,\J}(0) + O( e^{-\frac{3\pi t}{2}}(1+|t|)^{-A-3}).
\]
Now observe that  $\Lambda_{\I,\J}(-s)=\Lambda_{-\J,-\I}(s)$ and thus
\[
  I_2 = -\frac{1}{2 \pi i} \int_{(1)} \Lambda_{-\J,-\I}(s) \frac{G(s)}{s} ds. 
\]
Set 
$\bZ_{\I,\J,t}(s)
  = \prod_{j=1}^{3} \zeta(\tf+s+a_j+it) \zeta(\tf+s+b_{j}-it)$
and 
\begin{equation}
  \label{GIJt}
 \bG_{\I,\J,t}(s)
  = 
  \pi^{-\frac{3}{2}-3s-\frac{1}{2}\sum_{j=1}^{3}(a_j+b_j)}
  \prod_{j=1}^{3} \Gamma \Big(\frac{\tf+s+a_j+it}{2} \Big) \Gamma \Big(\frac{\tf+s+b_{j}-it}{2} \Big).
\end{equation}
A calculation using the definition \eqref{LambdaIJ} establishes that 
$\Lambda_{\I,\J}(s)  =   \bZ_{\I,\J,t}(s)  \bG_{\I,\J,t}(s)$.
It follows that 
\begin{equation*}
\begin{split}
  \label{firstidentity}
   \bZ_{\I,\J,t}(0) & = \frac{1}{2 \pi i} \int_{(1)} 
    \bZ_{\I,\J,t}(s)
    \frac{  \bG_{\I,\J,t}(s) }{  \bG_{\I,\J,t}(0)} \frac{G(s)}{s} ds
    + 
     \frac{1}{2 \pi i} \int_{(1)} 
    \bZ_{-\J,-\I,t}(s)
    \frac{  \bG_{-\J,-\I,t}(s) }{  \bG_{\I,\J,t}(0)} \frac{G(s)}{s} ds \\
    & + O((1+|t|)^{-A}).
\end{split}
\end{equation*}
Here we make use of the bound $ |\bG_{\I,\J,t}(0)| \gg |t|^{-3} e^{-\frac{3 \pi t}{2}}$,
which follows from Stirling's formula.
We observe that
\begin{equation*}
    \frac{  \bG_{\I,\J,t}(s) }{  \bG_{\I,\J,t}(0)} = \pi^{-3s} g_{\I,\J}(s,t)
 \text{ and }    \frac{  \bG_{-\J,-\I,t}(s) }{  \bG_{\I,\J,t}(0)}  =\pi^{-3s} X_{\I, \J;t} g_{-\J,-\I}(s,t)
\end{equation*}
which follows from  \eqref{GIJt} and definitions \eqref{gIJ} and \eqref{XIJ}.
The second equality makes use of the identity
$X_{\I, \J;t}= \frac{  \bG_{-\J,-\I,t}(0) }{  \bG_{\I,\J,t}(0)}$.
Combining the above facts  with $\bZ_{\I,\J,t}(0) =   \zeta_{\I}(\tf+it) \zeta_{\J}(\tf-it)$
we arrive at
\begin{equation}
\begin{split}
  \label{secondidentity}
    \zeta_{\I}(\tf+it) \zeta_{\J}(\tf-it) & = \frac{1}{2 \pi i} \int_{(1)} 
    \bZ_{\I,\J,t}(s) \pi^{-3s} g_{\I,\J}(s,t)
    \frac{G(s)}{s} ds \\
    &+ 
     \frac{1}{2 \pi i} \int_{(1)} 
    \bZ_{-\J,-\I,t}(s)  \pi^{-3s} X_{\I, \J;t} g_{-\J,-\I}(s,t)
   \frac{G(s)}{s} ds 
    + O((1+|t|)^{-A}).
\end{split}
\end{equation}
However, we have the Dirichlet series expansions 
\[
    \bZ_{\I,\J,t}(s) = \sum_{m,n=1}^{\infty} \frac{\sigma_{\I}(m) \sigma_{\J}(n)}{m^{\frac{1}{2}+s+it}
    n^{\frac{1}{2}+s-it}} \text{ and }
    \bZ_{-\J,-\I,t}(s) = \sum_{m,n=1}^{\infty} \frac{\sigma_{-\J}(m) \sigma_{-\I}(n)}{m^{\frac{1}{2}+s+it}
    n^{\frac{1}{2}+s-it}}.
\]
These expressions are inserted in \eqref{secondidentity}.  Since they 
are absolutely convergent on $\Re(s)=1$, we may exchange integration and summation order.
Thus using definition \eqref{VIJ} we arrive at \eqref{smoothafe} in the case that $\I$ and $\J$ 
satisfy \eqref{ajbjspecial}.  We now deduce \eqref{smoothafe} in the case that 
\begin{equation}
  \label{clregion}
|a_j|,  |b_j| \le  2^{j-1} \delta^j \text{ for } j=1,2,3.  
\end{equation}
Note that $\zeta_{\I}(\tf+it) \zeta_{\J}(\tf-it)$ is holomorphic for $|a_j| < \frac{1}{2}$ and $|b_j|
< \frac{1}{2}$ for $j =1,2,3$.
Likewise the sum of the two terms on the right hand side of the equality \eqref{smoothafe} is holomorphic in  the same region in $\C^6$. 
Since $\delta < \frac{1}{2}$, the closed region \eqref{clregion}
is a subset of $|a_j|,|b_j| < \frac{1}{2}$ for $j=1,2,3$. 
It follows from six applications of the maximum modulus principle that \eqref{smoothafe} also holds assuming  \eqref{clregion}. Now as $\delta$ approaches $\tfrac{1}{2}$ from below,   $2^{j-1} \delta^j  \to \frac{1}{2}$
for $j=1,2,3$.
It follows that for any $\delta' \in (0,\tfrac{1}{2})$, \eqref{smoothafe} holds assuming $|a_j|,|b_j| \le \delta'$ 
for $j=1,2,3$ and this completes the  proof. 
\end{proof}

The next lemma gives estimates for the functions
$  X_{\I,\J;t}$,
 $ g_{\I,\J}(s,t)$, and $V_{\I,\J;t}$. 
\begin{lem} \label{Stirling} 
Let $A$ be a positive constant
and  $\delta \in (0, \frac{1}{6})$.
 Suppose that $t \ge 2A+3$.  Assume that 
$\I = \{a_1,a_2,a_3\}$ and $\J= \{b_1,b_2,b_3\}$
satisfy $|a_i|, |b_i| \le \delta$. 
 Then one has: \\
\begin{itemize}
\item[(i)]
$\displaystyle X_{\I,\J;t} = \Big( \frac{t}{2 \pi} \Big)^{-\sum_{i=1}^{3} (a_i+b_i)} 
 (1+O(t^{-1})) \, ;$
\item[(ii)]
For $0 \le \Re(s) \le A$, we have 
$\displaystyle g_{\I,\J}(s,t) = \Big( \frac{t}{2} \Big)^{3s} \Big(1+O \Big(
\Big(
\frac{|s|^2+1}{t} \Big)  t^{3 \delta} \Big)  \Big) \, ;$ \\
\item[(iii)]  Let $\varepsilon >0$ and $i \ge 0$. For $\Re(s)=\varepsilon$, $|\Im(s)| \le \sqrt{T}$, and $c_1 T \le t \le c_2 T$ where $0 < c_1 < c_2$,
\begin{equation}
  \label{gstiderivatives}
   \frac{d^{i}}{dt^{i}} g_{\I,\J}(s,t) \ll_{i,\varepsilon}  |s|^{i} T^{3\varepsilon+3 \delta-i};
\end{equation}
\item[(iv)]   
For $x > t^3$,
$\displaystyle V_{\I,\J;t}(x) = O \Big( \Big( \frac{t^3}{x} \Big)^{A} \Big) \,$.
\end{itemize}
\end{lem}
\begin{proof} The first two parts follow from Stirling's formula and are technical calculations. 
Since the proof of (i) is similar and easier than (ii), we leave it as an exercise. 
The proof of (ii) and (iii) will be deferred to Appendix \ref{appendix2}.   
Proof of part (iv).  Note that we can move the contour right to $\Re(s)=A$  
so that 
\begin{equation*}
\begin{split}
  V_{\I,\J;t}(x) & = \frac{1}{2 \pi i} \int_{(A)} \frac{G(s)}{s} g_{\I,\J}(s,t) x^{-s} ds 
   \ll   \int_{A-i \infty}^{A +i \infty} \frac{|G(s)|}{|s|} \Big( \frac{t^3}{8x} \Big)^{\Re(s)}
  \Big( 1 + \frac{c|s|^{2}}{t^{1-3 \delta}} \Big)
   |ds|
\end{split}
\end{equation*}
for some positive constant $c$,  
by part (ii).  It follows that $
 V_{\I,\J;t}(x) \ll   ( \frac{t^3}{x} )^{A}$ as desired.
\end{proof}
In  our evaluation of $I_{\I,\J}(\omega)$ we shall encounter the Dirichlet series 
$\mathcal{Z}_{\I,\J}(s) = \sum_{n=1}^{\infty} \frac{\sigma_{\I}(n) \sigma_{\J}(n)}{n^{1+s}}$.  We now provide a factorization
of this series into zeta factors times an absolutely convergent product near $s=0$. 
\begin{lem} \label{ZIJ}
Let $\I=\{ a_1,a_2,a_3\}$ and $\J = \{b_1,b_2,b_3\}$.  Let $\delta \in (0,\frac{1}{6})$ and assume 
$|a_i|,|b_i| \le \delta$ for $i=1,2,3$. 
We have for $\Re(s) \ge 3\delta$ 
\begin{equation}
  \label{ZIJprod}
  \mathcal{Z}_{\mathscr{\I},\J}(s) = 
 \Big( \prod_{i,j=1}^{3} \zeta(1+s+a_i+b_j) \Big)
   \cA_{\I,\J}(s)
\end{equation}
where 
\begin{equation}
  \label{AIJ}
  \cA_{\I,\J}(s) = \prod_{p}
  \cA_{p; \I, \J}(s)
\end{equation}
and
\begin{equation}
  \label{ApIJ}
   \cA_{p; \I, \J}(s) =
   P(p^{-a_1},p^{-a_2},p^{-a_3},p^{-b_1},p^{-b_2},p^{-b_3};p^{-s-1})
\end{equation}
where 
\begin{equation}
\begin{split}
\label{P}
 & P(X_1,X_2,X_3,Y_1,Y_2,Y_3,U) \\
 &  =  
  1  - X_1 X_2 X_3 Y_1 Y_2 Y_3
  (X_{1}^{-1}+X_{2}^{-1}+X_{3}^{-1})(Y_{1}^{-1}+Y_{2}^{-1}+Y_{3}^{-1}) U^2   \\
  & + X_1 X_2 X_3 Y_1 Y_2 Y_3 \times \\
 &  \Big(  
  (X_{1}^{-1}+X_{2}^{-1}+X_{3}^{-1})(X_1+X_2+X_3)
  +  (Y_{1}^{-1}+Y_{2}^{-1}+Y_{3}^{-1})(Y_1+Y_2+Y_3)
  -2
  \Big) U^3 \\
   & - X_1 X_2 X_3 Y_1 Y_2 Y_3
   (X_1+X_2+X_3)(Y_1+Y_2+Y_3)U^4 \\
   & +(X_1 X_2 X_3 Y_1 Y_2 Y_3)^2 U^6.
\end{split}
\end{equation}
This implies that $\cA_{\I,\J}(s)$ is absolutely convergent in $\Re(s) \ge 3 \delta -\frac{1}{2}$
since $  \cA_{p; \I, \J}(s) = 1 +O(p^{-1-2\delta})$ in this region. 
\end{lem}
\noindent {\bf Remark}.  This Lemma provides a memorphic continuation of $ \mathcal{Z}_{\mathscr{\I},\J}(s)$
to $\Re(s) \ge 3 \delta -\frac{1}{2}$ and shows that it has poles at $s=-a_i-b_j$ for $(i,j) \in \{ 1,2,3 \}^2$.   Indeed it is possible to show that the elements $-a_i-b_j$ are not zeros 
of $\cA_{\I,\J}(s)$.
\begin{proof}
Let $s=2z-1$.  It is shown in \cite{CFKRS} that
\begin{align*}
 & \sum_{n=1}^{\infty}
 \frac{\sigma_{\I}(n) \sigma_{\J}(n)}{n^{2z}}
  = \Big( \prod_{i,j=1}^{3} \zeta(2z+a_i+b_j)  \Big) \mathcal{B}_{\I,\J}(z)
\end{align*}
where 
\begin{equation*}
  \mathcal{B}_{\I,\J}(z) = 
  \prod_{p} \sum_{m=1}^{3} 
  \prod_{i \ne m} \frac{\prod_{j=1}^{3} \Big( 1-\frac{1}{p^{2z+a_j+b_i}} \Big)}{1-p^{b_m-b_i}}. 
\end{equation*} 
It follows that $
  \cA_{\I,\J}(s) = 
  \prod_{p}
   \cA_{p;\I,\J}(s)$
where
\begin{equation}
  \label{ApIJ2}
  \cA_{p;\I,\J}(s) =
   \sum_{m=1}^{3} 
  \prod_{i \ne m} \frac{\prod_{j=1}^{3} \Big( 1-\frac{1}{p^{1+s+a_j+b_i}} \Big)}{1-p^{b_m-b_i}}.
\end{equation}
It is proven in \cite[p.66]{CFKRS} that \eqref{ApIJ2} equals 
$P(p^{-a_1},p^{-a_2},p^{-a_3},p^{-b_1},p^{-b_2},p^{-b_3};p^{-s-1})$
where $P$ is explicitly given in \cite[eq.(2.6.7), p.64]{CFKRS}.  Observe that, since
$a_i, b_i$ satisfy \eqref{sizerestrictiondelta} it follows that $X_i =p^{-a_i}$ and $Y_i=p^{-b_i}$ satisfy $|X_i|, |Y_i| \le p^{\delta}$ for $i=1,2,3$. 
Thus it follows from \eqref{P} that 
\begin{equation}
\begin{split}
   \cA_{p; \I, \J}(s) & =1+ O 
   \Big( p^{4\delta-2-2\sigma} +  p^{6\delta-3-3\sigma} +  p^{8\delta-4-4\sigma}
   +p^{12\delta-6-6\sigma}  \Big) \\
   & = 1 + O(p^{-2(\sigma+1-2 \delta})= 1 + O(p^{-1-2 \delta})
\end{split}
\end{equation}
assuming $\sigma \ge -\frac{1}{2}+3 \delta$ and $\delta < \frac{1}{6}$.  
\end{proof}
In the next lemma, we provide standard bounds for $\zeta(s)$ and  bounds for $\mathcal{Z}_{\mathscr{\I},\J}(s)$. 
\begin{lem} 
\begin{itemize}
\item[(i)] For $t > t_0 >0$, uniformly in $\sigma$
\begin{equation}
  \label{zetabounds}
  |\zeta(\sigma+it)| \ll \begin{cases}
  1  & \sigma \ge 2, \\
  \log t   & 1 \le \sigma \le 2, \\
  t^{\frac{1-\sigma}{2}} \log t & 0 \le \sigma \le 1, \\
  t^{\frac{1}{2}-\sigma} \log t & \sigma \le 0. 
  \end{cases}
\end{equation}
\item[(ii)] Let $|a_i|,|b_i| \le \delta$ for $i=1,2,3$ with $\delta \in (0,\frac{1}{6})$. For $\Re(s) \ge 2 \delta$, 
\begin{equation*}
  \label{ZIJbd1}
   \mathcal{Z}_{\mathscr{\I},\J}(2s) \ll_{\delta} 1. 
\end{equation*}
\item[(iii)] Let $|a_i|,|b_i| \le \delta$ for $i=1,2,3$ with $\delta \in (0,\tfrac{1}{12})$.
For $\Re(s) = -\frac{1}{4} + 2\delta$ and $\Im(s)=u$,  
\begin{equation}
  \label{ZIJbd2}
     \mathcal{Z}_{\mathscr{\I},\J}(2s) \ll_{\delta}  (1+|u|)^{\frac{3}{2}}.
\end{equation}
\end{itemize}
\end{lem}
\begin{proof} (i) is \cite[Theorem 1.9]{Iv}.   (ii) For $\Re(s) \ge 2\delta$, it follows that $\Re(1+2s+a_i+b_j) \ge 1+2\delta$
and thus by \eqref{ZIJprod} and the absolute convergence of  $\cA_{\I,\J}(s)$  we have 
\[
       \mathcal{Z}_{\mathscr{\I},\J}(2s) \ll \zeta(1+2 \delta)^9 \ll_{\delta} 1.  
\]
(iii)  Since  $\Re(s) = -\frac{1}{4} + 2\delta$, we have that $\frac{1}{2}+2 \delta \le \Re(1+2s+a_i+b_j) 
\le \frac{1}{2}+6 \delta < 1$
and thus it follows from \eqref{zetabounds}  that each factor in the product \eqref{ZIJprod} is 
$\ll_{\delta} (1+|u|)^{\frac{1}{4}}$ where $u = \Im(s)$.  Therefore \eqref{ZIJbd2} follows since
there are 9 zeta factors in \eqref{ZIJprod}
and Lemma \ref{ZIJ} implies $\cA_{\I,\J}(2s) \ll_{\delta} 1$.
\end{proof}

\section{A formula for the sixth moment} \label{formulasixthmoment}

We now commence with our evaluation of 
$ I_{\I,\J}(\omega)$.  At the outset we assume a number of conditions on the elements of $\I$ and $\J$:
\begin{equation}
    \label{initialconditionA}
     |a_{j_1} \pm b_{j_2}| \ge \frac{C_0}{\log T}, \
      \text{ for } j_1,  j_2 \in \{1,2,3\},
\end{equation}
\begin{equation}
  \label{initialconditionB}
  |a_{j_1} \pm a_{j_2}| \ge \frac{C_0}{\log T}, \
  |b_{j_1} \pm b_{j_2}| \ge \frac{C_0}{\log T}, \
  \text{ for } j_1 \ne  j_2 \in \{1,2,3\},
\end{equation}
\begin{equation}
  \label{initialconditionC}
  |a_{i_1} \pm a_{i_2} \pm b_{j_1} \pm b_{j_2}| \ge \frac{C_0}{\log T}
   \text{ for } i_1 \ne  i_2 \in \{1,2,3\} \text{ and }
   j_1 \ne  j_2 \in \{1,2,3\}
\end{equation}
where $C_0$ is an arbitrary positive constant which satisfies $1 \ll C_0 \ll 1$. 
These size conditions shall be assumed throughout sections \ref{formulasixthmoment}, \ref{diagonalterms},  \ref{proofofmainthmoffdiagonal}, and \ref{section6}.  
They are made so that we will be able  to bound 
terms such as $\zeta(1+a_1-a_2)$ and $\zeta(1-a_1+a_2 -b_1+b_2)$ which can be very large without these assumptions.  They will also ensure that all the poles we
encounter are distinct. 
However, in the proof of Theorem \ref{mainthm} at the beginning of 
section \ref{proofofmainthmoffdiagonal} we show how to remove 
these conditions by making use of the fact that the $ I_{\I,\J}(\omega)$ and its main term  are holomorphic functions of the $a_i$ and $b_i$. 
By Proposition \ref{afe} it follows that
\begin{align*}
   I_{\I,\J}(\omega) & =   \sum_{m,n=1}^{\infty} \frac{\sigma_{\I}(m) \sigma_{\J}(n)}{(mn)^{\frac{1}{2}}}
 \int_{-\infty}^{\infty} \Big( \frac{m}{n} \Big)^{-it}
V_{\I,\J;t}( \pi^3 mn) \omega(t) dt  \\
 &+ \sum_{m,n=1}^{\infty}  \frac{\sigma_{-\J}(m)\sigma_{-\I}(n)}{(mn)^{\frac{1}{2}}}
  \int_{-\infty}^{\infty} \Big( \frac{m}{n} \Big)^{-it}
 X_{\I,\J;t}  V_{-\J,-\I;t}(\pi^3 mn) \omega(t) dt  +O \Big(\int_{-\infty}^{\infty} |\omega(t)| t^{-1} dt \Big) \\
 & := I^{(1)}+I^{(2)}+O(1).
\end{align*}
Opening the integral formulae for $V_{\I,\J;t}$ and $V_{-\J,-\I;t}$ yields 
\begin{equation*}
\begin{split} \label{Iu1}
 I^{(1)} 
 = \sum_{m,n=1}^{\infty} \frac{\sigma_{\I}(m) \sigma_{\J}(n)}{(mn)^{\frac{1}{2}}}
 \frac{1}{2 \pi i} \int_{(1)} \frac{G(s)}{s} (\pi^3 mn)^{-s} 
 \int_{-\infty}^{\infty} \Big( \frac{m}{n} \Big)^{-it}
g_{\I,\J}(s,t) \omega(t) dt ds,  
\end{split}
\end{equation*}
and 
\begin{equation*}
\begin{split} \label{Iu2}
 I^{(2)} 
 = \sum_{m,n=1}^{\infty} \frac{\sigma_{-\J}(m) \sigma_{-\I}(n)}{(mn)^{\frac{1}{2}}}
 \frac{1}{2 \pi i} \int_{(1)} \frac{G(s)}{s} (\pi^3 mn)^{-s} 
 \int_{-\infty}^{\infty} \Big( \frac{m}{n} \Big)^{-it}
 X_{\I,\J;t}  \  g_{-\J,-\I}(s,t) \omega(t) dt ds.
\end{split}
\end{equation*}
We now define the diagonal terms $I_{D}^{(1)}$ and $I_{D}^{(2)}$ to be those terms above where 
$m=n$.  Likewise the off-diagonal terms $I_{O}^{(1)}$ and $I_{O}^{(2)}$ are those terms above where
$m \ne n$.  More precisely,
\begin{equation} \label{ID1}
   I_{D}^{(1)} 
 = \sum_{n=1}^{\infty} \frac{\sigma_{\I}(n) \sigma_{\J}(n)}{n}
 \frac{1}{2 \pi i} \int_{(1)} \frac{G(s)}{s} (\pi^3 n^2)^{-s} 
 \int_{-\infty}^{\infty} 
 g_{\I,\J}(s,t) \omega(t) dt ds,  
\end{equation}
\begin{equation*} \label{ID2}
   I_{D}^{(2)} 
 = \sum_{n=1}^{\infty} \frac{\sigma_{-\J}(n) \sigma_{-\I}(n)}{n}
 \frac{1}{2 \pi i} \int_{(1)} \frac{G(s)}{s} (\pi^3 n^2)^{-s} 
 \int_{-\infty}^{\infty}
 X_{\I,\J;t}  \  g_{-\J,-\I}(s,t) \omega(t) dt ds,
\end{equation*}
\begin{equation} \label{IO1}
    I_{O}^{(1)} 
 = \sum_{m \ne n} \frac{\sigma_{\I}(m) \sigma_{\J}(n)}{(mn)^{\frac{1}{2}}}
 \frac{1}{2 \pi i} \int_{(1)} \frac{G(s)}{s} (\pi^3 mn)^{-s} 
 \int_{-\infty}^{\infty} \Big( \frac{m}{n} \Big)^{-it}
g_{\I,\J}(s,t) \omega(t) dt ds,  
\end{equation}
\begin{equation} \label{IO2}
  I_{O}^{(2)} 
 = \sum_{m \ne n} \frac{\sigma_{-\J}(m) \sigma_{-\I}(n)}{(mn)^{\frac{1}{2}}}
 \frac{1}{2 \pi i} \int_{(1)} \frac{G(s)}{s} (\pi^3 mn)^{-s} 
 \int_{-\infty}^{\infty} \Big( \frac{m}{n} \Big)^{-it}
 X_{\I,\J;t}  \  g_{-\J,-\I}(s,t) \omega(t) dt ds.
\end{equation}
Summarizing, we have 
$I^{(j)} = I_{D}^{(j)} + I_{O}^{(j)}$ for  $j=1,2$
and thus 
\begin{equation}
   \label{IIJTexpans}
    I_{\I,\J}(\omega) = (I_{D}^{(1)} + I_{O}^{(1)})+(I_{D}^{(2)} + I_{O}^{(2)}) +O(1). 
\end{equation}
The asymptotic evaluation of $I_{\I,\J}(\omega)$ is reduced to evaluating $I_{D}^{(j)}$ and $I_{O}^{(j)}$.
The calculations of the $I_{D}^{(j)}$ are straightforward.  The majority of this 
article concerns the evaluation of the off-diagonal sums $I_{O}^{(j)}$. 

\section{Diagonal terms} \label{diagonalterms}
In this section we evaluate the diagonal terms $I_{D}^{(j)}$.    
\begin{prop} \label{diagonal}
Let $\delta \in (0,\frac{1}{12})$ and assume 
$a_i,b_i$  satisfy 
$|a_i|,|b_i| \le \delta$ for $i=1,2,3$ and \eqref{initialconditionA} and \eqref{initialconditionC}.
Then 
\begin{equation}
\begin{split}
  \label{ID1asymp}
  I_{D}^{(1)}  & =   \int_{-\infty}^{\infty}  \mathcal{Z}_{\I,\J}(0) \omega(t) dt \\
& + \sum_{i,j=1}^3 \mathrm{Res}_{s=  \frac{-a_i -b_j}{2}} \mathcal{Z}_{\I,\J}(2s) \frac{G \Big( \frac{-a_i -b_j}{2}  \Big)}{ (\frac{-a_i -b_j}{2})}
  \int_{-\infty}^{\infty} \omega(t) 
  \Big( 
  \frac{t}{2 \pi}
  \Big)^{-\frac{3}{2}(a_i+b_j)}dt +O(T^{\frac{1}{4}+6 \delta})
\end{split}
\end{equation}
and 
\begin{equation}
\begin{split}
   \label{ID2asymp}
  I_{D}^{(2)} & =  \int_{-\infty}^{\infty} 
   \Big( \frac{t}{2 \pi} \Big)^{-\sum_{k=1}^{3} (a_k+b_k)} 
   \mathcal{Z}_{-\J,-\I}(0) \omega(t) dt \\
 & + \sum_{i,j=1}^{3}  
 \mathrm{Res}_{s=  \frac{a_i +b_j}{2}} \mathcal{Z}_{-\J,-\I}(2s) \frac{G \Big( \frac{a_i +b_j}{2}  \Big)}{ (\frac{a_i +b_j}{2})}
  \int_{-\infty}^{\infty} \omega(t) 
  \Big( 
  \frac{t}{2 \pi}
  \Big)^{-\sum_{k=1}^{3} (a_k+b_k)+\frac{3}{2}(a_i+b_j)}dt +O(T^{\frac{1}{4}+6 \delta}).
\end{split}
\end{equation}
\end{prop}
\begin{proof}
By \eqref{ID1}, moving the sum inside the integral, 
\begin{align*}
 I_{D}^{(1)}
 & 
 = 
 \int_{-\infty}^{\infty} \omega(t)  
 \Bigg\{
 \frac{1}{2 \pi i} \int_{(1)} \frac{G(s)}{s} \pi^{-3s} 
g_{\I,\J}(s,t) 
\mathcal{Z}_{\I,\J}(2s)
ds \Bigg\} dt,  
\end{align*}
where $\mathcal{Z}_{\I,\J}$ is defined by Definition \ref{ZXYs}.
By Cauchy's theorem, we move the $s$ integral to the $\Re(s) = \varepsilon_1$ line where $\varepsilon_1 =2\delta$,
and then apply Lemma \ref{Stirling} (ii) to obtain
\begin{equation}
  \label{ID1formula}
   I_{D}^{(1)}
 = 
 \int_{-\infty}^{\infty} \omega(t)   \Bigg\{
 \frac{1}{2 \pi i} \int_{(\varepsilon_1)} \frac{G(s)}{s} 
 \Big( \frac{t}{2 \pi} \Big)^{3s} (1+O(|s|^2|t|^{3 \delta-1})) 
\mathcal{Z}_{\I,\J}(2s)
ds \Bigg\} dt.
\end{equation}
The contribution to the contour integral from the $O(|s|^2|t|^{3 \delta-1})$ term is 
\[
  \ll_{\varepsilon_1} |t|^{3 \delta+3 \varepsilon_1-1} \int_{-\infty}^{\infty}  \frac{|G(\varepsilon_1+iu)|}{|\varepsilon_1+iu|} |\varepsilon_1+iu|^2  du 
  \ll_{\varepsilon_1} |t|^{3 \delta+3 \varepsilon_1-1},
\]
since $\mathcal{Z}_{\I,\J}(2s) \ll_{\delta} 1$ for $\Re(s)=\varepsilon_1$ and $G$ is of rapid decay. 
Since $\int_{-\infty}^{\infty} \omega(t) |t|^{3\varepsilon_1-1} dt \ll T^{9 \delta}$, it follows that 
\[
    I_{D}^{(1)}
 = 
 \int_{-\infty}^{\infty} \omega(t)  \Bigg\{
 \frac{1}{2 \pi i} \int_{(\varepsilon_1)} \frac{G(s)}{s} 
 \Big( \frac{t}{2 \pi} \Big)^{3s}\mathcal{Z}_{\I,\J}(2s)
ds \Bigg\}  dt +O(T^{9 \delta}). 
\]
The inner integrand has a pole at $s=0$ and by Lemma \ref{ZIJ} it also has poles at 
$s = \frac{-a_i -b_j}{2}$ for  $(i,j) \in \{1,2,3\}^2$.  By the conditions \eqref{initialconditionA},  \eqref{initialconditionB},
and \eqref{initialconditionC} it follows that 
$a_i +b_j \ne 0$ for $(i,j) \in  \{1,2,3\}^2$ and $a_i +b_j \ne a_h+b_k$ for $(i,j), (h,k) \in  \{1,2,3\}^2$. 
Therefore all of these poles are distinct. 
We now move the line of integration to $\Re(s) = -\frac{1}{4} +\varepsilon_1$,   crossing these poles.  
The contribution from the residue at $s=0$ is 
\begin{equation}
  \label{residue1}
  \int_{-\infty}^{\infty}  \mathcal{Z}_{\I,\J}(0) \omega(t) dt
\end{equation}
and, for $1 \le i,j \le 3$, that of  the residue at $s=  \frac{-a_i -b_j}{2}$ is
\begin{equation}
  \label{residue2}
  \text{Res}_{s=  \frac{-a_i -b_j}{2}} \mathcal{Z}_{\I,\J}(2s) \frac{G \Big( \frac{-a_i -b_j}{2}  \Big)}{ \frac{-a_i -b_j}{2}}
  \int_{-\infty}^{\infty} \omega(t) 
  \Big( 
  \frac{t}{2 \pi}
  \Big)^{-\frac{3}{2}(a_i+b_j)}dt.
\end{equation}
The $s$-integral on the line $\Re(s) = -\frac{1}{4}+\varepsilon_1$ is 
\[
   \frac{1}{2 \pi i} \int_{(-\frac{1}{4}+\varepsilon_1)} \frac{G(s)}{s} 
 \Big( \frac{t}{2 \pi} \Big)^{3s}\mathcal{Z}_{\I,\J}(2s)
ds  \ll  
 t^{-\frac{3}{4}+3 \varepsilon_1} \int_{-\infty}^{\infty} \frac{|G(-\frac{1}{4}+\varepsilon_1+iu)|}{|-\frac{1}{4}+\varepsilon_1+iu|} (1+|u|)^{\frac{3}{2}} du  \ll t^{-\frac{3}{4}+3 \varepsilon_1}  
\]
by virtue of \eqref{ZIJbd2} and the rapid decay of $G$.
Therefore by \eqref{cond1}-\eqref{cond3}
\[
 \int_{-\infty}^{\infty} \omega(t)  \Bigg\{
 \frac{1}{2 \pi i} \int_{(-\frac{1}{4}+\varepsilon_1)} \frac{G(s)}{s} 
 \Big( \frac{t}{2 \pi} \Big)^{3s}\mathcal{Z}_{\I,\J}(2s)
ds  \Bigg\} dt  \ll  \int_{c_1 T}^{c_2 T}  t^{-\frac{3}{4}+3 \varepsilon_1}  |\omega(t)| dt \ll T^{\frac{1}{4}+6 \delta}.
\]
Since $\delta < \frac{1}{12}$, it follows that $T^{9\delta} \ll T^{\frac{1}{4}+6 \delta}$.
Thus, combining \eqref{residue1}, \eqref{residue2}, and the last displayed equation 
we obtain  \eqref{ID1asymp}.

The evaluation of $ I_{D}^{(2)}$ can be done in a completely analogous fashion. 
For instance with the help of Lemma \ref{Stirling} (i) and (ii), we can show that 
\[
   I_{D}^{(2)} = 
 \int_{-\infty}^{\infty} \omega(t)\Big( \frac{t}{2 \pi} \Big)^{-\sum_{k=1}^{3} (a_k+b_k)}
 \frac{1}{2 \pi i} \int_{(\varepsilon_1)} \frac{G(s)}{s} 
 \Big( \frac{t}{2 \pi} \Big)^{3s} 
\mathcal{Z}_{-\J,-\I}(2s)
ds dt + O(T^{9 \delta}).   
\]
This formula is obtained from \eqref{ID1formula} by formally replacing $\mathcal{I}$ by $-\mathcal{J}$ and
by replacing $\mathcal{J}$ by $-\mathcal{I}$ and by inserting the factor 
$ ( \frac{t}{2 \pi} )^{-\sum_{k=1}^{3} (a_k+b_k)}$. 
As before we shall move the contour in the $s$-integral left,  passing poles at $s=0$ and 
$s = \frac{a_i+b_j}{2}$ for $i,j=1, \ldots, 3$. 
Calculating the residues as before, we arrive at \eqref{ID2asymp}. 
\end{proof}

\section{Proof of main theorem and initial evaluation of the off-diagonal terms}
\label{proofofmainthmoffdiagonal}

In this section we  begin the evaluation of the off-diagonal terms $I_{O}^{(j)}$. 
In addition, we shall prove Theorem \ref{mainthm}. 
 This is the most involved part of the argument. 
Combining all of the results from sections \ref{proofofmainthmoffdiagonal}
and \ref{section6} we shall establish the following.  The proof of Theorem \ref{offdiagonal} may be found 
in section \ref{section6} just after Proposition \ref{sumtozero}. 
\begin{prop} \label{offdiagonal}
Let
$a_i,b_i$  satisfy 
$|a_i|, |b_i| \ll (\log T)^{-1}$  for $i=1,2,3$ and \eqref{initialconditionA}, 
 \eqref{initialconditionB}, and 
\eqref{initialconditionC} with $C_0 >0$.
Suppose that $\vartheta \in [\frac{1}{2},\frac{2}{3})$ and $C >0$ are such that the hypothesis
$\mathcal{AD}(\vartheta,C)$ holds. 
For every $\varepsilon >0$, there exists $T_{\varepsilon} >0$ such that for $T \ge T_{\varepsilon}$
\begin{equation}
\begin{split}
  \label{IOasymp}
&   I_{O}^{(1)}+ I_{O}^{(2)} = 
   \int_{-\infty}^{\infty} 
   \Big(
   \sum_{j=1}^{2}
        \sum_{\substack{\mathcal{S} \subset \I, \mathcal{T} \subset \J \\ |\mathcal{S}|=|\mathcal{T}| =j }  }
   \mathcal{Z}_{\I_{\mathcal{S}},\J_{\mathcal{T}}}(0) 
   \Big( \frac{t}{2 \pi} \Big)^{-\mathcal{S}-\mathcal{T}} \Big) \omega(t) dt \\
   & - \sum_{i,j=1}^3 \mathrm{Res}_{s=  \frac{-a_i -b_j}{2}} \mathcal{Z}_{\I,\J}(2s) \frac{G \Big( \frac{-a_i -b_j}{2}  \Big)}{ \frac{-a_i -b_j}{2}}
  \int_{-\infty}^{\infty} \omega(t) 
  \Big( 
  \frac{t}{2 \pi}
  \Big)^{-\frac{3}{2}(a_i+b_j)}dt \\
   & - \sum_{i,j=1}^{3}  
 \mathrm{Res}_{s=  \frac{a_i +b_j}{2}} \mathcal{Z}_{-\J,-\I}(2s) \frac{G \Big( \frac{a_i +b_j}{2}  \Big)}{ \frac{a_i +b_j}{2}}
  \int_{-\infty}^{\infty} \omega(t) 
  \Big( 
  \frac{t}{2 \pi}
  \Big)^{-\sum_{k=1}^{3} (a_k+b_k)+\frac{3}{2}(a_i+b_j)}dt \\
  & +O \Big( T^{\frac{3 \vartheta}{2}+\varepsilon} \Big( \frac{T}{T_0} \Big)^{1+C}   
      \Big).
\end{split}
\end{equation}
\end{prop}
The main theorem, Theorem \ref{mainthm}, follows from Proposition \ref{diagonal} and
Proposition \ref{offdiagonal}.
\begin{proof}[Proof of Theorem \ref{mainthm}]
We begin by establishing Theorem \ref{mainthm} in the case that $\I$ and $\J$ satisfy
$|a_i|, |b_i| \le C' (\log T)^{-1}$ 
for $i=1,2,3$, with a fixed choice of $C' >0$, along with \eqref{initialconditionA},   \eqref{initialconditionB}, and  \eqref{initialconditionC}, 
where it is assumed that $0 <C_0 < C'$. 
Later, we shall remove the last three conditions. 
Combining \eqref{IIJTexpans}, \eqref{ID1asymp}, \eqref{ID2asymp}, and \eqref{IOasymp}
\begin{equation}
\begin{split}
   \label{IIJomegaid}
   I_{\I,\J}(\omega) & =  \int_{-\infty}^{\infty} 
    \Big(
     \sum_{j=0}^{3}
        \sum_{\substack{\mathcal{S} \subset \I, \mathcal{T} \subset \J \\ |\mathcal{S}|=|\mathcal{T}| =j }  }
   \mathcal{Z}_{\I_{\mathcal{S}},\J_{\mathcal{T}}}(0) 
   \Big( \frac{t}{2 \pi} \Big)^{-\mathcal{S}-\mathcal{T}} \Big) \omega(t) dt  
    +O \Big( T^{\frac{3 \vartheta}{2}+\varepsilon} \Big( \frac{T}{T_0} \Big)^{1+C}    
        \Big). 
\end{split}
\end{equation}
Notice that the sum of residues in \eqref{ID1asymp} and \eqref{ID2asymp} exactly cancel the two sums
of residues in \eqref{IOasymp}. Also, the first terms in \eqref{ID1asymp} and \eqref{ID2asymp} are added into 
the first sum of \eqref{IOasymp} making the sum over $j \in \{0,1,2,3\}$. 
This completes the proof of the main theorem under the 
assumptions \eqref{initialconditionA}, \eqref{initialconditionB}, and  \eqref{initialconditionC}. 

We now demonstrate how to remove the conditions 
\eqref{initialconditionA}, \eqref{initialconditionB}, and  \eqref{initialconditionC}
and prove the theorem without these restrictions. 
This will be similar to an argument employed in \cite{Br}. 
The main idea is to apply the Cauchy integral formula.  
Note that $I_{\I,\J}(\omega)$ is holomorphic if the $a_i$'s and $b_i$'s satisfy $|a_i| < \frac{1}{2}$ and 
$|b_i| < \frac{1}{2}$.  In addition, by Lemma 2.5.1 of 
\cite{CFKRS} the first term after the equality in \eqref{IIJomegaid} is holomorphic if  $|a_i|<\eta$ and $|b_i|< \eta$ 
for a sufficiently small $\eta$.  It follows that the error term in \eqref{IIJomegaid} is holomorphic in the $a_i, b_i$, as long
they are restricted to small enough disks. 
Let $L(a_1,a_2,a_3,b_1,b_2,b_3)=I_{\I,\J}(\omega)$ and 
let $R(a_1,a_2,a_3,b_1,b_2,b_3)$ denote the main term on the  right hand side of \eqref{IIJomegaid}.
We observe that both of these functions are analytic functions of the $a_i$ and $b_i$. 
We have just shown in \eqref{IIJomegaid} that 
\begin{equation}
  \label{initialcase}
L(a_1,a_2,a_3,b_1,b_2,b_3) - R(a_1,a_2,a_3,b_1,b_2,b_3)+O \Big( T^{\frac{3 \vartheta}{2}+\varepsilon} \Big( \frac{T}{T_0} \Big)^{1+C}   
      \Big)
\end{equation}
subject to   \eqref{initialconditionA}, \eqref{initialconditionB}, and \eqref{initialconditionC}.
Suppose that $a_1,a_2,a_3,b_1,b_2,b_3$ are complex numbers satisfying
\begin{equation}
  \label{generalcase}
|a_j|, |b_j| \le \frac{C_0}{\log T}, \text{ for } j=1,2,3.
\end{equation}  
Now consider the polydisc $D \subset \mathbb{C}^6$ defined by $D= \prod_{j=1}^{3} D_j
\prod_{j=1}^{3} \tilde{D}_j$ 
where 
\begin{equation*}
  \label{Dj}
  D_{j} := \{  z \in \mathbb{C} \ | \ |z-a_j| \le r_j \}  \text{ and } 
  \tilde{D}_{j}  := \{  z \in \mathbb{C} \ | \ |z-b_j| \le s_j \} \text{ for } j=1,2,3
\end{equation*}
and
\begin{equation*}
 \label{circlesizes2}
 r_j= \frac{2^{j+1}C_0}{\log T} \text{ and }
 s_j= \frac{2^{j+4}C_0}{\log T} \text{ for } j=1,2,3.
\end{equation*}
We shall denote $z_j$ and $w_j$ as the complex variables lying on the boundaries of $D_j$ and $\tilde{D}_j$ respectively
for $j=1,2,3$. 
By Cauchy's integral formula
\begin{equation}
\begin{split}
  \label{cif}
  & L(a_1,a_2,a_3,b_1,b_2,b_3) - R(a_1,a_2,a_3,b_1,b_2,b_3) \\
  & = \frac{1}{(2 \pi i)^6} 
  \int_{\partial D_1} \cdots \int_{ \partial \tilde{D}_3} 
  \frac{L(z_1,z_2,z_3,w_1,w_2,w_3) - R(z_1,z_2,z_3,w_1,w_2,w_3)}{\prod_{j=1}^{3}(z_j-a_j)(w_j-b_j) } 
  d\vec{z} d\vec{w}.
\end{split} 
\end{equation}
where $d \vec{z} =dz_1 dz_2 dz_3$, $d \vec{w}=dw_1 dw_2 dw_3$, and $\partial D_j, \partial \tilde{D}_j$
denote the boundaries of the polydiscs. 
To simplify notation we set $\eta=\frac{C_0}{\log T}$.  
Observe that, if $1 \le j_1, j_2 \le 3$, 
\begin{equation}
    |z_{j_1}\pm w_{j_2}| \ge |w_{j_2}-b_{j_2}| - |z_{j_1} - a_{j_1}|  - |a_{j_1}| - |b_{j_2}| 
    \ge  32 \eta -16 \eta -\eta -\eta =14 \eta.
\end{equation}
If $1 \le j_2 < j_1 \le 3$
\begin{equation}
  |z_{j_1}\pm z_{j_2}|  \ge |z_{j_1}-a_{j_1}| -|z_{j_2}-a_{j_2}| -|a_{j_1}| -|a_{j_2}| 
  \ge 2^{j_1+1} \eta - 2^{j_2+1} \eta -2 \eta 	
  \ge 2\eta 
\end{equation}
and likewise
$
  |w_{j_1} \pm w_{j_2}| \ge
  2^{j_1+4} \eta - 2^{j_2+4} \eta -2 \eta 
  \ge 30\eta$. 
  Finally, if $1 \le i_2 < i_1 \le 3$ and $1 \le j_2 < j_1 \le 3$
\begin{equation}
\begin{split}
&   |z_{i_1} \pm z_{i_2} \pm w_{j_1} \pm w_{j_2}|    \\
& \ge  | w_{j_1}-b_{j_1}| -|w_{j_2} -b_{j_2}| - |z_{i_1}-a_{i_1}| - |z_{i_1}-a_{i_1}|
  - |a_{i_1}| -|a_{i_2}| - |b_{j_1}| - |b_{j_2}| \\
  & \ge ( 2^{j_1+4} -2^{j_2+4} -2^{i_1+1}-2^{i_2+1} -4)\eta \ge 4 \eta.  
\end{split}
\end{equation}
Thus we have shown that the numbers $z_j$ and $w_j$ satisfy  \eqref{initialconditionA},
\eqref{initialconditionB}, and \eqref{initialconditionC} and thus \eqref{initialcase} holds
with $a_j=z_j \in \partial D_j$ and $b_j=w_j \in \partial \tilde{D}_j$ for $j=1,2,3$.  
Using this bound in the integrand in \eqref{cif}  we obtain 
\begin{equation}
\begin{split}
  &     L(a_1,a_2,a_3,b_1,b_2,b_3) - R(a_1,a_2,a_3,b_1,b_2,b_3)  \\
 &  \ll  \Big( \prod_{j=1}^{3} \frac{\text{length}(\partial D_j ) 
   \text{length}(\partial \tilde{D}_j ) 
 }{r_j s_j}   \Big) T^{\frac{3 \vartheta}{2}+\varepsilon} \Big( \frac{T}{T_0} \Big)^{1+C}  
 \ll    T^{\frac{3 \vartheta}{2}+\varepsilon} \Big( \frac{T}{T_0} \Big)^{1+C}  
\end{split}
\end{equation}
and this completes the proof of Theorem \ref{mainthm} in the general case. 
\end{proof}
We devote the remainder of this section, and all of the next section, to presenting our proof of Proposition
\ref{offdiagonal} (for the conclusion of its proof, see below the statement of Proposition \ref{sumtozero}).
Note that in this proof we have the assumption $
|a_i|, |b_i| \ll (\log T)^{-1}$ for $i=1,2,3$.
However, since $T$ may be assumed sufficiently large, we can assume the weaker condition 
\begin{equation*}
 |a_i|, |b_i| \le \delta \text{ for } i=1,2,3, 
 \text{ where }  \delta \in (0,\tfrac{1}{12}).
 \end{equation*} 
It shall be convenient at times to make use of this weaker condition.  
This bound shall be assumed throughout sections \ref{proofofmainthmoffdiagonal} and \ref{section6}.
This will simplify  some of the estimates and also it will simplify parts of the argument.
Moreover, the the only time we actually use \eqref{sizerestriction} is in the invocation of Conjecture \ref{divconj}.  Indeed as mentioned in Remark (ii) following Corollary \ref{maincor}, we could prove
 a version of Theorem \ref{mainthm}
assuming  just \eqref{sizerestrictiondelta}, instead of \eqref{sizerestriction}.
However, we would have to  assume a version of Conjecture \ref{divconj}  where the condition
\eqref{sizerestriction} 
is replaced by 
\eqref{sizerestrictiondelta}  with $\delta < \delta_0$
for a sufficiently small $\delta_0 >0$. 
We now begin an initial evaluation of $I_{O}^{(1)}$.  The evaluation of  $I_{O}^{(2)}$ 
is similar and  may be derived from our formula for $I_{O}^{(1)}$. 
\subsection{Initial steps: smooth partition of unity, restricting $M$ and $N$}
Let $\e >0$ and let 
\begin{equation}
  \label{fstardefn}
    f^{*}(x,y) = \frac{1}{2 \pi i} 
    \int_{(\varepsilon)} \frac{G(s)}{s} 
    \Big(  \frac{1}{\pi^3 xy} \Big)^s 
    \frac{1}{T} 
    \int_{-\infty}^{\infty} 
    \Big( \frac{x}{y} \Big)^{-it}
    g_{\I,\J}(s,t) \omega(t) dt \, ds.
\end{equation}
Observe that by \eqref{IO1} and an application of Cauchy's theorem we have 
\[
  I_{O}^{(1)} = T  \sum_{m \ne n} 
  \frac{\sigma_{\I}(m) \sigma_{\J}(n)}{\sqrt{mn}}
  f^{*}(m,n). 
\]
We introduce a smooth partition of unity to simplify the evaluation of this sum.
The main reason  this is done is so that we can apply a version of the additive divisor conjecture where the variables 
are restricted to boxes of the shape $[M,2M] \times [N,2N]$ with $M \asymp N$.   
Let $W_0$ be a smooth function supported in $[1,2]$ such that for $x > 0$,  $\sum_{k \in \mathbb{Z}} W_0 ( x/2^{\frac{k}{2}} ) =1$.
In particular, for $x \ge 1$, we have 
\begin{equation}
  \label{W0dyadic2}
   \sum_{ \substack{ M= 2^{\frac{k}{2}} \\ k \ge -1} } W_0 \Big( \frac{x}{M} \Big) =1. 
\end{equation}
See \cite[p. 360]{H} for an example of such a function $W_0$.   Thus 
\begin{equation}
     \label{IO1dyadic}
      I_{O}^{(1)} =  \sum_{M,N} I_{M,N}
\end{equation}
where summation is over $M,N \in \{ 2^{\frac{k}{2}} \ | \ k \in  \Z \text{ and } k \ge -1 \}$,  while 
\begin{equation*}
  \label{IMN}
   I_{M,N} = \frac{T}{\sqrt{MN}} 
\sum_{m \ne n} \sigma_{\I}(m) \sigma_{\J}(n) W \Big(\frac{m}{M} \Big) W \Big(\frac{n}{N} \Big) f^{*}(m,n)
\end{equation*}
and  $W(x) =x^{-\frac{1}{2}} W_0(x)$. We shall use later on that the constraints on $M$ and $N$ in the summation
in \eqref{IO1dyadic} imply that one has  $|\{ M \ | \ M \le X \}|
= |\{ N \ | \ N \le X \}| \ll 1+\log X$ for $X \ge 1$.
  
We now restrict the size of $M$ and $N$ to $MN \ll T^{3+\e}$, by showing the contribution 
from $MN \gg T^{3+\e}$ is negligible.    Observe that
\begin{equation}
  \label{fstarV}
   f^{*}(x,y) = \frac{1}{T} \int_{-\infty}^{\infty}  \Big( \frac{x}{y} \Big)^{-it} \omega(t) V_{\I, \J;t}(\pi^3 xy) dt.
\end{equation}
Note that by Lemma \ref{Stirling} (iv)  $V_{\I,\J;t}(\pi^3 mn)$ is very small when $mn \gg T^{3 +\varepsilon}$
and $c_1 T \le t \le c_2 T$. 
From this we can deduce that, for any $B >0$ 
\begin{equation}
  \label{MNbig}
  \sum_{\substack{M,N \\ MN \gg T^{3+\varepsilon}}} I_{M,N}  \ll T^{-B}.  
\end{equation}
Thus, in our analysis of $I_{M,N}$, we may assume $MN \ll T^{3 +\varepsilon}$. 

The next step is to observe that $m$ and $n$ must be close to each other.
This is since  oscillations of the factor
$(\tfrac{x}{y})^{-it}$ cause
$f^{*}(x,y)$ to be small unless $x$ and $y$ are close to each other.
We now provide details of this fact. 
We begin by bounding the inner integral in \eqref{fstardefn} when $\Re(s) =\varepsilon$ and $|s| \le \sqrt{T}$. 
Integrating by parts $j$ times, it follows that  
\[
\int_{-\infty}^{\infty} 
\Big( \frac{x}{y} \Big)^{-it} g_{\I,\J}(s,t) \omega(t) dt 
\ll \frac{1}{|\log(\tfrac{x}{y})|^j}
\int_{-\infty}^{\infty} 
\Big|
\frac{\partial^j}{\partial t^j} g_{\I,\J}(s,t) \omega(t) 
\Big| dt.
\]
By the generalized product rule, Lemma \ref{Stirling} (iii), and \eqref{cond3}
\begin{align*}
 \frac{\partial^j}{\partial t^j} g_{\I,\J}(s,t) \omega(t)  & = \sum_{a+b=j} \binom{j}{a} 
  \frac{\partial^a}{\partial t^a} g_{\I,\J}(s,t) \omega^{(b)}(t) 
  \ll \sum_{a+b=j} \binom{j}{a} |s|^a T^{3 \varepsilon+3 \delta-a} T_{0}^{-b} 
   = T^{3 \varepsilon+3 \delta} \Big( \frac{|s|}{T} + \frac{1}{T_0} \Big)^j  \\
   & \ll  T^{3 \varepsilon+3 \delta}  \frac{\max(|s|,1)^j }{T_0^j}
\end{align*} 
and thus 
\[
\int_{-\infty}^{\infty} 
\Big( \frac{x}{y} \Big)^{-it} g_{\I,\J}(s,t) \omega(t) dt 
\ll   T^{1+3 \varepsilon+3 \delta}  \frac{\max(|s|,1)^j }{|\log(\tfrac{x}{y})|^jT_0^j},
\]
if $|s| \le \sqrt{T}$. 
We now bound this integral when $\Re(s) =\varepsilon$ and $|s| > \sqrt{T}$.   
Note that Lemma \ref{Stirling} (ii) implies that, when $\Re(s) =\varepsilon >0$, one will have  
\begin{align*}
 \int_{-\infty}^{\infty}  \Big(\frac{x}{y} \Big)^{-it}   g_{\I,\J}(s,t) \omega(t) dt
\ll \int_{c_1T}^{c_2T}  |g_{\I,\J}(s,t) | dt 
\ll  \int_{c_1T}^{c_2T}  \Big( \frac{t}{2} \Big)^{3 \varepsilon} \Big( 1 + O \Big(\frac{|s|^2}{t^{1-3\delta}}\Big)
\Big) dt  \ll (T+|s|^2) T^{3\varepsilon+3 \delta}. 
\end{align*}
However, since $|s| > \sqrt{T}$
\begin{equation*}
\begin{split}
   & (T+|s|^2) T^{3\varepsilon+3 \delta}
   \ll T^{3 \varepsilon+3 \delta}  |s|^2 = T^{3 \varepsilon+3 \delta} \max(|s|,1)^2 \\
  &  \le  T^{3 \varepsilon+3 \delta} \max(|s|,1)^2 
   \Big( \frac{\max(|s|,1)}{\sqrt{T}} \Big)^{2j}
   = T^{1+3 \varepsilon+3 \delta}   \Big( \frac{\max(|s|,1)^{2j+2}}{T^{j+1}}  \Big).
\end{split}
\end{equation*}
Since $1 \ll x,y$ and $xy \ll T^{3+\varepsilon}$, we have $|\log(x/y)| \le |\log x| + |\log y|
\ll \log T \ll_{j} T^{1/j}$ and thus $|\log(x/y)|^j T_0^j \ll_j T T_{0}^j \le T^{1+j}$. 
Hence we obtain for all $s$ satisfying $\Re(s) = \varepsilon$
\[
\int_{-\infty}^{\infty} 
\Big( \frac{x}{y} \Big)^{-it} g_{\I,\J}(s,t) \omega(t) dt 
\ll   T^{1+3 \varepsilon+3 \delta}  \frac{\max(|s|,1)^{2j+2} }{|\log(\tfrac{x}{y})|^jT_0^j}.
\]
Therefore, by \eqref{fstardefn}, 
\begin{equation}
\begin{split}
  \label{f*bound}
  f^{*}(x,y) \ll 
  \frac{T^{3 \varepsilon+3 \delta}}{(xy)^{\varepsilon} |\log(\tfrac{x}{y})|^j T_{0}^j}
    \int_{(\varepsilon)} \frac{|G(s)|}{|s|} \max(|s|,1)^{2j}
    |ds|  
 \ll_{j}   \frac{T^{3 \varepsilon+3 \delta}}{|\log(\tfrac{x}{y})|^j T_{0}^j},
\end{split}
\end{equation} 
subject to $x,y \gg 1$ and $xy \ll T^{3 + \e}$. 
If 
\begin{equation*}
   \label{logcondition}
|\log(\tfrac{x}{y})| \gg T_{0}^{-1+\varepsilon},
\end{equation*}
 then for any $B>0$ we obtain from \eqref{f*bound} and \eqref{cond3} 
\begin{equation*}
  \label{f*bound2}
  f^{*}(x,y) \ll \frac{T^{3 \varepsilon+3 \delta}}{T_{0}^{j \varepsilon}} 
  \le T^{3(\e+\delta-\frac{j}{4} \e)} \ll T^{-B},
\end{equation*}
by choosing $j \ge \frac{4B}{3 \e} + 4 + \frac{4 \delta}{\e}$. 
Letting $m-n=r$, it follows that 
\[
  I_{M,N} 
= \frac{T}{\sqrt{MN}}
\sum_{r \ne 0} \sum_{\substack{m-n=r \\ |\log(\frac{m}{n})| \ll T_{0}^{-1+\varepsilon}}} \sigma_\I(m) \sigma_{\J}(n) 
W \Big(\frac{m}{M} \Big) W \Big(\frac{n}{N} \Big)f^{*}(m,n)
+ O(T^{-B})
\]
for any constant $B>0$.

We now impose several other conditions on $M$ and $N$.  
Note that the condition $|\log(m/n)| \ll T_{0}^{\varepsilon-1}$ attached to the above sum implies
that it is an empty sum unless one has $N/3 \le M \le 3N$: for if $M < N/3$ or $M >3N$, then,
when $m,n$ are such that $W(m/M) W(n/N) \ne 0$, one will have $|\log(m/n)| \ge \log(3/2)$. 
Also, observe that the conditions of summation imply $0 \ne r/N \ll T_{0}^{\varepsilon-1}$, so that one 
must have either an empty sum, or else $3N \ge M \ge N/3 \gg |r| T_{0}^{1-\varepsilon} \gg T_{0}^{1-\varepsilon}$.
We shall henceforth, in this section, write $M \asymp N$ to mean that 
\begin{equation}
  \label{MNcondition}
 N/3 \le M \le 3N 
\end{equation}
 (we only use this restricted meaning
for $\asymp$
with the symbols $M$ and $N$).  The contribution to the sum 
$\sum_{MN \ll T^{3 +\varepsilon}} I_{M,N}$ from cases where one does not have $M \asymp N \gg T_{0}^{1-\varepsilon}$
is therefore of size $O((\log T)^2 T^{-B})$.  Thus we may restrict ourselves to evaluating $I_{M,N}$ in the cases where 
$M \asymp N$ and $M,N \gg T_{0}^{1-\varepsilon}$. 
Observe that if $x-y=r$ is fixed, then by \eqref{fstardefn}, one has:
\begin{equation}
  \label{f*2}
   f^{*}(x,y) = \frac{1}{2 \pi i} \int_{(\varepsilon)} 
\frac{G(s)}{s} \Big( \frac{1}{\pi^3 xy} \Big)^s 
\frac{1}{T} \int_{-\infty}^{\infty} \Big( 
1+\frac{r}{y} \Big)^{-it} g_{\I,\J}(s,t) \omega(t) dt ds. 
\end{equation}
We have thus shown the following.
\begin{prop} 
\label{IMNestimate}
Let $B >0$ be arbitrary and fixed.  Then for $M,N \gg T_{0}^{1-\varepsilon}$ satisfying
$M \asymp N$ and $MN \ll T^{3+\varepsilon}$, we have 
\[
  I_{M,N} = \frac{T}{\sqrt{MN}} 
\sum_{0 < |r| \ll \frac{M}{T_0} T^{\varepsilon}}  D_{f_r;\I,\J}(r) 
 + O(T^{-B})
\]
where, for $0 \ne r \in \mathbb{Z}$, we define
\begin{equation}
  \label{fMN}
  f_r(x,y) = f_{r;M,N}(x,y)=
   W\Big( \frac{x}{M} \Big)
W \Big( \frac{y}{N} \Big)
\frac{1}{2 \pi i} \int_{(\varepsilon)} \frac{G(s)}{s} 
\Big( \frac{1}{\pi^3 xy} \Big)^s 
\frac{1}{T} \int_{-\infty}^{\infty} \Big( 
1+\frac{r}{y} \Big)^{-it} g_{\I,\J}(s,t) \omega(t) dt \, ds
\end{equation}
when $x,y >0$ (and put $f_r(x,y)=0$ otherwise), and take $D_{f_r;\I,\J}(r)$
to be the ternary additive divisor sum that is given by \eqref{DfIJ}, for $f=f_r$. 
For those pairs $M,N \gg 1$ such that one has $MN \ll T^{3 +\varepsilon}$ and either 
$M \not \asymp N$ or $\min(M,N) \ll T_{0}^{1-\varepsilon}$, the bound $I_{M,N} \ll T^{-B}$ applies.  
\end{prop}
\subsection{Application of the additive divisor conjecture $\mathcal{AD}(\vartheta,C)$}

In summary, by \eqref{IO1dyadic}, \eqref{MNbig}, and Proposition \ref{IMNestimate} we have established 
\begin{equation}
  \label{IO1dyadic2}
  I_{O}^{(1)} = \sum_{\substack{ M, N \\ M \asymp N, MN \ll T^{3+\varepsilon}
  \\ M, N \gg T_{0}^{1-\varepsilon}
  }} I_{M,N} +O(1). 
\end{equation}

We assume the validity of Conjecture \ref{divconj}, concerning the ternary additive divisor problem, which puts us now in 
a position to estimate $I_{M,N}$ for the relevant $M$ and $N$ in \eqref{IO1dyadic2}. Indeed, it suffices to note
that we have the following result on $f_r(x,y)$ in \eqref{fMN}, implying that the requisite conditions 
\eqref{fsupport} and \eqref{fcond} are satisfied when one has there: $X=M, Y=N, P=T^{1+\varepsilon}T_{0}^{-1}$,
and $f(x,y)=T^{- \frac{3 \varepsilon}{2}} f_r(x,y)$.
\begin{lem} \label{fpartials}
Let $\e$ and $\e_0$ satisfy $0 <\e \le \frac{\e_0}{2} \le \frac{1}{8}$. 
We have that $\text{support}(f_{r;M,N}) \subseteq [M,2M] \times [N,2N]$ and we have for  $M \ll T^{\frac{3}{2}+\varepsilon}$, $M \asymp N$,
and $1 \le |r| \ll \frac{M}{T_0} T^{\varepsilon}$ that 
\begin{equation*}
  \label{ffpartials}
  x^m y^n f_{r;M,N}^{(m,n)}(x,y) \ll   T^{\frac{3}{2} \varepsilon} P^{n} \text{ where } P = T^{1+\varepsilon} T_{0}^{-1}.
\end{equation*}
\end{lem}
Note that this trivially gives $ x^m y^n  f_{r;M,N}^{(m,n)}(x,y) \ll  T^{\frac{3}{2} \varepsilon}    P^{m+n}$ with  $P=T^{1+\varepsilon} T_{0}^{-1}$.
The proof of this technical lemma is deferred to Appendix \ref{appendix2}. 
We now apply Conjecture \ref{divconj} to those $I_{M,N}$ occurring in the sum \eqref{IO1dyadic2}.
Observe that $f(x,y):=T^{- \frac{3 \varepsilon}{2}} f_r(x,y)$ is supported in $[M,2M] \times [N,2N]$ with $M \asymp N$, and (by Lemma \ref{fpartials}) satisfies the condition  \eqref{fcond}. 
We must also check that $|r| \ll M^{\frac{1}{2}-\e_2}$ for some $\e_2 > 0$. 
Since  $0 < \varepsilon \le \frac{\varepsilon_0}{2} \le \frac{1}{8}$, we deduce
\[
   |r| \ll \frac{M}{T_0} T^{\varepsilon} \le M T^{-\frac{3}{4}-\frac{\varepsilon_0}{2}}
   \ll M^{\frac{1}{2}} T^{\frac{3}{4} + \frac{\varepsilon}{2}-\frac{3}{4}-\frac{\varepsilon_0}{2}}
   \le M^{\frac{1}{2}} T^{-\frac{\varepsilon_0}{4}} \ll M^{\frac{1}{2}-\frac{\e_0}{8}}
\]
(using that $M \ll T^{\frac{1}{2}+\e} < T^2$). 
Therefore by Conjecture \ref{divconj}
\begin{equation*}
\begin{split}
   \label{IMNnew}
     I_{M,N} & = \frac{T}{\sqrt{MN}} 
\sum_{0 < |r| \ll \frac{M}{T_0} T^{\varepsilon}}
 \sum_{i_1=1}^{3} \sum_{i_2=1}^{3} 
    \prod_{j_1 \ne i_1} \zeta(1-a_{i_1}+a_{j_1})   \prod_{j_2 \ne i_2} \zeta(1-b_{i_2}+b_{j_2})  \\
  &  \times \sum_{\ell=1}^{\infty} \frac{c_{\ell}(r)G_{\I}(1-a_{i_1},\ell)G_{\J}(1-b_{i_2},\ell)  }{\ell^{2-a_{i_1}-b_{i_2}}}
 \int_{\max(0,r)}^{\infty} f_r(x,x-r) x^{-a_{i_1}}(x-r)^{-b_{i_2}} dx  +O(\mathcal{E}_{M,N}) 
 \end{split}
 \end{equation*}
where 
\begin{equation*}
  \label{EMN}
   \mathcal{E}_{M,N}  =
   \frac{T^{1+\frac{3 \e}{2}}}{\sqrt{MN}} \sum_{0 < |r| \ll \frac{M}{T_0} T^{\varepsilon}}  \Big(\frac{T}{T_0} \Big)^{C}M^{\vartheta+\varepsilon}.
\end{equation*} 
Next we consider the contribution of the errors $\mathcal{E}_{M,N}$ to \eqref{IO1dyadic2}.  
\begin{lem}   
Let $\varepsilon >0$, $0 < \vartheta < 1$,  and $T$ is sufficiently large with respect to $\varepsilon$.  Then 
\begin{equation}
  \label{sumMNEMN}
 \sum_{\substack{ M, N \\ M \asymp N, MN \ll T^{3+\varepsilon}}}   \mathcal{E}_{M,N}  \ll 
  T^{\frac{3 \vartheta}{2}+\frac{5 \varepsilon}{2}} \Big( \frac{T}{T_0} \Big)^{1+C}.
\end{equation}
\end{lem}
\begin{proof}
Since $M \asymp N$, we have
\begin{align*}
 \mathcal{E}_{M,N} &  \ll \frac{T^{1+\frac{3 \e}{2}}}{M}  \sum_{0 < |r| \ll \frac{M}{T_0} T^{\varepsilon}}  
  \Big( \frac{T}{T_0} \Big)^{C}M^{\vartheta+\varepsilon}
  \ll  T^{\frac{3 \e}{2}} \Big( \frac{T}{M}   \Big)
      \Big( \frac{T}{T_0} \Big)^{C} M^{\vartheta +\varepsilon}
      \Big(
       \frac{M}{T_0}  T^{\varepsilon} \Big)
    \ll 
    T^{ \frac{ 5\varepsilon}{2}} \Big( \frac{T}{T_0} \Big)^{1+C} M^{\vartheta}. 
\end{align*} 
Summing this over $M$ and $N$ we find
 \begin{equation*}
 \begin{split}
   \label{sumMN1}
   \sum_{\substack{MN \ll T^{3+ \e} \\ M \asymp N}}  T^{\frac{5 \e}{2} } \Big( \frac{T}{T_0} \Big)^{1+C} M^{\vartheta} 
 &   \ll  T^{\frac{5 \varepsilon}{2}} \Big( \frac{T}{T_0} \Big)^{1+C}  
   \sum_{\substack{M \ll T^{\frac{3}{2}+\frac{\e}{2}} \\
  N \asymp M} } M^{\vartheta}  
     \ll T^{\frac{3 \vartheta}{2}+3 \e} \Big( \frac{T}{T_0} \Big)^{1+C}
\end{split}
\end{equation*}
and the lemma is established. 
\end{proof} 
The last lemma will be used to bound the contribution of the 
error terms $\mathcal{E}_{M,N}$ when \eqref{IMNnew} is inserted into \eqref{IO1dyadic2}.
Next we give an initial simplification of  the main term of $I_{O}^{(1)}$
that is  obtained, via \eqref{IO1dyadic2},  from the main term in \eqref{IMNnew}.  
\subsection{Extending the range to $0 <|r| \le R_0$}
 First, we may extend the range of $r$ in \eqref{IMNnew}.
Let 
\begin{equation}
  \label{R0}
R_0 = T^{5}. 
\end{equation}  
 Notice that, by \eqref{f*2}, \eqref{fMN}, the integral in \eqref{IMNnew} is 
\begin{equation}
  \label{iMNr}
 i_{M,N,r;i_1,i_2} =  \int_{\max(0,r)}^{\infty} W \Big( \frac{x}{M} \Big) W \Big(  \frac{x-r}{N} \Big) f^{*}(x,x-r) 
 x^{-a_{i_1}}(x-r)^{-b_{i_2}} \,  dx.
\end{equation}

In \eqref{IMNnew}, we seek (at the cost of adding to the error term there) to weaken the conditions
on the variable of summation $r$ to just $0 < |r| \le R_0$.  We argue as follows. 
We begin by observing that the variable of integration $x$ in \eqref{iMNr} is  constrained to lie
in the interval $(M,2M) \cap (N+r, 2N+r)$ since $W$ is supported in $[1,2]$.  
If this open interval is the empty set then $ i_{M,N,r;i_1,i_2} =0$. 
Suppose without loss of generality that $r \ge 1$.   Notice that $x>N+r$ implies $x-r \ge N \ge 2^{-\frac{1}{2}}$ and thus $x \ge r+2^{-\frac{1}{2}}$.  We must  also have  $r \le \frac{5}{3}M$.  This is since
we must have $N+r \le 2M$ (if not, then $(M,2M) \cap (N+r, 2N+r)$ is empty).
Therefore $r \le 2M-N \le 2M- \frac{M}{3}= \frac{5M}{3}$ since we are assuming \eqref{MNcondition}.  By these observations and 
\eqref{f*bound} it follows that, for $r \gg \frac{M}{T_0} T^{\varepsilon}$ and $j$ sufficiently large 
we have 
\begin{equation}
\begin{split}
  \label{fstarxr}
  f^{*}(x,x-r) 
  & \ll \frac{T^{3 \varepsilon+3 \delta}}{ |\log(\frac{x}{x-r})  T_0|^j} \le \frac{T^{3 \varepsilon+3 \delta}}{ |\log(\frac{2M}{2M-r})  T_0|^j} \text{ since } x \in [r+2^{-\frac{1}{2}},2M] \\
  & =   \frac{T^{3 \varepsilon+3 \delta}}{ |\log(1-\frac{r}{2M})  T_0|^j}   \ll   \frac{T^{3 \varepsilon+3 \delta}}{  |\frac{r}{2M} T_0|^j} \text{ since } |\log(1-x)| \gg x \text{ for }
  0 < x \le \frac{5}{6} \\
  & \ll T^{3 \varepsilon+3 \delta-j \varepsilon}  \ll T^{-B}.
\end{split}
\end{equation}

By \eqref{iMNr}, \eqref{sizerestrictiondelta}, and \eqref{fstarxr} one has $i_{M,N,r;i_1,i_2} \ll M^{1+2 \delta}T^{-B}$ for $r \gg \frac{M}{T_0} T^{\varepsilon}$.  One finds, similarly, that $i_{M,N,r;i_1,i_2} \ll M^{1+2 \delta}T^{-B}$ for $r$ satisfying $-r \gg \frac{M}{T_0}  T^{\varepsilon}$. 
Therefore the difference $\Delta_{M,N}$ made to the main term in \eqref{IMNnew} by extending the range of
the outer summation to include all non-zero integers $r$ with $|r| \le R_0$ satisfies
\begin{equation}
 \Delta_{M,N} :=  \frac{T}{\sqrt{MN}} \sum_{ \frac{M}{T_0} T^{\varepsilon} \le    |r| \le R_0 }
   \sum_{i_1=1}^{3} \sum_{i_2=1}^{3}  |{\bf c}_{i_1,i_2}|
 \sum_{\ell=1}^{\infty}  |u_{\ell,r;i_1,i_2}|   M^{1+2\delta}T^{-B} 
\end{equation}
where 
\begin{equation}
\begin{split}
  \label{ci1i2ulr}
 {\bf c}_{i_1,i_2}  & =  
    \prod_{j_1 \ne i_1} \zeta(1-a_{i_1}+a_{j_1})   \prod_{j_2 \ne i_2} \zeta(1-b_{i_2}+b_{j_2}), \\
u_{\ell,r;i_1,i_2}  & =  c_{\ell}(r)G_{\I}(1-a_{i_1},\ell)G_{\J}(1-b_{i_2},\ell)  \ell^{-2+a_{i_1}+b_{i_2}}.
\end{split}
\end{equation}
Observe that by \eqref{initialconditionB}  we have 
\begin{equation}
  \label{ci1i2bd}
 {\bf c}_{i_1,i_2}  \ll (\log T)^4
\end{equation}
and 
\begin{equation}
\begin{split}
  \label{ulrbound}
  \sum_{\ell \ge 1} |u_{\ell,r;i_1,i_2}|  & \ll \sum_{\ell \ge 1} \gcd(\ell,r) \ell^{-2+5 \delta }
  = \sum_{g \mid r} g^{-1+5 \delta}  \sum_{\substack{ \ell' \ge 1 \\ \gcd(\ell',r/g)=1 }} (\ell')^{-2 +5 \delta} 
  \ll d_2(|r|).
\end{split}
\end{equation}
Here we have used the well-known estimate $|c_{\ell}(r)| \le \gcd(\ell,r)$.
In addition, we applied 
 $|\ell^{a_{i_1}+b_{i_2}}|
\le \ell^{2 \delta}$  by
\eqref{sizerestrictiondelta} 
and thus by Lemma \ref{GIbd} of Appendix \ref{appendix1}  $|G_{\I}(1-a_{i_1},\ell)G_{\J}(1-b_{i_2},\ell)|  \le  \ell^{2 \delta} d_2(\ell^2)^2 \ll \ell^{3 \delta}$. 
Since $M^{2 \delta} \le M^{\frac{1}{3}} \ll (T^2)^{\frac{1}{3}} \le T$ as $\delta < \frac{1}{6}$, it follows that 
\begin{equation}
  \label{Erbound}
  \Delta_{M,N} \ll  (\log T)^4   T^{2-B} \Big(\sum_{1 \le |r| \le R_0} d_2(|r|) \Big) 
   \ll  (\log T)^4   T^{2-B}  R_0 (\log R_0)  \ll (\log T)^{-1}, 
\end{equation}
by choosing $B=8$.  
By this, combined with \eqref{IO1dyadic2}, \eqref{IMNnew}, and \eqref{sumMNEMN}, 
we obtain
   \begin{equation}
  \label{IO1b}
  I_{O}^{(1)} = \sum_{\substack{ M, N \\ M \asymp N, MN \ll T^{3+\varepsilon}
  \\ M, N \gg T_{0}^{1-\varepsilon}
  }} \tilde{I}_{M,N} + O \Big(   T^{\frac{3 \vartheta}{2}+3 \e} \Big( \frac{T}{T_0} \Big)^{1+C}
  \Big)
\end{equation}
where 
\begin{equation}
\begin{split}
  \label{ItildeMN}
  \tilde{I}_{M,N} &  = 
  \frac{T}{\sqrt{MN}} 
\sum_{0 < |r| \le R_0}
 \sum_{i_1=1}^{3} \sum_{i_2=1}^{3}   {\bf c}_{i_1,i_2} \Big(
     \sum_{\ell=1}^{\infty}u_{\ell,r;i_1,i_2} \cdot  i_{M,N,r;i_1,i_2} \Big).
\end{split}
\end{equation}
\subsection{Removing the conditions on $M$ and $N$}
The explicit conditions on $M$ and $N$ in \eqref{IO1b} shall be removed
so that we may  sum  over all $M$ and $N$.  For $M$ and $N$ not satisfying the conditions 
in the summand of \eqref{IO1b} we shall show that $ i_{M,N,r;i_1,i_2}$ is small.   
If  $M \not \asymp N$, then we have $N > 3M$ or $N < M/3$.  Suppose without loss of generality that $N > 3M$. 
In the integrand in \eqref{iMNr} one effectively has both
$M \le x \le 2M$ and $N \le x-r \le 2N$,  and so
$|\log ( \frac{x}{x-r} )| = \log  ( \frac{x-r}{x}  ) \ge \log(\frac{N}{2M}) \ge \log(\frac{3}{2})$.
Therefore  by \eqref{sizerestrictiondelta}, \eqref{f*bound}, and \eqref{iMNr}, we obtain the bound
\begin{equation*}
\begin{split}
   i_{M,N,r;i_1,i_2} &  \ll \int_{M}^{2M} O_j \Big(  \frac{T^{3 \varepsilon+3 \delta}}{T_{0}^j} \Big)
   (MN)^{\delta} dx
   \ll_{j} M T^{3 \varepsilon+3 \delta} T_{0}^{-j} (MN)^{\delta} \\
   &  \ll \frac{T^{3 + \varepsilon}}{T_{0}^{1-\varepsilon}}  T^{3 \varepsilon+3 \delta} T_{0}^{-j}
   (T^{3+\e})^{\delta}
   \le T^{4 + 6 \varepsilon-\frac{3}{4}(j+1)} \ll T^{-B}
\end{split}
\end{equation*}
by choosing $j$ sufficiently large, where we have used $MN \ll T^{3+\varepsilon}$, $N \gg T_{0}^{1-\e}$
\eqref{cond3}, and $\delta< \frac{1}{6}$. 
  Thus for these $M$ and $N$ it follows, by an argument similar to that which gave \eqref{Erbound}, 
 that $ \tilde{I}_{M,N} \ll (\log T)^{-2}$. 

We now treat the case $M$ or $N \ll T_{0}^{1-\varepsilon}$. Without loss of generality we may assume $M,N \ge 2^{-\frac{1}{2}}$, $M \ll
T_{0}^{1-\varepsilon}$ and $r\ge 1$. 
Observe that we have $x \ge N+r > r \ge 1$ and thus   $|\log(\frac{x}{x-r})| \ge \log(\frac{x}{x-1}) > x^{-1} \ge (2M)^{-1} 
\gg T_{0}^{\varepsilon-1}$.    As in the calculation \eqref{fstarxr}, it follows that $f^{*}(x,x-r) \ll T^{3 \varepsilon+3 \delta} T_{0}^{-\varepsilon j}
\ll T^{-B}$, for any $B>0$ and again we obtain  $ \tilde{I}_{M,N}  \ll (\log T)^{-2}$.  
By extending the range of summation in \eqref{IO1b}, so as to include all $M,N$ that satisfy
$MN \ll T^{3 +\varepsilon}$ and have either $M \not \asymp N$ or $\min(M,N) \ll T_{0}^{1-\varepsilon}$, we just add 
 $(\log T)^2 (\log T)^{-2}  \ll 1$ to the error term there, since there are $O(\log T)$ choices for each of $M$ and $N$. 
 
Finally, we include  the contribution, $\Delta'$ from pairs $M,N$ satisfying $MN \gg T^{3 +\varepsilon}$.  By \eqref{fstarV} and Lemma \ref{Stirling} (iv)
it follows that, for any constant $A>0$, we have $f^{*}(x,y) \ll ( \frac{T^3}{xy} )^{A}$ when $xy \ge T^{3+\varepsilon}$.
Therefore, by \eqref{iMNr}, \eqref{ci1i2bd}, and \eqref{ulrbound} 
 we have
\begin{align*}
 \Delta'  
 & \ll   \sum_{r \ne 0}
   \sum_{i_1,i_2} |{\bf c}_{i_1,i_2}| \sum_{\ell=1}^{\infty} |u_{\ell,r;i_1,i_2}|
  \sum_{\substack{ M,N \\   MN \gg T^{3+\varepsilon}  \\ M-2N < r < 2M-N }}   \frac{T}{\sqrt{MN}}
  \int_{\max( M, N+r,r,0)}^{\min(2M,2N+r)}  \Big(\frac{T^3}{(x(x-r)}  \Big)^{A}  
  (MN)^{\delta}
  dx \\
  & \ll  (\log T)^4  T \sum_{ r \ne 0 } d_2(|r|) \sum_{\substack{M,N \\ MN \gg \max(T^{3+\varepsilon}, |r|)}  }  (MN)^{-\frac{1}{2}+\delta} 
  \Big( \frac{T^3}{MN} \Big)^{A} \min(M,N), 
\end{align*}
where the constant $A$ is arbitrary.  
In the last sum we put $MN=H$ and note that $\min(M,N) \le H^{\frac{1}{2}}$.
Observe that $H=2^{\frac{h-2}{2}} \gg T^{3+\varepsilon}$, for some integer 
$h=h(H) \ge 0$.  For each $H$ there are at most $h(H)+1 \ll 3 + \log(H) \ll \log H$ pairs 
$M,N$ with $MN=H$.  It follows that 
\begin{align*}
   \Delta'  
 & \ll T (\log T)^4  \sum_{H \gg T^{3+\varepsilon}}
 H^{-\frac{1}{2}+\delta} \Big( \frac{T^3}{H} \Big)^{A}
 H^{\frac{1}{2}} (\log H) \sum_{1 \le r \ll H} d_2(r)
 \ll T (\log T)^4   \sum_{H \gg T^{3+\varepsilon}}
 H^2   \Big( \frac{T^3}{H} \Big)^{A},
\end{align*}
where we have used $\delta \le \frac{1}{2}$.  
If $A > 2$, the last series converges and is bounded by a constant times the first term in the sum so that 
$\Delta' \ll T (\log T)^4  (T^{3+\varepsilon})^2  T^{-\varepsilon A}
  \ll T^{7 +3 \varepsilon-\varepsilon A}$.
By choosing $A = \frac{7}{\varepsilon}+3$, we obtain $\Delta' \ll 1$.  
All this allows us to remove the restrictions on $M$ and $N$ to obtain 
\begin{equation}
 \label{IO1expression}
  I_{O}^{(1)} = \sum_{M,N}  \tilde{I}_{M,N} + O \Big(   T^{\frac{3 \vartheta}{2}+3\varepsilon} \Big( \frac{T}{T_0} \Big)^{1+C}
  \Big)
\end{equation} 
where $\tilde{I}_{M,N}$ is given by \eqref{ItildeMN}.   
\subsection{Reintroducing the smooth partition of unity}
We simplify the last sum by rearranging summation variables so that 
\begin{equation}
  \label{ITildeMNsum}
   \sum_{M,N}  \tilde{I}_{M,N}=T \sum_{i_1=1}^{3} \sum_{i_2=1}^{3}  {\bf c}_{i_1,i_2} \sum_{0 < |r| \le R_0} \sum_{\ell=1}^{\infty} u_{\ell,r;i_1,i_2}
  \Bigg( \sum_{M,N}  \frac{1}{\sqrt{MN}} i_{M,N,r;i_1,i_2} \Bigg)
\end{equation}
where we recall that $ i_{M,N,r;i_1,i_2}$ is given by \eqref{iMNr} and ${\bf c}_{i_1,i_2},  u_{\ell,r;i_1,i_2}$ are given by 
\eqref{ci1i2ulr}.  We simplify the inner sum by exchanging summation and integration to obtain 
\begin{equation}
   \label{MNsum}
    \sum_{M,N}  \frac{1}{\sqrt{MN}} i_{M,N,r;i_1,i_2} =  
      \int_{\max(0,r)}^{\infty}  \frac{1}{\sqrt{x(x-r)}}  \Bigg( \sum_{M,N}  W_0 \Big( \frac{x}{M} \Big) W_0 \Big(  \frac{x-r}{N} \Big)  \Bigg) f^{*}(x,x-r) 
 x^{-a_{i_1}}(x-r)^{-b_{i_2}} \,  dx,
\end{equation} 
since $W(x)=x^{-\frac{1}{2}} W_0(x)$.   Observe that the function in brackets is zero if $x \le \max(0,r) +2^{-\frac{1}{2}}$
and is  equal to one for all $x >\max(0,r)+1$.  
However, this is not true if $\max(0,r) +2^{-\frac{1}{2}} < x \le \max(0,r)+1$.  
By \eqref{W0dyadic2} we have that, for $0 \ne r \in \Z,$
\[
  \sum_{M,N}  W_0 \Big( \frac{x}{M} \Big) W_0 \Big(  \frac{x-r}{N} \Big) -1 =
  \begin{cases}
  0 & \text{ if } x >  \max(0,r)+1, \\
    W_0 \Big( \frac{x -\max(0,r)}{2^{-\frac{1}{2}}} \Big) -1& \text{ if }   \max(0,r) +2^{-\frac{1}{2}}< x \le  \max(0,r)+1, \\
    - 1 & \text{ if }   \max(0,r) \le  x \le  \max(0,r)+2^{-\frac{1}{2}}.
  \end{cases}
\]
With these observations, we see that  \eqref{MNsum} becomes
\begin{equation}
\begin{split}
  \label{MNsum2}
\sum_{M,N}  \frac{1}{\sqrt{MN}} i_{M,N,r;i_1,i_2} & =  
      \int_{\max(0,r)}^{\infty}  \frac{1}{\sqrt{x(x-r)}} f^{*}(x,x-r) 
 x^{-a_{i_1}}(x-r)^{-b_{i_2}} \,  dx \\
& + O \Big( 
   \int_{\max(0,r)}^{\max(0,r)+1} 
   |f^{*}(x,x-r) |
 (x-\min(0,r))^{-\frac{1}{2}+\delta}(x-\max(0,r))^{-\frac{1}{2}-\delta} \,  dx
\Big).
\end{split}
\end{equation}

The contribution from the  error term of \eqref{MNsum2} to   \eqref{ITildeMNsum} is
\begin{equation}
\begin{split}
  \label{diff1}
 & T   \sum_{i_1,i_2=1}^{3}  |{\bf c}_{i_1,i_2}|
 \sum_{1 \le |r| \le R_0}  \sum_{\ell \ge 1} |u_{\ell,r;i_1,i_2} |
 \int_{\max(0,r)}^{\max(0,r)+1} 
   |f^{*}(x,x-r) |
 (x-\min(0,r))^{-\frac{1}{2}+\delta}(x-\max(0,r))^{-\frac{1}{2}-\delta} \,  dx \\
 & \ll T (\log T)^4  \sum_{1 \le |r| \le R_0}  d_2(|r|)
  \int_{\max(0,r)}^{\max(0,r)+1}  | f^{*}(x,x-r) | (x-\min(0,r))^{-\frac{1}{2}+\delta}(x-\max(0,r))^{-\frac{1}{2}-\delta} \, dx
\end{split}
\end{equation}
where we have used \eqref{ci1i2bd} and \eqref{ulrbound}. 
In the case that $r < 0$, we have that $|\log (\frac{x}{x-r})| = \log( \frac{x+|r|}{x}) \ge \log(1+|r|) \ge \log(2)$
for $\max(0,r) < x\le \max(0,r)+1$.  An analogous argument gives the same bound in the case $r >0$. 
It follows, using the bounds from which \eqref{f*bound} was deduced,
 that the last expression in \eqref{diff1} is bounded by 
\begin{equation}
\begin{split}
  \label{diff2}
 &  \ll_{\e,j}  T (\log T)^4  \sum_{1 \le |r| \le R_0}  d_2(|r|)  \int_{\max(0,r)}^{\max(0,r)+1}
   (x-\min(0,r))^{-\frac{1}{2}+\delta-\e}(x-\max(0,r))^{-\frac{1}{2}-\delta-\e} 
   T^{3 \varepsilon+3 \delta}  T_0^{-j}  \, dx \\
   &=  \frac{2 T^{1+3 \e +3 \delta} (\log T)^4 }{T_0^j  } 
    \sum_{1 \le r \le R_0}  d_2(r) \int_{0}^{1} (y+r)^{-\frac{1}{2}+\delta-\e} y^{-\frac{1}{2}-\delta-\e} dy  
   \\
   & \le  \frac{2 T^{1+3 \e +3 \delta} (\log T)^4 }{T_0^j  } 
    \sum_{1 \le r \le R_0}  d_2(r) r^{-\frac{3}{8}} 
    \int_{0}^{1} y^{-\frac{3}{4}} dy 
    \ll T^7 T_{0}^{-j} \ll 1,
\end{split}
\end{equation}
since $\delta < \frac{1}{12}$, $\e$ is sufficiently small, and in the last inequality 
$j$ is chosen sufficiently large.  In summary, we find from \eqref{ITildeMNsum}, \eqref{MNsum2}, \eqref{diff1}, and
\eqref{diff2} that
\begin{equation*}
   \sum_{M,N}  \tilde{I}_{M,N}=
   T \sum_{i_1=1}^{3} \sum_{i_2=1}^{3}  {\bf c}_{i_1,i_2} \sum_{0 < |r| \le R_0} \sum_{\ell=1}^{\infty} u_{\ell,r;i_1,i_2}
  \int_{\max(0,r)}^{\infty} f^{*}(x,x-r) 
 x^{-\frac{1}{2}-a_{i_1}}(x-r)^{-\frac{1}{2}-b_{i_2}} \,  dx + O(1).  
\end{equation*}
Combining this last expression with \eqref{IO1expression} we obtain the following.
\begin{prop}  \label{IO1prop}
Let
$a_i,b_i$  satisfy 
$|a_i|, |b_i| \ll (\log T)^{-1}$  for $i=1,2,3$ and the condition
 \eqref{initialconditionB}. 
Assume Conjecture \ref{divconj} holds for some positive $\theta$ and $C$, then for any $\varepsilon >0$
and sufficiently large $T$
 \begin{equation}
   \label{IO1sum}
   I_{O}^{(1)} = \sum_{i_1=1}^{3} \sum_{i_2=1}^{3} I_{(i_1,i_2)}^{(1)} 
   + O \Big( T^{\frac{3 \vartheta}{2}+3 \e} \Big( \frac{T}{T_0} \Big)^{1+C}
  \Big)
 \end{equation}
 where
 \begin{equation*}
 \begin{split}
    \label{I1i1i2formula}
     & I_{(i_1,i_2)}^{(1)} 
      = 
     T
     \prod_{j_1 \ne i_1} \zeta(1-a_{i_1}+a_{j_1})   \prod_{j_2 \ne i_2} \zeta(1-b_{i_2}+b_{j_2})   \\
     & \times 
     \sum_{1 \le |r| \le R_0}
     \int_{\max(0,r)}^{\infty}  f^{*}(x,x-r) x^{-\frac{1}{2}-a_{i_1}}(x-r)^{-\frac{1}{2}-b_{i_2}}  dx \cdot
    \sum_{\ell=1}^{\infty} \frac{c_{\ell}(r)G_{\I}(1-a_{i_1},\ell)G_{\J}(1-b_{i_2},\ell)  }{\ell^{2-a_{i_1}-b_{i_2}}} \\
 \end{split}
 \end{equation*}
 with  $f^{*}(x,y)$ is defined by \eqref{fstardefn} and \eqref{gIJ}. 
 \end{prop}
It is important to note that the  definition of $f^{*}(x,y)$ is independent of $r$.  
 In the next section, the main term in \eqref{IO1sum} is further simplified. 

\section{Further evaluation of $I_O^{(1)}$. Evaluation of  $ I_{(i_1,i_2)}^{(1)} $} \label{section6}
 By Proposition \ref{IO1prop} the evaluation of $I_{O}^{(1)}$ has been reduced to 
 the evaluation of $ I_{(i_1,i_2)}^{(1)}$.  In this calculation  we shall encounter the Dirichlet series
\begin{equation}
  \label{HIJa1b1}
   H_{\I,\J;\{ a_1 \} , \{ b_1 \} }(s) =   \sum_{r=1}^{\infty}
      \sum_{\ell=1}^{\infty} \frac{c_{\ell}(r)G_{\I}(1-a_1,\ell)G_{\J}(1-b_1,\ell)  }{\ell^{2-a_1-b_1} r^{a_1+b_1+2s}}.
\end{equation}
Moreover,  $ H_{\I,\J;\{ a_1 \} , \{ b_1 \} }(s)$ equals a product of $\zeta$ functions times a nice infinite product  
$\mathcal{C}_{\I,\J; \{ a_1 \}, \{ b_1 \}}(s)$.
In this section, we shall evaluate the integrals $ I_{(i_1,i_2)}^{(1)} $ under the less
restrictive conditions \eqref{sizerestrictiondelta}.
\begin{prop} \label{Hid}
Let $\delta \in (0,\frac{1}{12})$, $\e \in (0,\frac{1}{4})$,  and assume $|a_i|,|b_i| \le \delta$ for $i=1,2,3$.  \\
(i) For $\Re(s) > \frac{3}{2} \delta$, 
\begin{equation}
  \label{Hfactorization}
    H_{\I,\J; \{ a_1 \} , \{ b_1 \}}(s) =  \zeta(a_1+b_1+2s) \prod_{\substack{k_1 \ne 1 \\ k_2 \ne 1}} \zeta(1+a_{k_1}+b_{k_2}+2s)     \mathcal{C}_{\I,\J; \{ a_1 \},\{ b_1 \}}(s)
\end{equation}
where
\begin{equation}
  \label{CIJa1b1sdefn}
 \mathcal{C}_{\I,\J; \{ a_1 \}, \{ b_1 \}}(s) = \prod_{p} \mathcal{C}_{\I,\J; \{ a_1 \}, \{ b_1 \}}(p;s),
\end{equation}
\begin{equation}
   \label{lps}
   \mathcal{C}_{\I,\J; \{ a_1 \}, \{ b_1 \}}(p;s) = Q(p^{-a_2},p^{-a_3},p^{-b_2},p^{-b_3};p^{-a_1},p^{-b_1};p^{-1},p^{-2s}),
\end{equation}
and
\begin{equation}
\begin{split}
 \label{Q}
  Q(X_2,X_3,Y_2,Y_3; X_1,Y_1; U,V)
  =  
 \Bigg( 1  & + \Bigg(
 \frac{UVX_2Y_2}{1-UVX_2 Y_2} \frac{1-UX_3 X_{1}^{-1}}{1-X_3 X_{2}^{-1}} \frac{1-UY_3Y_{1}^{-1}}{1-Y_3 Y_{2}^{-1}} \\
 & + \frac{UVX_2Y_3}{1-UVX_2 Y_3} \frac{1-UX_3 X_{1}^{-1}}{1-X_3 X_{2}^{-1}} \frac{1-UY_2 Y_{1}^{-1}}{1-Y_2 Y_3^{-1}} \\
 & + \frac{UVX_3 Y_2}{1-UVX_3 Y_2} \frac{1-UX_2 X_{1}^{-1}}{1-X_2 X_{3}^{-1}} \frac{1-UY_3 Y_{1}^{-1}}{1-Y_3 Y_{2}^{-1}} \\
 & + \frac{UVX_3 Y_3}{1-UVX_3 Y_3}  \frac{1-UX_2 X_{1}^{-1}}{1-X_2 X_{3}^{-1}} \frac{1-UY_2 Y_{1}^{-1}}{1-Y_2 Y_{3}^{-1} }
 \Bigg)(1-UV^{-1} X_{1}^{-1}Y_{1}^{-1}) \Bigg)
 \\
& \times (
1- UVX_2Y_2 )(
1- UVX_2Y_3 )(
1- UVX_3Y_2 )(
1- UVX_3Y_3 ).
\end{split}
\end{equation}
(ii) We have, for $\Re(s)=\sigma \ge -\frac{1}{2}+\delta+\e$,  
\begin{equation}
  \label{CIJexpansion}
     \mathcal{C}_{\I,\J; \{ a_1 \} , \{ b_1 \}}(p;s) =1 -p^{-a_2-a_3-b_2-b_3-2-4s} +O_{\varepsilon}(p^{8 \delta+\vartheta(\sigma)})
\end{equation}
where 
\begin{equation}
   \label{thetasigma}
    \vartheta(\sigma) =
    \begin{cases}
    -2 & \text{ if } \sigma \ge 0; \\
    -2-2 \sigma & \text{ if } 0 > \sigma \ge -\frac{1}{4}; \\
    -3-6 \sigma & \text{ if } - \frac{1}{4} > \sigma
    \end{cases}
\end{equation}
and hence $\mathcal{C}_{\I,\J; \{ a_1 \}, \{ b_1 \}}(s)$ is holomorphic and absolutely convergent for $\Re(s) > -\frac{1}{4} +  \delta$.   Furthermore,  $ H_{\I,\J; \{ a_1 \} , \{ b_1 \}}(s)$ is holomorphic in $\Re(s) > -\frac{1}{4}+\delta$
with 
the exception of simple poles at 
\begin{equation}
  \label{Hpoles}
  \frac{1}{2}-\frac{(a_1+b_1)}{2} \text{ and }
  -\frac{a_{k_1}+b_{k_2}}{2} \text{ for } k_1 \ne 1, k_2 \ne 1.
\end{equation} 
\end{prop}
The following proposition shows that the local factors $\mathcal{C}_{\I,\J; \{ a_1 \}, \{ b_1 \}}(p;s)$
satisfy certain identities at special values of $s$, relating them to the local factors $A_{p;\I, \J}(s)$ which occur in 
Lemma \ref{ZIJ}.  
\begin{prop} \label{ACidentities} 
Let $\I = \{a_1, a_2, a_3 \}$ and $\J = \{b_1, b_2, b_3 \}$.  
We have the identities
\begin{itemize}
\item[(i)]
$ \cA_{\I_{\{a_1 \}}, \J_{ \{ b_1 \} }}(0) =  
  \cC_{\I,\J; \{ a_1 \} , \{ b_1 \} }(0),$
\item[(ii)]
$ \cA_{\I,\J}(-a_1-b_1) =   \cC_{\I,\J; \{ a_1 \} , \{ b_1 \}}( -\tfrac{a_1+b_1}{2}),$ 
\item[(iii)]   
$ \cC_{\I,\J; \{a_1 \}, \{ b_1 \} }(-\tfrac{a_{2}+b_{2}}{2})=
     \cC_{-\J,-\I; \{-b_{3}\}, \{ -a_{3} \}}(\tfrac{b_{2}+a_{2}}{2}).$
\item[(iv)]   
Let $(i_1,i_2) \in \{ 1,2,3 \}^2$ and $(k_1,k_2) \in \{ 1,2,3 \}^2$ with $k_1 \ne i_1$ and $k_2 \ne i_2$. 
Then 
$$ \cC_{\I,\J; \{a_{i_1} \}, \{ b_{i_2} \} }(-\tfrac{a_{k_1}+b_{k_2}}{2})=
     \cC_{-\J,-\I; \{-b_{r_2}\}, \{ -a_{r_1} \}}(\tfrac{b_{k_2}+a_{k_1}}{2})$$
where       $r_1 =r(i_1,k_1)$ and $r_2=r(i_2,k_2)$ are defined as follows:   
\begin{equation}
\begin{split}
  \label{rdefinition}
   & \text{ Given distinct elements } i,k \text{ of }  \{1,2,3\},  \text{ then } r=r(i,k) \text{ is the unique number } r \\
   & \text{ such that }  \{ 1, 2, 3\} = \{ i, k , r\}. 
\end{split}
\end{equation}
\end{itemize}
\end{prop}
We shall prove Proposition \ref{Hid} and \ref{ACidentities} in Appendix \ref{appendix1}. Regarding, in particular,
our proof of Proposition \ref{ACidentities}: the proof of part (i) shall be reduced to the verification of 
a polynomial identity and (ii) and (iii) are deduced from (i).  
Alternately
parts (i) and (ii) also follow from an identity in  \cite[Sections 3,4]{CK3}.  
In fact, the argument in \cite{CK3} establishes such identities for general sets 
$\I = \{a_1, a_2, \ldots, a_k \}$ and $\J = \{b_1, b_2, \ldots, b_k \}$.
Next we have a result (also to be proved in Appendix \ref{appendix1}) which gives useful bounds for $H_{\I,\J;\{ a_1\}, \{ b_1 \}}(s)$.
\begin{lem} \label{Hbounds}
Let $\delta \in (0,\frac{1}{7})$ and assume  assume $|a_i|,|b_i| \le \delta$ for $i=1,2,3$.
\begin{itemize}
\item[(i)]  For $\Re(s) = 2 \delta$, we have 
\begin{equation*}
  \label{Hbd}
  |H_{\I,\J;\{ a_1\}, \{ b_1 \}}(s)| 
  \ll |s|^{\frac{1}{2}}.
\end{equation*}
\item[(ii)]  For $s$ satisfying $|\Re(s)-\frac{1}{2}| \le 2 \delta$ and $\Im(s) \asymp T$, then 
\begin{equation*}
  \label{Hbd2}
    |H_{\I,\J;\{ a_1\}, \{ b_1 \}}(s)| \ll T^{4 \delta}. 
\end{equation*}
\end{itemize}
\end{lem}
We now state our formulae for $I_{O}^{(1)}$ and $I_{O}^{(2)}$. The proofs of Propositions \ref{I1i1i2} and \ref{I2i1i2}
which follow shall make use of the previous propositions and lemma. 
  \begin{prop}  \label{I1i1i2}
 Let
$a_i,b_i$  satisfy 
$|a_i|, |b_i| \ll (\log T)^{-1}$,  for $i=1,2,3$ and  the conditions \eqref{initialconditionA}, 
 \eqref{initialconditionB}, and  \eqref{initialconditionC}.
Assume Conjecture \ref{divconj} holds for some positive $\theta \ge \frac{1}{2}$ and $C$.
Then, for any $\varepsilon >0$
and  sufficiently large $T$, one has
  \begin{equation}
   \label{IO1sumb}
    I_{O}^{(1)}
   =  \sum_{i_1=1}^{3} \sum_{i_2=1}^{3} J_{(i_1,i_2)}^{(1)} 
  + O\Big(  T^{\frac{3 \vartheta}{2}+ \varepsilon} \Big( \frac{T}{T_0} \Big)^{1+C} \Big)
 \end{equation}
 where 
\begin{equation}
\begin{split}  
  \label{J1i1i2id}
 J_{(i_1,i_2)}^{(1)}  &=  \mathcal{Z}_{\I_{\{a_{i_1} \}}, \J_{ \{ b_{i_2} \} }}(0) \cdot
 \int_{-\infty}^{\infty} \omega(t)   \Big( \frac{t}{2 \pi} \Big)^{-a_{i_1}-b_{i_2}}  dt \\
&  -\mathrm{Res}_{s=  \frac{-a_{i_1} -b_{i_2}}{2}} \mathcal{Z}_{\I,\J}(2s) \frac{G \Big( \frac{-a_{i_1} -b_{i_2}}{2}  \Big)}{( \frac{-a_{i_1} -b_{i_2}}{2} )}
  \int_{-\infty}^{\infty} \omega(t) 
  \Big( 
  \frac{t}{2 \pi}
  \Big)^{-\frac{3}{2}(a_{i_1}+b_{i_2})}dt \\
  &+   \sum_{ \substack{ (k_1,k_2) \\
 k_1 \ne i_1, k_2 \ne i_2}} \mathcal{T}_{(i_1,i_2);(k_1,k_2)}
\end{split}
\end{equation}
and 
\begin{equation} 
\begin{split}
  \label{T}
  \mathcal{T}_{(i_1,i_2);(k_1,k_2)} & = 
\frac{1}{2}
  \prod_{j_1 \ne i_1}\zeta(1-a_{i_1}+a_{j_1}) \prod_{j_2 \ne i_2} \zeta(1-b_{i_2}+b_{j_2})  
 \int_{-\infty}^{\infty} \omega(t) \Big( \frac{t}{2 \pi} \Big)^{-a_{i_1}-b_{i_2} -\frac{a_{k_1}+b_{k_2}}{2}} dt   \\
  & \times  \Big(
  \prod_{{\begin{substack}{ l_1 \ne i_1, \l_2 \ne i_2
         \\ (l_1,l_2) \ne (k_1,k_2)}\end{substack}}} 
      \zeta(1+a_{l_1}+b_{l_2}-a_{k_1}-b_{k_2}) \Big)
      \zeta(1-a_{i_1}-b_{i_2}+a_{k_1}+b_{k_2}) \\
   & \times   \mathcal{C}_{\I,\J; \{a_{i_1} \}, \{ b_{i_2} \} }(-\tfrac{a_{k_1}+b_{k_2}}{2})
      \frac{G(-\frac{a_{k_1}+b_{k_2}}{2})}{(-\frac{a_{k_1}+b_{k_2}}{2})}.
\end{split}
\end{equation}
\end{prop}
\noindent {\bf Remark}.  (i) It should be observed that the formulae for $J_{(i_1,i_2)}^{(1)}$ contains
the extra unwanted residues $ \mathcal{T}_{(i_1,i_2);(k_1,k_2)}$.  Notice that these do not
appear in the formula for $I_{O}^{(1)}+I_{O}^{(2)}$ given by \eqref{IOasymp}. \\
(ii)  The proof of this Proposition actually establishes unconditionally that one has $ \sum_{i_1=1}^{3} \sum_{i_2=1}^{3} I_{(i_1,i_2)}^{(1)} 
   =  \sum_{i_1=1}^{3} \sum_{i_2=1}^{3} J_{(i_1,i_2)}^{(1)}  + O(T^{\frac{3}{4}+5 \delta}) $
   when $|a_i|,|b_i| \le \delta < \frac{1}{12}$ for $i=1,2,3$.   \\
(iii) On the Riemann hypothesis it is possible to improve the last error term to $ O(T^{\frac{2}{3}+5 \delta})$. 
This is since we can show $
   \mathcal{C}_{\I,\J; \{ a_1 \} , \{ b_1 \}}(s) 
   =\zeta(2+4s + a_2+b_2+a_3+b_3)^{-1}  \mathcal{D}_{\I,\J; \{ a_1 \} , \{ b_1 \}}(s)$ 
where $ \mathcal{D}_{\I,\J; \{ a_1 \} , \{ b_1 \}}(s)$ is an Euler product which is absolutely convergent in 
$\Re(s) \ge -\frac{1}{3} + 2 \delta$.  \\

We can also prove an analogous result for $I_{O}^{(2)}$.  
Note that by \eqref{IO2} and \eqref{IO1} we see that $I_{O}^{(2)}$  is the same as $I_{O}^{(1)}$,
 except that $\I \to -\J$ and $\J \to -\I$ and there is the additional factor of $X_{\I,\J;t}$.   Thus 
equations \eqref{IOsumc}, \eqref{I2i1i2id}, and \eqref{U} below will be obtained from \eqref{IO1sumb}, \eqref{J1i1i2id},
and \eqref{T} by replacing each $\I$ by $-\J$, each $\J$ by $-\I$,
 and inserting the factor $( \frac{t}{2 \pi} )^{-\sum_{i=1}^{3} (a_i+b_i)} $ which arises from the Stirling 
 approximation for $X_{\I,\J;t}$ as derived in Lemma \ref{Stirling}. 
\begin{prop} \label{I2i1i2}
Let
$a_i,b_i$  satisfy 
$|a_i|, |b_i| \ll (\log T)^{-1}$  for $i=1,2,3$ and the conditions \eqref{initialconditionA}, 
 \eqref{initialconditionB}, and  \eqref{initialconditionC}.  
Assume Conjecture \ref{divconj} holds for some positive $\theta \ge \frac{1}{2}$ and $C$. Then, for any $\varepsilon >0$
and sufficiently large $T$, one has
\begin{equation}
  \label{IOsumc}
  I_{O}^{(2)} = \sum_{i_1=1}^{3} \sum_{i_2=1}^{3} J_{(i_1,i_2)}^{(2)} 
   + O\Big(  T^{\frac{3 \vartheta}{2}+ \e } \Big( \frac{T}{T_0} \Big)^{1+C} \Big)
\end{equation}
where 
\begin{equation}
\begin{split}
  \label{I2i1i2id}
 J_{(i_1,i_2)}^{(2)} & = 
   \mathcal{Z}_{\I_{\I \backslash \{ a_{i_2} \} }, \J_{ \J \backslash \{ b_{i_1} \} }}(0) \cdot
  \int_{-\infty}^{\infty}  \omega(t)   \Big( \frac{t}{2 \pi} \Big)^{-
\sum_{k \ne i_2} a_k - \sum_{k \ne i_1} b_k}  dt \\
  &  -\mathrm{Res}_{s=  \frac{b_{i_1} +a_{i_2}}{2}} \mathcal{Z}_{-\J,-\I}(2s) \frac{G \Big( \frac{b_{i_1} +a_{i_2}}{2}  \Big)}{( \frac{b_{i_1}+a_{i_2}}{2} )}
  \int_{-\infty}^{\infty} \omega(t) 
  \Big( 
  \frac{t}{2 \pi}
  \Big)^{
  -\sum_{k=1}^{3}(a_k+b_k)
  +\frac{3(b_{i_1}+a_{i_2})}{2}}dt  \\
  & +  \sum_{ \substack{ (k_1,k_2) \\
 k_1 \ne i_1, k_2 \ne i_2}} \mathcal{U}_{(i_1,i_2);(k_1,k_2)} 
\end{split}
\end{equation}
and
\begin{equation}
\begin{split}
  \label{U}
  \mathcal{U}_{(i_1,i_2);(k_1,k_2)} & = 
\frac{1}{2}
  \prod_{j_1 \ne i_1}\zeta(1+b_{i_1}-b_{j_1}) \prod_{j_2 \ne i_2} \zeta(1+a_{i_2}-a_{j_2})  
 \int_{-\infty}^{\infty} \omega(t) \Big( \frac{t}{2 \pi} \Big)^{-a_{r(i_2,k_2)}-b_{r(i_1,k_1)} -\frac{b_{k_1}+a_{k_2}}{2}} dt   \\
  & \times  \Big(
  \prod_{{\begin{substack}{l_1 \ne i_1, l_2 \ne i_2
         \\ (l_1,l_2) \ne (k_1,k_2)}\end{substack}}} 
      \zeta(1-b_{l_1}-a_{l_2}+b_{k_1}+a_{k_2}) \Big)
      \zeta(1+b_{i_1}+a_{i_2}-b_{k_1}-a_{k_2}) \\
   & \times   \mathcal{C}_{-\J,-\I; \{-b_{i_1}\}, \{ -a_{i_2} \}}(\tfrac{b_{k_1}+a_{k_2}}{2})
      \frac{G(\frac{b_{k_1}+a_{k_2}}{2})}{(\frac{b_{k_1}+a_{k_2}}{2})}
\end{split}
\end{equation}
where $r(i_1,k_1)$ and $r(i_2,k_2)$  are defined in \eqref{rdefinition}. 
\end{prop}
\noindent {\bf Remark}. Notice that the formulae for $I_{(i_1,i_2)}^{(2)}$ also contain
extra unwanted residues $ \mathcal{U}_{(i_1,i_2);(k_1,k_2)}$ that do not appear in
\eqref{IOasymp}.  Fortunately, we shall establish that the   $\mathcal{T}_{(i_1,i_2);(k_1,k_2)}$
and $ \mathcal{U}_{(i_1,i_2);(k_1,k_2)}$ cancel each other out. \\
\begin{prop}  \label{sumtozero}
We have that
\begin{equation*}
  \sum_{i_1=1}^{3} \sum_{i_2=1}^{3} \sum_{ \substack{ (k_1,k_2) \\
 k_1 \ne i_1, k_2 \ne i_2}} ( \mathcal{T}_{(i_1,i_2);(k_1,k_2)}+\mathcal{U}_{(i_1,i_2);(k_1,k_2)})=0.
\end{equation*}
\end{prop}
Now that we have stated Propositions \ref{I1i1i2}, \ref{I2i1i2}, and \ref{sumtozero} we can 
complete the proof of Proposition \ref{offdiagonal}. 
\begin{proof}[Proof of Proposition \ref{offdiagonal}]
By combining Propositions \ref{I1i1i2}, and \ref{I2i1i2} 
we see that we get exactly the first three terms in \eqref{IOasymp}
plus 
\begin{equation*}
  \label{extraterms}
   \sum_{i_1=1}^{3} \sum_{i_2=1}^{3} \sum_{ \substack{ (k_1,k_2) \\
 k_1 \ne i_1, k_2 \ne i_2}} ( \mathcal{T}_{(i_1,i_2);(k_1,k_2)}+\mathcal{U}_{(i_1,i_2);(k_1,k_2)})
 +O \Big( T^{\frac{3 \vartheta}{2}+ \e} \Big( \frac{T}{T_0} \Big)^{1+C} 
    \Big). 
\end{equation*}
Moreover, Proposition \ref{sumtozero} shows that the sum in \eqref{extraterms} equals 0.  Thus we establish 
\eqref{IOasymp}.
\end{proof}
 
 \begin{proof}[Proof of Proposition \ref{I1i1i2}]
 Let $\delta \in (0,\tfrac{1}{12})$ and let $\e >0$ satisfy 
 \begin{equation}
    \label{deltaepsineq}
    \delta < \e < \tfrac{1}{2} - \delta.
 \end{equation}
 We shall begin by assuming that $ |a_i|, |b_i| \le \delta$ for $i=1,2,3$.
 At the end of the argument we shall assume the more restrictive condition $ |a_i|, |b_i|  \ll (\log T)^{-1}$ 
 for $i=1,2,3$. 
 We  focus on one of the nine terms in the sum \eqref{IO1sum}: the one with $i_1=i_2=1$.  We will obtain 
 the result for other indices just by permuting them appropriately.  
 Recall that in \eqref{R0} we set $R_0 = T^{5}$.  Recall also the definition given in Proposition 
 \ref{IO1prop} for the term $I_{(i_1,i_2)}^{(1)}$ inside the sum on the right hand side of \eqref{IO1sum}.
 Considering the term in the case $i_1=i_2=1$,
we write $I_{(1,1)}^{(1)}= I^{+} + I^{-}$ where $I^{+}$ is the sum over $r >0$ and $I^{-}$ is the sum 
over $r <0$.   Since $c_{\ell}(-r)=c_{\ell}(r)$ (for all $\ell, r \in \mathbb{N}$), we find that 
\begin{align*}
 I^{\pm} & =  \sum_{1 \le r \le R_0} 
     \prod_{j_1 \ne 1} \zeta(1-a_1+a_{j_1})   \prod_{j_2 \ne 1} \zeta(1-b_1+b_{j_2})  
      \sum_{\ell=1}^{\infty} \frac{c_{\ell}(r)G_{\I}(1-a_1,\ell)G_{\J}(1-b_1,\ell)  }{\ell^{2-a_1-b_1}} \\
    & \times \int_{\max(0,\pm r)}^{\infty} x^{-\frac{1}{2}-a_1}(x \mp r)^{-\frac{1}{2}-b_1} 
     \int_{(\varepsilon)} \frac{G(s)}{2 \pi i s} 
     \Big( \frac{1}{\pi^3 x(x \mp r)} \Big)^{s}
     \int_{-\infty}^{\infty} 
     \Big( 1 \mp \frac{r}{x} \Big)^{it} g_{\I,\J}(s,t) \omega(t) dt ds dx
\end{align*}
for both of the two possible fixed choices of opposing signs $\pm, \mp$. 
(in $I^{-}$ we made the variable change $r \to -r$).
We let $K^{+}$ and $K^{-}$ denote the triple integrals appearing in $I^{+}$ and $I^{-}$. 
In $K^{\pm}$ we make the change of variables $x=ry$ to obtain 
\begin{align*}
  & K^{\pm} 
   = r^{-a_1-b_1} \int_{\max(0,\pm1)}^{\infty}
     y^{-\frac{1}{2}-a_1}(y\mp1)^{-\frac{1}{2}-b_1}
     \int_{(\varepsilon)} \frac{G(s)}{2 \pi i s}
     \Big( 
     \frac{1}{\pi^3 r^2 y (y\mp1)}
     \Big)^s
    \int_{-\infty}^{\infty} 
     \Big( \frac{y \mp 1}{y} \Big)^{it}  g_{\I,\J}(s,t) \omega(t)dtdsdy.
\end{align*}
For $\eta = \pm 1$, we have 
\begin{equation}
  \label{Bsteta}
  \mathcal{B}_{s,t}(\eta) := \int_{\max(0,\eta)}^{\infty}
   y^{-\frac{1}{2}-a_1-s-it} (y -\eta)^{-\frac{1}{2}-b_1-s+it} dy
   = \begin{cases}
   B(\frac{1}{2}-b_1-s+it, a_1+b_1+2s) & \text{ if } \eta =1, \\
   B(\frac{1}{2}-a_1-s-it, a_1+b_1+2s) & \text{ if } \eta =-1, 
   \end{cases}
\end{equation}
where, by Euler's beta function identity 
$B(u,v) = \int_{0}^{\infty} y^{u-1}(1+y)^{-u-v}dy=\frac{\Gamma(u)\Gamma(v)}{\Gamma(u+v)}$ for $\Re(u),\Re(v) >0$.
Observe that the convergence regions for the  $\mathcal{B}_{s,t}(\eta)$ integrals in \eqref{Bsteta} are well-defined only if 
\begin{equation}
  -\frac{1}{2} (\Re(a_1)+\Re(b_1)) < \e =\Re(s) < \frac{1}{2}+ \max(  -\Re(a_1),-\Re(b_1)). 
\end{equation}
However, \eqref{deltaepsineq} implies that this condition holds.  
Hence, by a change in order of integration, we find that 
\[
  K^{\pm} = r^{-a_1-b_1} \int_{-\infty}^{\infty} \omega(t) \int_{(\varepsilon)} 
  \frac{G(s)}{2 \pi i s} \Big( \frac{1}{\pi^3 r^2} \Big)^{s}
  g_{\I,\J}(s,t)  \mathcal{B}_{s,t}(\pm 1) ds dt.  
\]
Recalling that $K^{\pm}$ is the triple integral appearing in the sum $I^{\pm}$, above, and that 
$I_{(1,1)}^{(1)}= I^{+} + I^{-}$, we therefore find (with the help of the above Beta function identity) 
that  
\begin{equation}
\begin{split}
  \label{I111}
 & I_{(1,1)}^{(1)} = \prod_{j_1 \ne 1} \zeta(1-a_1+a_{j_1})   \prod_{j_2 \ne 1} \zeta(1-b_1+b_{j_2})  
     \sum_{1 \le r \le R_0}
      \sum_{\ell=1}^{\infty} \frac{c_{\ell}(r)G_{\I}(1-a_1,\ell)G_{\J}(1-b_1,\ell)  }{\ell^{2-a_1-b_1} r^{a_1+b_1}} \\
    & \int_{-\infty}^{\infty} \omega(t)
     \frac{1}{2 \pi i} \int_{(\varepsilon)} \frac{G(s)}{s} g_{\I,\J}(s,t)
     \Big( 
     \frac{1}{\pi^3 r^2}
     \Big)^s  \Gamma(a_1+b_1+2s)
  \Bigg( 
   \frac{\Gamma(\tfrac{1}{2}-b_1-s+it)}{\Gamma(\tfrac{1}{2}+a_1+s+it)}+
  \frac{\Gamma(\tfrac{1}{2}-a_1-s-it)}{\Gamma(\tfrac{1}{2}+b_1+s-it)}
\Bigg)
          ds \, dt.
\end{split}
\end{equation}

We now move the $s$ integral right to the line $\Re(s)=1$;
this is in order that we may later be in  a position to apply Proposition \ref{Hid} (i).
Observe that the term in the brackets has simple poles at 
$p_1 = \frac{1}{2}-b_1+it$  and $p_2 = \frac{1}{2} -a_1-it$.    
Using $\Gamma(z) \sim z^{-1}$ as $z \to 0$, 
we find that the residue at $s=p_j$ is
$-\frac{G(p_j)}{p_j} g_{\I,\J}(p_j,t)
     ( \pi^3 r^2)^{-p_j}$ for $j=1,2$.    
Using Lemma \ref{Stirling} (ii), we find that when $c_1T \le t \le c_2 T$ the contribution of each 
residue to the $t$ integrand in \eqref{I111} is
\[
   \frac{G(p_j)}{p_j} g_{\I,\J}(p_j,t)
        ( \pi^3 r^2)^{-p_j} 
      \ll    \frac{|G(p_j)|}{t} t^{\frac{5}{2}}  r^{-1+2 \delta} 
      \ll |G(p_j)| t^{\frac{3}{2}} r^{-1+2 \delta} 
\]
(it being assumed here that $r$ satisfies the conditions of summation in \eqref{I111} where 
$R_0 = T^{5}$).  For $j=1,2$, we have both $|\Re(p_j) - \frac{1}{2}| \le \delta$ and $|p_j| \asymp t$, 
and so $G(p_j) \ll t^{-B}$.   
Hence the total contribution of these residues to \eqref{I111} is 
\begin{equation*}
\begin{split}
  \sum_{1 \le r \le R_0} \frac{ {\bf c}_{1,1}}{r^{a_1+b_1}}
  \int_{c_1 T}^{c_2 T}  O( t^{\frac{3}{2}-B} r^{-1+2 \delta}) dt \cdot \sum_{\ell=1}^{\infty} u_{\ell,r;1,1}
  & = \sum_{1 \le r \le T^{5}} 
  O( (\log T)^4 r^{4 \delta} T^{\frac{5}{2}-B} r^{-1} d_2(r) ) \\
  & \ll (\log T)^4  T^{\frac{5}{2}-B} \sum_{1 \le r \le T^{5}} \frac{d_2(r) r^{4 \delta}}{r}
  \ll  T^{\frac{5}{2}+20 \delta -B}  (\log T)^6
\end{split}
\end{equation*}
where we have used the bounds
\eqref{ci1i2bd} and  \eqref{ulrbound}.  Since $\delta < \frac{1}{12}$ we may choose  $B=5$ 
to obtain 
\begin{equation}
\begin{split}
  \label{I111b}
 & I_{(1,1)}^{(1)} = \prod_{j_1 \ne 1} \zeta(1-a_1+a_{j_1})   \prod_{j_2 \ne 1} \zeta(1-b_1+b_{j_2})  
     \sum_{1 \le r \le R_0}
      \sum_{\ell=1}^{\infty} \frac{c_{\ell}(r)G_{\I}(1-a_1,\ell)G_{\J}(1-b_1,\ell)  }{\ell^{2-a_1-b_1} r^{a_1+b_1}} \\
    & \int_{-\infty}^{\infty} \omega(t)
     \frac{1}{2 \pi i} \int_{(1)} \frac{G(s)}{s} g_{\I,\J}(s,t)
     \Big( 
     \frac{1}{\pi^3 r^2}
     \Big)^s  \Gamma(a_1+b_1+2s)
  \Bigg( 
   \frac{\Gamma(\tfrac{1}{2}-b_1-s+it)}{\Gamma(\tfrac{1}{2}+a_1+s+it)}+
  \frac{\Gamma(\tfrac{1}{2}-a_1-s-it)}{\Gamma(\tfrac{1}{2}+b_1+s-it)}
\Bigg)
          ds \, dt \\
          & + O(1 ).
\end{split}
\end{equation}  

At this point we will add back in those terms with $r > R_0$.  However, we require the following consequence of Stirling's formula.
Let $\Re(s) \in  [2 \delta,0.49] \cup [0.51,1] $, $|\Im(s)| \le t+1$,  and $|a_1|, |b_1| \le \delta$
where $\delta \le \frac{1}{6}$.  Then 
\begin{equation}
  \label{stircritical}
  \frac{\Gamma(p_j-s)}{\Gamma(a_1+b_1+p_j+s)} 
  = t^{-a_1-b_1-2s}
  \exp 
  ( (-1)^j \tfrac{\pi i}{2}(a_1+b_1+2s)
  )
  (1 + O ( \tfrac{1+|s|^2}{t^{1-2 \delta}} ) )
  \text{ for } j=1,2. 
\end{equation}
The proof of  \eqref{stircritical} is very similar to the proof of Lemma \ref{Stirling} (ii) and is left as an exercise.
We deduce
\begin{equation}
   \label{gammastirlingB}
   \sum_{j=1}^{2} \frac{\Gamma(p_j-s)}{\Gamma(a_1+b_1+p_j+s)} 
  = 2t^{-a_1-b_1-2s}  \cos 
  ( \tfrac{\pi }{2}(a_1+b_1+2s)) 
  +O \Big(  t^{-2\Re(s)}   e^{\pi |\Im(s)|}  (  \tfrac{1+|s|^2}{t^{1-4\delta}} )
  \Big)
\end{equation}
for $\Re(s) \in  [2 \delta,0.49] \cup [0.51,1] $ and $|\Im(s)| \le t+1$.  
Note that we have a weak form of Stirling's formula, namely
\begin{equation}
  \label{gammazbd}
|\Gamma(x+iy)| \asymp |y|^{x-\frac{1}{2}} e^{-\frac{\pi}{2}|y|} \text{ for } |y| \ge \frac{1}{2}, -4 \le x \le 4.
\end{equation}
This implies that, for $j=1,2$, we have 
\begin{equation}
  \label{gammabound}
  \frac{\Gamma(p_j-s)}{\Gamma(a_1+b_1+p_j+s)}  
  \ll \frac{\exp((-1)^{j-1} \text{sgn}(\Im(s)) \pi t  ) }{(|\Im(s)|^2-t^2)^{\Re(s)-\delta }}
  \text{  for } |\Im(s)| \ge t+1, -1 \le \Re(s) \le 2. 
 \end{equation}
Again by \eqref{gammazbd} along with a trivial estimate for $|\Im(s)| \le 2$ we obtain 
\begin{equation}
  \label{gammaoneline}
 |\Gamma(a_1+b_1+2s)| \ll |s|^{\frac{3}{2}+2 \delta} e^{-\pi |\Im(s)|} \text{ for } 
 \Re(s)=1.
\end{equation}
We also require the bound
\begin{equation}
  \label{coszbd}
|\cos(\tfrac{\pi}{2}z)| \ll e^{\frac{\pi}{2} |\Im(z)|}  \text{ for }  z  \in \mathbb{C}.
\end{equation}
Combining \eqref{gammastirlingB}, \eqref{gammabound}, \eqref{coszbd},  \eqref{sizerestrictiondelta}, and the assumption $1 \ll c_1 T \le t \le c_2 T$, it follows that, when 
$\Re(s)=1$ one has 
\[
  \sum_{j=1}^{2}   \frac{\Gamma(p_j-s)}{\Gamma(a_1+b_1+p_j+s)} 
  \ll 
  \begin{cases}
  t^{2 \delta-2} e^{\pi |\Im(s)|} (1 + \frac{|s|^2}{t^{1-2 \delta}}) & \text{if } |\Im(s)| \le t+1, \\
  (| |\Im(s)|^2 -t^2)^{\delta-1}  e^{\pi t} & \text{otherwise.}
  \end{cases}
\]
It follows  that for all $s$ satisfying $\Re(s)=1$ and $t \ge c_1 T \gg 1$
\begin{equation}
   \label{gammafoneline}    \sum_{j=1}^{2}   \frac{\Gamma(p_j-s)}{\Gamma(a_1+b_1+p_j+s)} 
  \ll t^{2\delta-2} |s|^{2+2 \delta} e^{\pi |\Im(s)|}.
\end{equation}
Using \eqref{gammaoneline}, \eqref{gammafoneline}, and Lemma \ref{Stirling} (ii), one finds that the absolute value
of the $s$ integral in \eqref{I111b} is bounded above by:
\[
\int_{(1)} \frac{|G(s)|}{|s|}
\cdot \frac{t^3}{8 \pi^3 r^2} \Big(  1+ O \Big( \frac{|s|^2}{t^{1-3 \delta}} \Big) \Big)
\cdot O( t^{2 \delta-2} |s|^{\frac{7}{2}+4 \delta} ) \cdot |ds|
\ll r^{-2} t^{1+2\delta} \int_{(1)} |G(s)| |s|^{\frac{9}{2}+4 \delta}
|ds| \ll r^{-2} t^{1+2\delta}. 
\]
Applying this along with \eqref{ci1i2bd} and \eqref{ulrbound},  it follows  that  
removal of the condition of summation
$r \le R_0=T^{5}$ would add to the right-hand side of 
\eqref{I111b} a sum equal to
\begin{equation}
\begin{split}
  \label{tailR0}
  \sum_{r > R_0} \frac{{\bf c}_{1,1}}{r^{a_1+b_1}}
  \int_{c_1 T}^{c_2 T} O(r^{-2} t^{1+2\delta}) dt \cdot \sum_{\ell \ge 1} u_{\ell,r;1,1}
  & = \sum_{r > T^{5}} O ( (\log T)^4 r^{2 \delta-2} T^{2+2\delta} d_2(r)) \\
  & \ll (\log T)^4 T^{2+2\delta} \sum_{r > T^{5}} \frac{d_2(r)}{r^{2-2 \delta}} 
  < \zeta^2 (\tfrac{3}{2}-2 \delta) T^{2 +3 \delta- \frac{5}{2}}   \ll1 
\end{split}
\end{equation}
since $\delta < \frac{1}{12}$.
Since  \eqref{I111b} already includes an error term $O(1)$, it follows from \eqref{tailR0} that, on
the right-hand side of the equation  \eqref{I111b}, we may simply omit the condition $r \le R_0$
(this changes the equation, but it does not lose its validity).  Following this step, the variable of 
summation in  \eqref{I111b}  is free to range over all positive integer values.  Then, by absolute
convergence,
we may swap the order of summation and integration there, and so obtain 
\begin{align*}
 & I_{(1,1)}^{(1)} 
    =      \prod_{j_1 \ne 1} \zeta(1-a_1+a_{j_1})   \prod_{j_2 \ne 1} \zeta(1-b_1+b_{j_2})  
    \int_{-\infty}^{\infty} \omega(t) \\
    & 
    \int_{(1)} \frac{G(s)}{2 \pi i s} g_{\I,\J}(s,t)   H_{\I,\J; \{a_1 \}, \{ b_1 \}}(s)
   \pi^{-3s}  \Gamma(a_1+b_1+2s)
  \Bigg( 
   \frac{\Gamma(\tfrac{1}{2}-b_1-s+it)}{\Gamma(\tfrac{1}{2}+a_1+s+it)}+
  \frac{\Gamma(\tfrac{1}{2}-a_1-s-it)}{\Gamma(\tfrac{1}{2}+b_1+s-it)}
\Bigg)
          ds \, dt +O(1)
\end{align*}
where we recall $H_{\I,\J; \{a_1 \}, \{ b_1 \}}(s)$ is defined in \eqref{HIJa1b1}.

We know that $H_{\I,\J; \{a_1 \}, \{ b_1 \}}(s)$ has an analytic continuation to $\Re(s) \ge -\frac{1}{4}+2\delta$
with the exception of the poles \eqref{Hpoles}.
We now move this contour back to $\Re(s) =\varepsilon_1$ where we now set 
$ \varepsilon_1
 =2 \delta$.  Note that the poles $-\frac{a_{k_1}+b_{k_2}}{2}$ for $k_1 \ne 1$, $k_2 \ne 1$ lie in the region $\Re(s) \le \delta$.  
 Throughout the rest of this section we shall use this value of $\varepsilon_1$. 
 In moving the contour we  pass poles at $s =p_1 $ and $s=p_2$.  
Observe that the pole $\frac{1}{2}-\frac{a_{1}+b_{1}}{2}$ of $H_{\I,\J; \{ a_1 \}, \{ b_1 \}}(s)$  
is cancelled by a corresponding zero of  $ \frac{\Gamma(\tfrac{1}{2}-b_1-s+it)}{\Gamma(\tfrac{1}{2}+a_1+s+it)}+
  \frac{\Gamma(\tfrac{1}{2}-a_1-s-it)}{\Gamma(\tfrac{1}{2}+b_1+s-it)}$.
By Lemma \ref{Stirling} (ii) and Lemma \ref{Hbounds} (ii), we find that  for $j=1,2$  the residue
at $s=p_j$ is
\[
  -\frac{G(p_j)}{p_j} g_{\I,\J}(p_j,t)   H_{\I,\J; \{ a_1\},\{ b_1 \}}(p_j) \pi^{-3 p_j} 
  \ll |p_j|^{-1-B} 
  \cdot \Big( \frac{t}{2} \Big)^{3\Re(p_j)} \cdot \frac{|p_j|^{2}}{t^{1-3 \delta}} 
  \cdot T^{4 \delta} 
 \ll T^{\frac{3}{2}-B+10 \delta}  
\]
since  we have \eqref{sizerestrictiondelta},  $p_j \asymp T$,  $c_1T \le t \le c_2T$.
Thus the poles contribute $(\log T)^{4} T^{\frac{5}{2}-B+10 \delta}  \ll 1$  by choosing $B$ sufficiently large. 
It follows that 
\begin{equation}
\begin{split}
  \label{I11gammafactors}
 & I_{(1,1)}^{(1)} = \prod_{j_1 \ne 1} \zeta(1-a_1+a_{j_1})   \prod_{j_2 \ne 1} \zeta(1-b_1+b_{j_2})   
        \int_{-\infty}^{\infty} \omega(t) \\
     & \frac{1}{2 \pi i} \int_{(\varepsilon_1)} \frac{G(s)}{s} g_{\I,\J}(s,t)  H_{\I,\J; \{a_1 \}, \{ b_1 \}}(s)
     \pi^{-3s}  \Gamma(a_1+b_1+2s)
  \Bigg( 
   \frac{\Gamma(\tfrac{1}{2}-b_1-s+it)}{\Gamma(\tfrac{1}{2}+a_1+s+it)}+
  \frac{\Gamma(\tfrac{1}{2}-a_1-s-it)}{\Gamma(\tfrac{1}{2}+b_1+s-it)}
\Bigg)
          ds \, dt +O(1).
\end{split}
\end{equation}

We give a bound for the portion of the inner integral with $|\Im(s)| \ge t+1$.  
We see, using Lemma \ref{Stirling} (ii), \eqref{gammazbd}, and \eqref{gammabound}, that it is bounded by 
\begin{equation}
\begin{split}
  \label{innerintegral}
&  \int_{\substack{\Re(s) = \varepsilon_1 \\ |\Im(s)| \ge t+1}  }
 \frac{|G(s)|}{|s|} 
 \cdot O \Big(t^{3 \varepsilon_1} \cdot \frac{|s|^2}{t^{1-3 \delta}}  \Big)
 |H_{\I,\J; \{a_1 \}, \{ b_1 \}}(s)| \cdot |s|^{2 \varepsilon_1 + 2 \delta- \frac{1}{2}} e^{-\pi |\Im(s)|} 
 \cdot |s|^{2 \delta} e^{\pi t} |ds| \\
 & \ll t^{3 \varepsilon_1+3 \delta-1} \int_{(\varepsilon_1)} |G(s) H_{\I,\J; \{a_1 \}, \{ b_1 \}}(s)| \cdot 
 |s|^{\frac{1}{2} + 4 \delta +2 \varepsilon_1} |ds|.
\end{split}
\end{equation}
Further, by   Lemma \ref{Hbounds} (i) we find \eqref{innerintegral} is 
\[
  \ll t^{9 \delta-1} \int_{(2 \delta)} |G(s)| |s|^{1+8 \delta} |ds| \ll t^{9 \delta-1}, 
\]
since  $\varepsilon_1=2 \delta$.
It follows that the contribution to \eqref{I11gammafactors} arising from $|\Im(s)| \ge t+1$ is 
$\ll {\bf c}_{1,1} \int_{c_1T}^{c_2 T} |\omega(t)| t^{9 \delta-1} dt 
\ll (\log T)^4 T^{9 \delta} \ll T^{10 \delta}$.
By  Lemma \ref{Stirling} (ii)  and \eqref{gammastirlingB} it follows that for $|\Im(s)| \le t+1$
\begin{align*}
   &g_{\I,\J}(s,t)   \sum_{j=1}^{2}   \frac{\Gamma(p_j-s)}{\Gamma(a_1+b_1+p_j+s)} \\
   & =  \Big(2t^{-a_1-b_1-2s}  \cos 
  ( \tfrac{\pi }{2}(a_1+b_1+2s)) 
  +O \Big(  t^{-2\Re(s)}   e^{\pi |\Im(s)|}  (  \tfrac{1+|s|^2}{t^{1-4\delta}} )
  \Big) \Big) \Big(1+O \Big(
\tfrac{1+|s|^2}{t^{1-3 \delta}}  \Big)  \Big)  \Big( \frac{t}{2} \Big)^{3s} \\
& = \Big(2t^{-a_1-b_1-2s}  \cos 
  ( \tfrac{\pi }{2}(a_1+b_1+2s)) 
  +O \Big( t^{-2 \Re(s)} e^{\pi |\Im (s)|} 
  \Big(    \tfrac{1+|s|^2}{t^{1-5\delta}}
  +   \tfrac{(1+|s|^2)^2}{t^{2-7\delta}} \Big) \Big) \Big)  \Big( \frac{t}{2} \Big)^{3s} .
\end{align*}
By this, and the bound obtained just before, 
we find that 
\begin{equation}
\begin{split}
  \label{I111approx}
 & I_{(1,1)}^{(1)} = \prod_{j_1 \ne 1} \zeta(1-a_1+a_{j_1})   \prod_{j_2 \ne 1} \zeta(1-b_1+b_{j_2})  
 \int_{c_1T}^{c_2 T} \omega(t)  
     \int_{\varepsilon_1-i(t+1)}^{\varepsilon_1+i(t+1)}  H_{\I,\J; \{ a_1 \}, \{ b_1 \}}(s)   \frac{G(s)}{2 \pi i s}  \Big( \frac{t}{2 \pi} \Big)^{3s}  \Gamma(a_1+b_1+2s) \\
     &   \times \Big(
        2t^{-a_1-b_1-2s}  \cos 
  ( \tfrac{\pi }{2}(a_1+b_1+2s)) 
  + O \Big( t^{-2 \Re(s)} e^{\pi |\Im (s)|} 
  \Big(    \tfrac{1+|s|^2}{t^{1-5\delta}}
  +   \tfrac{(1+|s|^2)^2}{t^{2-7\delta}} \Big)\Big) \Big)
               ds \,  dt  + O(T^{10 \delta}).
\end{split}
\end{equation}
We now have to estimate the the contribution from the $O$ term in \eqref{I111approx}.
By Lemma \ref{Hbounds} (i) and \eqref{gammaoneline}  we find this term contributes
\begin{equation}
\begin{split}
  \label{contr1}
 O & (( \log T)^4) \int_{c_1 T}^{c_2 T} |\omega(t)| \int_{\e_1-(t+1)}^{\e_1+(t+1)} 
 O_{\varepsilon_1}( |s|^{\frac{1}{2}})  \cdot \frac{|G(s)|}{|s|} t^{\varepsilon_1}  \cdot 
 O_{\varepsilon_1} ( |s|^{2 \varepsilon_1+2 \delta-\frac{1}{2}})
 \cdot O \Big(\frac{1+|s|^2}{t^{1-5 \delta}}+ \frac{|s|^4}{t^{2-7 \delta}} \Big) |ds| dt \\
 & \ll (\log T)^4 \int_{c_1 T}^{c_2 T} |\omega(t)| t^{\varepsilon_1+5 \delta-1} \cdot O_{\varepsilon_1}(1) dt 
 \ll_{\varepsilon_1} (\log T)^4 T^{7\delta} \ll T^{8 \delta}
\end{split}
\end{equation}
where the last $O_{\varepsilon_1}(1)$ arises from 
\[
   \int_{(\varepsilon_1)} |G(s)| \cdot |s|^{2\varepsilon_1+2 \delta-1} (1+|s|)^4 |ds| 
   \ll \varepsilon_1^{-4}    
    \int_{(\varepsilon_1)} |G(s)|  |s|^{6 \delta+3} |ds| < \infty. 
\]
In the  part of the integral in \eqref{I111approx} corresponding to the term  $2t^{-a_1-b_1-2s}  \cos 
  ( \tfrac{\pi }{2}(a_1+b_1+2s)) $, we would like to extend the integration 
to all of  $\Re(s) = \varepsilon_1$. 
By \eqref{gammazbd} and \eqref{coszbd}  it follows that 
\begin{equation}
  \label{gammacos}
| \Gamma(a_1+b_1+2s) \cos 
  ( \tfrac{\pi }{2}(a_1+b_1+2s))| \ll_{\delta} |u|^{2\varepsilon_1+2 \delta-\frac{1}{2}}
\end{equation}
where $s=\varepsilon_1+iu$. 
By  Lemma \ref{Hbounds} (i)  and \eqref{gammacos}, the part of the integral with $|\Im(s)| \ge t+1$
is bounded by  
\begin{equation}
\begin{split}
  \label{contr2}
 & (\log T)^4 \int_{c_1T}^{c_2 T} |\omega(t)| 
 \Big( \int_{\e_1 +i (t+1)}^{\e_1 +i\infty}  
 +
  \int_{\e_1 -i\infty}^{\e_1 -i (t+1)}
  \Big)    O_{\varepsilon_1}( |s|^{\frac{1}{2}})   \cdot
  \frac{|G(s)|}{|s|} t^{\varepsilon_1+2 \delta}  \cdot 
 O_{\varepsilon_1} ( |s|^{2 \varepsilon_1+2 \delta-\frac{1}{2}}) |ds| dt \\
 & \ll  (\log T)^4  \int_{c_1T}^{c_2 T} |\omega(t)|  t^{4 \delta} \int_{\e_1+i(t+1)}^{\e_1+i\infty} |s|^{-3} |ds| 
 \ll 1,
\end{split}
\end{equation}
since $\delta < \frac{1}{12}$.
By the above comments we may extend the integration in \eqref{I111approx} to all of $\Re(s) = \varepsilon_1$,
and may also
remove the  terms corresponding to the big $O$ term in  \eqref{I111approx}.
By Proposition \ref{Hid} and  the functional equation
$
 \zeta(1-z)=2^{1-z} \pi^{-z}  \cos(\frac{\pi z}{2}) \Gamma(z) \zeta(z)$, one has
\begin{equation*}
\begin{split}
  2 & \cos( \tfrac{\pi}{2}(a_1+b_1+2s)) \Gamma(a_1+b_1+2s) H_{\I,\J; \{a_1 \}, \{ b_1 \}}(s) \\
  & = 2  \cos( \tfrac{\pi}{2}(a_1+b_1+2s)) \Gamma(a_1+b_1+2s)  
  \zeta(a_1+b_1+2s)\mathcal{C}_{\I,\J;a_1,b_1}(s)
  \prod_{\substack{k_1 \ne 1   \\ k_2 \ne 1}} \zeta(1+a_{k_1}+b_{k_2}+2s) \\
  & =  (2 \pi)^{a_1+b_1+2s} \zeta(1-a_1-b_1-2s) \mathcal{C}_{\I,\J;a_1,b_1}(s) 
   \prod_{\substack{k_1 \ne 1   \\ k_2 \ne 1}} \zeta(1+a_{k_1}+b_{k_2}+2s).
\end{split}
\end{equation*}
Inserting this in the extended form of the integral in \eqref{I111approx}  (the version of \eqref{I111approx} where the integration
is over $\Re(s) = \varepsilon_1$) while taking into account the bounds
\eqref{contr1} and \eqref{contr2} we find

\begin{align*}
 & I_{(1,1)}^{(1)}  = \prod_{j_1 \ne 1} \zeta(1-a_1+a_{j_1})    \prod_{j_2 \ne 1} \zeta(1-b_1+b_{j_2})  
 \int_{-\infty}^{\infty} \omega(t)  
     \frac{1}{2 \pi i} \int_{(\varepsilon_1)}    \zeta(1-a_1-b_1-2s)  \times \\
& \Big( \prod_{\substack{(k_1,k_2) \\ k_1 \ne 1, k_2 \ne 1}}  
    \zeta(1+a_{k_1}+b_{k_2}+2s) \Big)   \mathcal{C}_{\I,\J;a_1,b_1}(s)
      \frac{G(s)}{s}    \Big( \frac{t}{2 \pi} \Big)^{s-a_1-b_1}
          ds \, dt    + O(T^{10 \delta}).
\end{align*}
This simplifies to 
 \begin{align*}
     I_{(1,1)}^{(1)}  = \prod_{j_1 \ne 1}  \zeta(1-a_1+a_{j_1})  \prod_{j_2 \ne 1}
       \zeta(1-b_1+b_{j_2}) 
 \int_{-\infty}^{\infty} \omega(t) \Big( \frac{t}{2 \pi} \Big)^{-a_1-b_1}  
     \frac{1}{2 \pi i} \int_{(\varepsilon_1)}   \varphi(s)  \Big( \frac{t}{2 \pi} \Big)^{s} ds \, dt
     + O( T^{10 \delta} )
\end{align*}
where
\begin{equation*}
  \varphi(s) =  \Big(
     \prod_{\substack{(k_1,k_2) \\ k_1  \ne 1, k_2 \ne 1}} \zeta(1+a_{k_1}+b_{k_2}+2s)
      \Big)
      \zeta(1-a_1-b_1-2s)
      \mathcal{C}_{\I,\J;a_1,b_1}(s)
      \frac{G(s)}{s}.
 \end{equation*}
 
 We further evaluate $I_{(1,1)}^{(1)}$ by applying the residue theorem.  
 Note that 
 $\varphi(s)$  has poles at $s=0$, $s= -\frac{a_1+b_1}{2}$, and 
$-\frac{a_{k_1}+b_{k_2}}{2}$ for $k_1 \ne 1$ and $k_2 \ne 1$.
Observe that conditions \eqref{initialconditionA}, \eqref{initialconditionB},
and \eqref{initialconditionC} guarantee that these are distinct poles. 
Also note that the condition $\delta < \frac{1}{12}$ ensures that all of these poles are to the right 
of the line $\Re(s) = -\frac{1}{4}+2 \delta$. 
We move the $s$ contour left past $\Re(s)=0$, picking up residues 
at the various poles.   Let 
\begin{align*}
 \mathcal{R}_1(t) & =\text{Res}_{s=0} \Big(\varphi(s)  \Big( \frac{t}{2 \pi} \Big)^{s} \Big),  \\
 \mathcal{R}_2(t) & =\text{Res}_{s=-\frac{a_1+b_1}{2}} \Big( \varphi(s)  \Big( \frac{t}{2 \pi} \Big)^{s} \Big),  \\
  \mathcal{R}_3(t) & = \sum_{\substack{k_1 \ne 1 \\ k_2 \ne 1}} \text{Res}_{s=-\frac{a_{k_1}+b_{k_2}}{2}} \Big(\varphi(s) \Big( \frac{t}{2 \pi} \Big)^{s} \Big).
\end{align*}
By the residue theorem,
\[
      \frac{1}{2 \pi i} \int_{(\varepsilon_1)}   \varphi(s)  \Big( \frac{t}{2 \pi} \Big)^{s} ds = \mathcal{R}_1(t) + \mathcal{R}_2(t)+ \mathcal{R}_3(t)
      +   \frac{1}{2 \pi i} \int_{(-\frac{1}{4}+2 \delta)}   \varphi(s)   \Big( \frac{t}{2 \pi} \Big)^{s} ds.
\]
Observe that for $s=-\frac{1}{4}+2 \delta+iu$ the product of all the zeta factors in $\varphi(s)$ is of size 
$O_{\delta}\Big( \Big( (|u|+1)^{\frac{1}{4}-\delta} \log(2+|u|) \Big)^4 \Big) \ll_{\delta} |u|+1$ which follows from \eqref{zetabounds}, and that, by Proposition \ref{Hid} (ii), we have    
 $\mathcal{C}_{\I,\J;a_1,b_1}(s) \ll O(1)$ when $\Re(s)=-\frac{1}{2}+4 \delta$.  Therefore
\[
     \frac{1}{2 \pi i} \int_{(-\frac{1}{4}+2 \delta)}   \varphi(s)  \Big( \frac{t}{2 \pi} \Big)^{s} ds \ll 
     t^{-\frac{1}{4}+2 \delta} \int_{-\infty}^{\infty} (|u|+1) |-\tfrac{1}{4}+2 \delta+iu|^{-1} \min(1,|u|^{-A}) du \ll 
     T^{-\frac{1}{4}+2 \delta}
\]
and
\[
     \int_{-\infty}^{\infty} \omega(t) \Big( \frac{t}{2 \pi} \Big)^{-a_1-b_1}  
     \frac{1}{2 \pi i} \int_{(-\frac{1}{4}+2 \delta)}   \varphi(s)   \Big( \frac{t}{2 \pi} \Big)^{s} ds \, dt \ll 
     T^{1+2\delta - \frac{1}{4} + 2 \delta} = T^{\frac{3}{4} + 4 \delta}. 
\]
In addition, as all the poles are simple we have the following residues:
\begin{align*}
     \mathcal{R}_1(t)= \Big( \prod_{\substack{(k_1,k_2) \\ k_1 \ne 1, k_2 \ne 1}} \zeta(1+a_{k_1}+b_{k_2}) \Big)
      \zeta(1-a_1-b_1)
      \cC_{\I,\J;a_1,b_1}(0),
\end{align*}
\begin{align*}
    \mathcal{R}_2(t)=-\frac{1}{2}\Big( \prod_{\substack{ (k_1,k_2) \\ k_1 \ne 1, k_2 \ne 1}}
    \zeta(1+a_{k_1}+b_{k_2}-a_1-b_1) \Big)
         \cC_{\I,\J;a_1,b_1}( -\tfrac{a_1+b_1}{2})
      \frac{G( -\frac{a_1+b_1}{2})}{ (-\frac{a_1+b_1}{2})} 
       \Big( \frac{t}{2 \pi} \Big)^{ -\frac{a_1+b_1}{2}}, 
\end{align*}
and 
\begin{align*}
    \mathcal{R}_3(t) &  =   \frac{1}{2}\sum_{\substack{k_1 \ne 1 \\ k_2 \ne 1}}
 \prod_{{\begin{substack}{l_1 \ne 1, l_2 \ne 1
         \\ (l_1,l_2) \ne (k_1,k_2)}\end{substack}}} 
      \zeta(1+a_{l_1}+b_{l_2}-a_{k_1}-b_{k_2})
      \zeta(1-a_1-b_1+a_{k_1}+b_{k_2})
      \cC_{\I,\J; \{ a_1 \} , \{ b_1 \}}(-\tfrac{a_{k_1}+b_{k_2}}{2}) \\
      & \times
      \frac{G(-\frac{a_{k_1}+b_{k_2}}{2})}{(-\frac{a_{k_1}+b_{k_2}}{2})}
       \Big( \frac{t}{2 \pi} \Big)^{ 
      -\frac{a_{k_1}+b_{k_2}}{2}}.
\end{align*}
Observe that the coefficients $-\frac{1}{2}$ and $\frac{1}{2}$ in front of $\mathcal{R}_2(t)$ and 
$\mathcal{R}_3(t)$  appear since $\zeta(1+\eta \pm 2s) = \frac{ \pm \frac{1}{2}}{s \pm \frac{\eta}{2}}+O(1)$
for $|s| \le 1$ and $|\eta| \le 2 \delta$.  
Since $T^{10 \delta} \ll T^{\frac{3}{4}+\delta}$ (given that $\delta < \frac{1}{12}$), it follows that 
\begin{equation}
  \label{I11expansion}
 I_{1,1}^{(1)}  = \mathcal{S}_1+\mathcal{S}_2+\mathcal{S}_3 + O(
(\log T)^4 T^{\frac{3}{4} +  \delta}),
\end{equation}
where
\begin{equation}
\begin{split}
  \label{S1}
 \mathcal{S}_1 
= \int_{-\infty}^{\infty} \omega(t) \Big( \frac{t}{2 \pi} \Big)^{-a_1-b_1} dt 
\cdot  \cC_{\I,\J; \{ a_1 \} , \{ b_1 \} }(0)
\cdot \Bigg( \prod_{\substack{(a,b) \\ a \in \I_{\{a_1 \}}, b \in  \J_{ \{ b_1 \} } }}
\zeta(1+a+b)  \Bigg),
\end{split}
\end{equation}
\begin{equation}
\begin{split}
  \label{S2}
 \mathcal{S}_2 & =  -\frac{1}{2}\prod_{j_1 \ne 1}  \zeta(1-a_1+a_{j_1})  \prod_{j_2 \ne 1} \zeta(1-b_1+b_{j_2}) 
 \prod_{\substack{(k_1,k_2) \\ k_1 \ne 1, k_2 \ne 1}} \zeta(1+a_{k_1}+b_{k_2}-a_1-b_1) \times \\
  &    \cC_{\I,\J; \{ a_1 \}, \{b_1 \}}( -\tfrac{a_1+b_1}{2})
      \frac{G( -\frac{a_1+b_1}{2})}{ (-\frac{a_1+b_1}{2})}    
    \int_{-\infty}^{\infty} \omega(t) \Big( \frac{t}{2 \pi} \Big)^{-\frac{3(a_1+b_1)}{2}} 
    dt,
\end{split}
\end{equation}
\begin{equation}
\begin{split}
 \label{S3}
 & \mathcal{S}_3  =  \frac{1}{2}
  \prod_{j_1 \ne 1}\zeta(1-a_1+a_{j_1}) \prod_{j_2 \ne 1} \zeta(1-b_1+b_{j_2})  
   \sum_{ \substack{ (k_1,k_2) \\
 k_1 \ne 1, k_2 \ne 1}}
 \int_{-\infty}^{\infty} \omega(t) \Big( \frac{t}{2 \pi} \Big)^{-a_1-b_1 -\frac{a_{k_1}+b_{k_2}}{2}}  \times  \\
  &   \Big(
  \prod_{{\begin{substack}{l_1 \ne 1, l_2 \ne 1
         \\ (l_1,l_2) \ne (k_1,k_2)}\end{substack}}} 
      \zeta(1+a_{l_1}+b_{l_2}-a_{k_1}-b_{k_2}) \Big)
      \zeta(1-a_1-b_1+a_{k_1}+b_{k_2}) 
      \cC_{\I,\J; \{ a_1 \}, \{ b_1 \} }(-\tfrac{a_{k_1}+b_{k_2}}{2})
      \frac{G(-\frac{a_{k_1}+b_{k_2}}{2})}{(-\frac{a_{k_1}+b_{k_2}}{2})}
      dt.
\end{split}
\end{equation}

We now provide further simplification of $\mathcal{S}_1$ and $\mathcal{S}_2$. 
By Proposition \ref{ACidentities} (i) and Lemma \ref{ZIJ}, we have 
\[
     \cC_{\I,\J; \{ a_1 \} , \{ b_1 \} }(0)= \cA_{\I_{\{a_1 \}}, \J_{ \{ b_1 \} }}(0) 
     =    \mathcal{Z}_{\I_{\{a_1 \}}, \J_{ \{ b_1 \} }}(0) 
     \cdot \Bigg( \prod_{\substack{(a,b) \\ a \in \I_{\{a_1 \}}, b \in  \J_{ \{ b_1 \} } }}
\frac{1}{\zeta(1+a+b)}  \Bigg).
\]
By this and \eqref{S1}, we obtain
\begin{equation}
  \label{S1id}
  \mathcal{S}_1 = \int_{-\infty}^{\infty} \omega(t)   \Big( \frac{t}{2 \pi} \Big)^{-a_1-b_1}  
 dt \cdot  \mathcal{Z}_{\I_{\{a_1 \}}, \J_{ \{ b_1 \} }}(0).
\end{equation}
 Next, we show that 
\begin{equation}
  \label{S2id}
  \mathcal{S}_2 = -\mathrm{Res}_{s=  \frac{-a_1 -b_1}{2}} \mathcal{Z}_{\I,\J}(2s) \frac{G \Big( \frac{-a_1 -b_1}{2}  \Big)}{ (\frac{-a_1 -b_1}{2})}
  \int_{-\infty}^{\infty} \omega(t) 
  \Big( 
  \frac{t}{2 \pi}
  \Big)^{-\frac{3}{2}(a_1+b_1)}dt.
\end{equation}  
First observe that by Lemma 2.3 and the fact that $\zeta(s)$ has a simple pole at $s=1$ with residue 1, 
\[
   \mathrm{Res}_{s=  \frac{-a_1 -b_1}{2}} \mathcal{Z}_{\I,\J}(2s) 
   =\frac{1}{2} \Big( \prod_{\substack{(i,j) \\ (i,j) \ne (1,1)}} \zeta(1+a_i+b_j-a_1-b_1) \Big)
   \cA_{\I,\J}(-a_1-b_1).
\] 
Thus, in order that \eqref{S2id} may be deduced from \eqref{S2}, it suffices to prove that 
$\cA_{\I,\J}(-a_1-b_1) = $\\
$  \cC_{\I,\J; \{ a_1 \} , \{ b_1 \}}( -\tfrac{a_1+b_1}{2})$ which is Proposition \ref{ACidentities} (ii).
From \eqref{I11expansion}, \eqref{S1id}, \eqref{S2id}, and \eqref{S3} we arrive at
$I_{(i_1,i_2)}^{(1)} = J_{(i_1,i_2)}^{(1)} + O(T^{\frac{3}{4} + 2 \delta})$   in the case that $i_1=i_2=1$, 
assuming \eqref{sizerestrictiondelta} and $\delta \in (0,\frac{1}{12})$. 
The case of general $i_1, i_2$, follows from a simple permutation of variables. 
In summary, we have shown that 
\begin{equation*}
  \sum_{i_1=1}^{3} \sum_{i_2=1}^{3} I_{(i_1,i_2)}^{(1)} 
   =  \sum_{i_1=1}^{3} \sum_{i_2=1}^{3} J_{(i_1,i_2)}^{(1)}  + O(T^{\frac{3}{4}+2 \delta}), 
\end{equation*}
subject to \eqref{sizerestrictiondelta}  and $\delta \in (0,\frac{1}{12})$. 
Combining this with Proposition \ref{IO1prop} we obtain for any $\e'>0$
\[
  I_{O}^{(1)}  =  \sum_{i_1=1}^{3} \sum_{i_2=1}^{3} J_{(i_1,i_2)}^{(1)}  + 
   O\Big(  T^{\frac{3 \vartheta}{2}+\varepsilon'} \Big( \frac{T}{T_0} \Big)^{1+C} +
  T^{\frac{3}{4}+2 \delta}  \Big), 
\]
subject to the stronger condition $|a_i|, |b_i| \ll (\log T)^{-1}$. 
Now for $\e' >0$, we may assume that $\delta < \frac{\e'}{2}$. 
Furthermore, as  $\vartheta \ge \frac{1}{2}$, $C >0$, and  $T_0 \ll T$,  
it  follows that  the second error term is dominated by the first, 
and so we have established the proposition.
\end{proof}
Next, we shall derive 
Proposition \ref{I2i1i2} from  Proposition \ref{I1i1i2}.
\begin{proof}[Proof of Proposition \ref{I2i1i2}]
The expression $I_{O}^{(2)}$ as defined by \eqref{IO2} may viewed as  a special case of \eqref{IO1}, by 
the simultaneous substitutions 
\begin{equation}
  \label{substitutionA}
\I \to -\J,  \,   \J \to -\I, \ \text{ (or } a_j \to -b_j  \text{ and  }  b_j \to -a_j  \text{ for }  j=1,2,3)
\end{equation}
and
\begin{equation}
  \label{substitutionB} \omega(t) \to \omega_1(t) =  X_{\I,\J;t}  \omega(t).
\end{equation} 
By Lemma \ref{Stirling} we have the  asymptotic $X_{\I,\J;t} \sim  ( \frac{t}{2 \pi}  )^{-\sum_{k=1}^{3}(a_k+b_k)}$. 
This means that the formula in Proposition \ref{I2i1i2} can be obtained from Proposition \ref{I1i1i2} by inserting 
the factor $ (t/ 2 \pi)^{-\sum_{k=1}^{3}(a_k+b_k)}$ and permuting the variables as in \eqref{substitutionA}.
However, it must be verified that $\omega_1$ satisfies the conditions \eqref{cond1}, \eqref{cond2}, and \eqref{cond3}. 
The first two conditions are immediate.  In order to show \eqref{cond3} it suffices to show 
for all $j \in \mathbb{Z}_{\ge 0}$ and $t \in [c_1T, c_2T]$ that 
\begin{equation}
   \label{omega1partials}
   \frac{\partial^j}{\partial t^j} X_{\I,\J;t}  \ll_j T^{-j}.
\end{equation}
We obtain this bound as follows.  The function $z \to X_{\I,\J;z}$ is holomorphic on the closed disc $\{ z \in \C \ | \ 
|z-t| \le \frac{t}{2} \}$ and thus we have 
\[
  \Big| \frac{\partial^j}{\partial t^j} X_{\I,\J;t} \Big| 
  \le \frac{ j! }{ (t/2)^j } \max_{ \substack{z \in \C \\ |z-t|=\frac{t}{2}} } |  X_{\I,\J;z}|.
\]
It follows from Stirling's formula (see \cite[(C.18), p. 523]{MV}) that, when $t \in [c_1T,c_2T]$ and $|a_i|,|b_i| \ll \frac{1}{\log T}$ for $i=1,2,3$, one has 
$$\max_{ \substack{z \in \C \\ |z-t|=\frac{t}{2}} } |  X_{\I,\J;z}| \ll 1.
$$
The last two bounds combine to give \eqref{omega1partials}.  Thus, by the generalized product rule, it follows
that $\omega_1$ satisfies \eqref{cond3}.   Since we have shown that 
$I_{O}^{(1)} = \sum_{i_1=1}^{3} \sum_{i_2=1}^{3} J_{(i_1,i_2)}^{(1)} 
   + O(  T^{\frac{3 \vartheta}{2}+\varepsilon} ( \frac{T}{T_0} )^{1+C} 
 )
$, it suffices to determine how the substitutions \eqref{substitutionA} and \eqref{substitutionB} change $J_{(1,1)}^{(1)}$. 
Thus we have  $I_{O}^{(2)}= \sum_{i_1=1}^{3} \sum_{i_2=1}^{3} {\tilde J}_{(i_1,i_2)}
   + O ( T^{\frac{3}{4}+4 \delta})$ 
where ${\tilde{J}}_{(1,1)}$ is obtained from $J_{(1,1)}^{(1)}$ by the  the substitutions \eqref{substitutionA}. 
It follows from \eqref{J1i1i2id} that the first term in $J_{(1,1)}^{(1)}$ is 
$   \mathcal{Z}_{-\J_{\{-b_1 \}}, \I_{ \{ -a_1 \} }}(0)  \int_{-\infty}^{\infty} 
  \omega_1(t)
 (\tfrac{t}{2 \pi})^{b_1+a_1}dt$.
Note that $
   -\J_{\{-b_1\}}=\{a_1,-b_2,-b_3\} =\I_{\{a_2,a_3\}}$ and 
$
  -\I_{\{-a_1\}} = \{ b_1,-a_2,-a_3 \} =\J_{ \{ b_2,b_3 \}}$
 and thus by Lemma \ref{Stirling} (i) the first term in  ${\tilde{J}}_{(1,1)}$ equals 
 \begin{equation}
   \label{firsttermB}
  \int_{-\infty}^{\infty}  \omega(t)   \Big( \frac{t}{2 \pi} \Big)^{-a_2-a_3-b_2-b_3}  
  \mathcal{Z}_{\I_{\{a_2,a_3 \}}, \J_{ \{ b_2,b_3 \} }}(0)(1+O(t^{-1})) dt.
\end{equation}
By Lemma \ref{ZIJ}, \eqref{sizerestriction}, \eqref{initialconditionA}, and \eqref{initialconditionB},
the term $O(t^{-1})$ in \eqref{firsttermB} leads to an error of size
\[
    O \Big( \prod_{x \in \I_{\{a_2,a_3 \}} , y \in \J_{ \{ b_2,b_3 \}} } |x+y|^{-1}  \int_{c_1 T}^{c_2 T} t^{-1} dt 
    \Big)
    \ll (\log T)^9. 
\]
Similarly,  the second term of ${\tilde{J}}_{(1,1)}$ is
\begin{equation}
  \label{secondterm}
  -\frac{1}{2}\mathrm{Res}_{s=  \frac{a_1 +b_1}{2}} \mathcal{Z}_{-\J,-\I}(2s) \frac{G \Big( \frac{a_1 +b_1}{2}  \Big)}{ (\frac{a_1+b_1}{2})}
  \int_{-\infty}^{\infty} \omega(t)
  \Big( 
  \frac{t}{2 \pi}
  \Big)^{
  -\sum_{k=1}^{3}(a_k+b_k)
  +\frac{3(a_1+b_1)}{2}} (1+O(t^{-1}))dt 
\end{equation}
and by Lemma \ref{ZIJ}, \eqref{initialconditionA}, \eqref{initialconditionB}, and \eqref{initialconditionC} the big $O$ term leads to an error 
\[
 O \Bigg( \Big( \prod_{\substack{ x \in -\J, y \in  -\I \\ (x,y) \ne (-b_1,-a_1) }  } \frac{1}{|a_1+b_1+x+y|} 
 \Big) \frac{1}{|a_1+b_1|}    \int_{c_1 T}^{c_2 T} t^{-1} dt 
 \Bigg)
 \ll (\log T)^{9}. 
\]
The third term in ${\tilde{J}}_{(1,1)}$ is 
\begin{equation}
\begin{split}
  \label{thirdterm}
  &\frac{1}{2}
  \prod_{j_1 \ne 1}\zeta(1+b_1-b_{j_1}) \prod_{j_2 \ne 1} \zeta(1+a_1-a_{j_2})  
   \sum_{ \substack{ (k_1,k_2) \\
 k_1 \ne 1, k_2 \ne 1}}   \prod_{{\begin{substack}{l_1 \ne 1, l_2 \ne 1
         \\ (l_1,l_2) \ne (k_1,k_2)}\end{substack}}} 
      \zeta(1-b_{l_1}-a_{l_2}+b_{k_1}+a_{k_2}) \\
  & \times  
      \zeta(1+b_1+a_1-b_{k_1}-a_{k_2})  
     \cC_{-\J,-\I; \{-b_1\}, \{ -a_1 \}}(\tfrac{b_{k_1}+a_{k_2}}{2})
      \frac{G(\frac{b_{k_1}+a_{k_2}}{2})}{(\frac{b_{k_1}+a_{k_2}}{2})} \\
     & \times  \int_{-\infty}^{\infty} \omega(t) \Big( \frac{t}{2 \pi} \Big)^{-\sum_{k=1}^{3}(a_k+b_k)+a_1+b_1 +\frac{b_{k_1}+a_{k_2}}{2}}
 (1+O(t^{-1})) dt.
\end{split}
\end{equation}
By Proposition \ref{Hid} and a calculation similar to those above, the $O(t^{-1})$ term in \eqref{thirdterm}
leads to an error there that is of size $O((\log T)^9)$. Note that the main term in \eqref{thirdterm}
simplifies to
\begin{equation}
\begin{split}
  \label{thirdtermB}
  &+  \frac{1}{2}
  \prod_{j_1 \ne 1}\zeta(1+b_1-b_{j_1}) \prod_{j_2 \ne 1} \zeta(1+a_1-a_{j_2})  
   \sum_{ \substack{ (k_1,k_2) \\
 k_1 \ne 1, k_2 \ne 1}}
 \int_{-\infty}^{\infty} \omega(t) \Big( \frac{t}{2 \pi} \Big)^{-a_{r_2}-b_{r_1} -\frac{b_{k_1}+a_{k_2}}{2}} dt   \\
  & \times  \Big(
  \prod_{{\begin{substack}{l_1 \ne 1, l_2 \ne 1
         \\ (l_1,l_2) \ne (k_1,k_2)}\end{substack}}} 
      \zeta(1-b_{l_1}-a_{l_2}+b_{k_1}+a_{k_2}) \Big)
      \zeta(1+b_1+a_1-b_{k_1}-a_{k_2}) \\
   & \times   \cC_{-\J,-\I; \{-b_1\}, \{ -a_1 \}}(\tfrac{b_{k_1}+a_{k_2}}{2})
      \frac{G(\frac{b_{k_1}+a_{k_2}}{2})}{(\frac{b_{k_1}+a_{k_2}}{2})}
\end{split}
\end{equation}
where we recall that $r_1=r_1(1,k_1)$ and $r_2=r_2(1,k_2)$ are defined by \eqref{rdefinition}. 
By \eqref{firsttermB}, \eqref{secondterm}, \eqref{thirdterm}, and \eqref{thirdtermB}, we find that 
when $(i_1,i_2)=(1,1)$ one has $\tilde{J}_{(i_1,i_2)} = J_{(i_1,i_2)}^{(2)} + 
O((\log T)^9)$, where $J_{(i_1,i_2)}^{(2)}$ is defined by \eqref{I2i1i2id} and \eqref{U}.  
The general case, $(i_1,i_2) \in \{ 1, 2, 3 \}^2$, it follows by considering all permutations of the form
$a_i \to a_{\sigma(i)}$,
$b_i \to b_{\tau(i)} (i=1,2,3)$ in which $\sigma$ and $\tau$ are given elements of the subgroup of permutations of the
set $\{  1,2,3\}$ that is generated by the $3$-cycle $(2 \, 3 \, 1)$.  
\end{proof}
\begin{proof}[Proof of Proposition \ref{sumtozero}]
Recall that we are trying to prove that 
\begin{equation}
  \label{sumtozero2}
  \sum_{i_1=1}^{3} \sum_{i_2=1}^{3} \sum_{ \substack{ (k_1,k_2) \\
 k_1 \ne i_1, k_2 \ne i_2}} ( \mathcal{T}_{(i_1,i_2);(k_1,k_2)}+\mathcal{U}_{(i_1,i_2);(k_1,k_2)})=0.
\end{equation}
We aim to prove this by matching terms in the two triple sums.   First we show that 
\begin{equation}
  \label{specialcase}
\mathcal{T}_{(1,1);(2,2)} + \mathcal{U}_{(3,3);(2,2)} = 0.
\end{equation}
We begin with a few observations.    Note that 
\begin{equation*}
\begin{split}
  \label{T1122}
\mathcal{T}_{(1,1);(2,2)} & = \frac{1}{2}
   \zeta(1-a_1+a_{2})\zeta(1-a_1+a_{3})\zeta(1-b_1+b_{2}) \zeta(1-b_1+b_{3}) 
 \int_{-\infty}^{\infty} \omega(t) \Big( \frac{t}{2 \pi} \Big)^{-a_1-b_1 -\frac{a_{2}+b_{2}}{2}} dt   \\
  & \times 
      \zeta(1+a_{3}+b_{2}-a_{2}-b_{2})
         \zeta(1+a_{2}+b_{3}-a_{2}-b_{2})
            \zeta(1+a_{3}+b_{3}-a_{2}-b_{2})
      \zeta(1-a_1-b_1+a_{2}+b_{2}) \\
   & \times   \cC_{\I,\J; \{a_1 \}, \{ b_1 \} }(-\tfrac{a_{2}+b_{2}}{2})
      \frac{G(-\frac{a_{2}+b_{2}}{2})}{(-\frac{a_{2}+b_{2}}{2})} \\
&= - \frac{1}{2}
   \zeta(1-a_1+a_{2})\zeta(1-a_1+a_{3})\zeta(1-b_1+b_{2}) \zeta(1-b_1+b_{3}) 
 \int_{-\infty}^{\infty} \omega(t) \Big( \frac{t}{2 \pi} \Big)^{-a_1-b_1 -\frac{a_{2}+b_{2}}{2}} dt   \\
  & \times 
      \zeta(1+a_{3}-a_{2})
         \zeta(1+b_{3}-b_{2})
            \zeta(1+a_{3}+b_{3}-a_{2}-b_{2})
      \zeta(1-a_1-b_1+a_{2}+b_{2}) \\
   & \times   \cC_{\I,\J; \{a_1 \}, \{ b_1 \} }(-\tfrac{a_{2}+b_{2}}{2})
      \frac{G(\frac{a_{2}+b_{2}}{2})}{(\frac{a_{2}+b_{2}}{2})},
\end{split}
\end{equation*}
since $G$ is even.  
We now try to identify a term which will cancel with this.  
We shall look in the terms coming 
from the second half of the approximate functional equation.  We guess the correct term 
arises from  $I_{(3,3)}^{(2)}$ and is $ \mathcal{U}_{(3,3);(2,2)}$. 
Note that, by the definition \eqref{rdefinition}, we have $r(3,2)=1$, and so, by \eqref{U}, 
\begin{equation*}
\begin{split}
  \label{U3322}
 \mathcal{U}_{(3,3);(2,2)}   & = \frac{1}{2}
  \zeta(1+b_{3}-b_{1})\zeta(1+b_{3}-b_{2})
  \zeta(1+a_{3}-a_{1})   \zeta(1+a_{3}-a_{2})  
 \int_{-\infty}^{\infty} \omega(t) \Big( \frac{t}{2 \pi} \Big)^{-a_{1}+b_{1} -\frac{b_{2}+a_{2}}{2}} dt   \\
  & \times 
      \zeta(1-b_{1}-a_{1}+b_{2}+a_{2})  
      \zeta(1-b_{1}-a_{2}+b_{2}+a_{2}) 
       \zeta(1-b_{2}-a_{1}+b_{2}+a_{2}) 
      \zeta(1+b_{3}+a_{3}-b_{2}-a_{2}) \\
   & \times   \cC_{-\J,-\I; \{-b_{3}\}, \{ -a_{3} \}}(\tfrac{b_{2}+a_{2}}{2})
      \frac{G(\frac{b_{2}+a_{2}}{2})}{(\frac{b_{2}+a_{2}}{2})} \\
& = \frac{1}{2}
  \zeta(1+b_{3}-b_{1})\zeta(1+b_{3}-b_{2})
  \zeta(1+a_{3}-a_{1})   \zeta(1+a_{3}-a_{2})  
 \int_{-\infty}^{\infty} \omega(t) \Big( \frac{t}{2 \pi} \Big)^{-a_{1}+b_{1} -\frac{b_{2}+a_{2}}{2}} dt   \\
  & \times 
      \zeta(1-b_{1}-a_{1}+b_{2}+a_{2})  
      \zeta(1-b_{1}+b_{2}) 
       \zeta(1-a_{1}+a_{2}) 
      \zeta(1+b_{3}+a_{3}-b_{2}-a_{2}) \\
   & \times   \cC_{-\J,-\I; \{-b_{3}\}, \{ -a_{3} \}}(\tfrac{b_{2}+a_{2}}{2})
      \frac{G(\frac{b_{2}+a_{2}}{2})}{(\frac{b_{2}+a_{2}}{2})}.
\end{split}
\end{equation*}
Observe that the two expressions we are considering are negatives of each other and add to zero 
if 
\begin{equation*}
\cC_{\I,\J; \{a_1 \}, \{ b_1 \} }(-\tfrac{a_{2}+b_{2}}{2})=
     \cC_{-\J,-\I; \{-b_{3}\}, \{ -a_{3} \}}(\tfrac{b_{2}+a_{2}}{2}).
\end{equation*}
However, this identity is Proposition \eqref{ACidentities} (iii). 
Thus this establishes \eqref{specialcase}.  
More generally, we can show that for $(i_1,i_2) \in \{ 1, 2, 3 \}^2$ and $(k_1,k_2)  \in \{ 1, 2, 3 \}^2$ such that 
$k_1 \ne i_1$ and $k_2 \ne i_2$ that 
\begin{equation}
  \label{generalcaseB}
\mathcal{T}_{(i_1,i_2);(k_1,k_2)} + \mathcal{U}_{(r_2,r_1);(k_2,k_1)} = 0
\end{equation} 
where   $r_1= r(i_1,k_1)$ and $r_2= r(i_2,k_2)$ are defined by \eqref{rdefinition}. 
By a calculation similar to that seen in our proof of \eqref{specialcase}, we find that 
\eqref{generalcaseB} follows if one has
\[
   \cC_{\I,\J; \{a_{i_1} \}, \{ b_{i_1} \} }(-\tfrac{a_{k_1}+b_{k_2}}{2})=
     \cC_{-\J,-\I; \{-b_{r_2}\}, \{ -a_{r_1} \}}(\tfrac{b_{k_2}+a_{k_1}}{2}).
\]
 This is Proposition \ref{ACidentities} (iv).
Finally, summing \eqref{generalcase} over $i_1, i_2, k_1, k_2$ 
with $i_1 \ne k_1$, $i_2 \ne k_2$
leads to \eqref{sumtozero2}. 
\end{proof}


%

\appendix 
\section*{Appendix}
\section{ Proof of Propositions \ref{Hid}, \ref{ACidentities}, Lemma \ref{Hbounds}} \label{appendix1}
We now establish Proposition \ref{Hid}. 
\begin{lem} \label{gmult}
Let $k \in \mathbb{N}$, $I=\{1, \ldots, k \}$, and
let $\X =\{ x_1, x_2, \ldots, x_k\}$  be distinct complex numbers. 
For $\Re(s) > -\min_{\ell=1, \ldots k} \Re(x_{\ell})$, $p$ prime, and $\alpha \ge 0$
\begin{equation}
  \label{gXspalpha}
  g_{\X}(s,p^{\alpha}) =  (1-p^{-s-x_1}) \cdots (1-p^{-s-x_k}) 
   \sum_{i=1}^{k}
  \frac{ p^{-x_i \alpha}}{1-p^{-x_i-s}}
  \prod_{\ell \in I \setminus \{ i \}} (1-p^{x_i-x_{\ell}})^{-1}.
\end{equation}
\end{lem}
\begin{proof}  We begin by recalling that 
\begin{equation}
  \label{gXdefnB}
  g_{\X}(s,p^{\alpha}) = \frac{\sum_{j=0}^{\infty} \frac{\sigma_{\X}(p^{\alpha+j})}{p^{js}}}{ \sum_{j=0}^{\infty} \frac{\sigma_{\X}(p^j)}{p^{js}}}.
\end{equation}
We now find an expression for $\sigma_{\X}(p^{j})$ for $j \ge 0$.  By the multiplicativity of $\sigma_{X}$ and the Euler product, 
we have the identitity
\[
\zeta(s+x_1) \cdots \zeta(s+x_k) = \prod_{p} \sum_{j=0}^{\infty} \sigma_{\X}(p^j) p^{-js}
=\prod_{p}
(1-p^{-s-x_1})^{-1} \cdots (1-p^{-s-x_k})^{-1}.
\]
This is valid for $\Re(s) > 1- \min_{\ell =1, \ldots, k} \Re(x_{\ell})$, where
the Euler product is absolutely convergent. 
It follows that
\begin{equation}
  \label{local1}
 \sum_{j=0}^{\infty} \sigma_{\X}(p^j) p^{-js}=(1-p^{-s-x_1})^{-1} \cdots (1-p^{-s-x_k})^{-1} \ne 0
\end{equation}
for $\Re(s) > 1- \min_{\ell =1, \ldots, k} \Re(x_{\ell})$. In \eqref{local1} both the sum and product
are holomorphic on the half-plane  $\Re(s) > - \min_{\ell =1, \ldots, k} \Re(x_{\ell})$, and the product
is non-zero there.  We can deduce that \eqref{local1} must hold for all $s$ lying in that half-plane:
in particular, the denominator of the right hand side of \eqref{gXdefnB} is non-zero for such values of $s$.
By partial fractions and the geometric series, 
\begin{equation}
\begin{split}
  \label{parfrac}
 (1-p^{-s-x_1})^{-1} \cdots (1-p^{-s-x_k})^{-1}
 & = \sum_{i=1}^{k}  (1-p^{-s-x_i})^{-1} \prod_{\ell \in I \setminus \{ i \}} (1-p^{x_i-x_{\ell}})^{-1} \\
 & =  \sum_{i=1}^{k} \prod_{\ell \in I \setminus \{ i \}} (1-p^{x_i-x_l})^{-1}
 \sum_{j=0}^{\infty}  p^{-x_i j} p^{-js}.
\end{split}
\end{equation}
Hence for $j \ge 0$
\[
  \sigma_{\X}(p^j) = \sum_{i=1}^{k} p^{-x_i j} \prod_{\ell \in I \setminus \{ i \}} (1-p^{x_i-x_{\ell}})^{-1}
\]
and for $\alpha \ge 1$
\[
  \sum_{j=0}^{\alpha-1}  \sigma_{\X}(p^j) p^{-js}
  = \sum_{i=1}^{k}
  \frac{1- (p^{-x_i-s})^{\alpha}}{1-p^{-x_i-s}}
  \prod_{\ell \in I \setminus \{ i \}} (1-p^{x_i-x_{\ell}})^{-1}. 
\]
Thus 
\begin{equation}
\begin{split}
  \label{local2}
  & \sum_{j = 0}^{\infty} \sigma_{\X}(p^{j+\alpha}) p^{-js}
  =p^{\alpha s} \Big( \sum_{j \ge 0} \sigma_{\X}(p^j) p^{-js}
  - \sum_{j=0}^{\alpha-1} \sigma_{\X}(p^j) p^{-js} \Big)  \\
  &  =p^{\alpha s} 
  \sum_{i=1}^{k}
  \frac{ (p^{-x_i-s})^{\alpha}}{1-p^{-x_i-s}}
  \prod_{\ell \in I \setminus \{ i \}} (1-p^{x_i-x_{\ell}})^{-1} 
   =
  \sum_{i=1}^{k}
  \frac{ p^{-x_i \alpha}}{1-p^{-x_i-s}}
  \prod_{\ell \in I \setminus \{ i \}} (1-p^{x_i-x_{\ell}})^{-1}.
\end{split}
\end{equation}
By \eqref{gXdefnB}, \eqref{local1}, and \eqref{local2} we obtain \eqref{gXspalpha} for $\alpha \ge 1$.
In the case that $\alpha=0$, we observe that the left hand side equals 1 since $g_X(s,n)$ is multiplicative
and the right hand side also equals 1 by \eqref{parfrac}.
\end{proof}

\begin{lem} \label{Gmult}
Let $k \in \mathbb{N}$, $I=\{1, \ldots, k \}$, and
$\X =\{ x_1, x_2, \ldots, x_k\}$ be distinct complex numbers. 
For $\Re(s) > -\min_{\ell=1, \ldots k} \Re(x_{\ell})$,
$p$ prime, and $j \ge 1$
\begin{equation}
 \label{GXspj}
  G_{\X}(s,p^j) = 
(1-p^{-s-x_1}) \cdots (1-p^{-s-x_k}) \frac{1}{p-1}
   \sum_{i=1}^{k}
  \frac{ p^{1-x_i j} - p^{s-x_i (j-1)}}{1-p^{-x_i-s}}
  \prod_{\ell \in I \setminus \{ i \}} (1-p^{x_i-x_{\ell}})^{-1}.
\end{equation}
\end{lem}
\begin{proof}  
By definition \eqref{GXdefn} it follows that 
\begin{align*}
  G_{\X}(s,p^j) & = \sum_{d \mid p^j} \frac{\mu(d) d^s}{\phi(d)}
  \sum_{e \mid d} \frac{\mu(e)}{e^s} g_{\X} \Big(s, \frac{p^j e}{d}\Big) 
  = g_{\X}(s,p^j)-\frac{p^s}{p-1} g_{\X}(s,p^{j-1})+ \frac{1}{\phi(p)} g_{\X}(s,p^j) \\
  & = \frac{p}{p-1} g_{\X}(s,p^j)- \frac{p^s}{p-1}g_{\X}(s,p^{j-1}).
 \end{align*}
 Inserting \eqref{gXspalpha} in the last expression with $\alpha =j $ and $\alpha=j-1$ yields
 \eqref{GXspj}. 
\end{proof}
Observe that we may apply the preceding result in the special case $\X=\I$ and $s=1-a_1$.  
\begin{lem} \label{GmultI3}
Let $\delta \in (0,\frac{1}{2})$ and $\I = \{a_1,a_2,a_3 \}$ where $a_1,a_2,a_3$, are distinct complex numbers
with $|a_i| \le \delta$ for $i=1,2,3$.
 Let $p$ be a prime and $j \ge 1$.
  Then
\begin{equation}
\label{GI3}
  G_{\I}(1-a_1,p^j) = p^{-a_2 j} \frac{1-p^{-1+a_1-a_3}}{1-p^{a_2-a_3}}
  + p^{-a_3 j} \frac{1-p^{-1+a_1-a_2}}{1-p^{a_3-a_2}}
\end{equation}
and, in  particular, 
\[
  G_{\I}(1-a_1,p) = p^{-a_2}+p^{-a_3}-p^{-1+a_1-a_2-a_3}.
\]
\end{lem}
\begin{proof}
Note that for $\X= \I=\{ a_1, a_2, a_3 \}$ and $s=1-a_1$ the condition
$\Re(s) >-\min_{\ell=1, \ldots 3} \Re(a_{\ell})$ is met since $\delta < \frac{1}{2}$.
By Lemma \ref{Gmult} it follows that 
\begin{align*}
   G_{\I}(1-a_1,p^j)
   & = (1-p^{-1})(1-p^{-1+a_1-a_2})(1-p^{-1+a_1-a_3})
   \frac{1}{p-1}  \\
   & \times
    \sum_{i=1}^{3}
  \frac{ p^{1-a_i j} - p^{(1-a_1)-a_i (j-1)}}{1-p^{-a_i-(1-a_1)}}
  \prod_{\ell \in I \setminus \{ i \}} (1-p^{a_i-a_{\ell}})^{-1}.
\end{align*}
Note that $(1-p^{-1})(p-1)^{-1}=p^{-1}$, and that for $i=1$ the exponents 
$1-a_1-a_i(j-1)$ and $1-a_i j$ are equal. 
Therefore 
\begin{align*}
 & G_{\I}(1-a_1,p^j)
   =(1-p^{-1+a_1-a_2})(1-p^{-1+a_1-a_3}) \\
   & \times
   \Big(
   \frac{p^{-a_2j}(1-p^{a_2-a_1})}{(1-p^{-1+a_1-a_2})(1-p^{a_2-a_1})(1-p^{a_2-a_3})}
   +   \frac{p^{-a_3j}(1-p^{a_3-a_1})}{(1-p^{-1+a_1-a_3})(1-p^{a_3-a_1})(1-p^{a_3-a_2})}
   \Big).
\end{align*}
Simplifying this yields \eqref{GI3}.
\end{proof}
\begin{lem}  \label{GIbd}
Let $0< \delta < \frac{1}{2}$ and $\I = \{ a_1, a_2, a_3 \}$, where $a_1,a_2,a_3 \in C$ are distinct
and satisfy 
$|a_i| \le \delta$ for $i=1,2,3$. Then , for $\ell \in \mathbb{N}$, we have  $|G_{\I}(1-a_1,\ell)| \le \ell^{\delta} d_2(\ell^2)$. 
\end{lem}
\begin{proof}
By multiplicativity, it suffices to prove the result for $\ell=p^j$ with $j \ge 1$.  
By \eqref{GI3} we have 
\begin{equation*}
\begin{split}
  G_{\I}(1-a_1,p^j) & = \frac{p^{-(j+1)a_2}-p^{-(j+1)a_3}}{p^{-a_2}-p^{-a_3} }
  - \Big( 
  \frac{ p^{-ja_2}-p^{-ja_3} }{p^{-a_2}-p^{-a_3} } 
  \Big) p^{-1+a_1-a_2-a_3}  \\
  & =  \sum_{\ell=0}^{j} p^{-(j-\ell) a_2} p^{-\ell a_3}
  -p^{-1+a_1-a_2-a_3} \sum_{\ell=0}^{j-1} p^{-(j-1-\ell)a_2} p^{-\ell a_3}. 
\end{split}
\end{equation*}
Since $|a_i| \le \delta$ for $i=1,2,3$, it follows that 
\begin{equation*}
\begin{split}
   |G_{\I}(1-a_1,p^j)| & \le 
    \sum_{\ell=0}^{j} p^{(j-\ell) \delta} p^{\ell \delta}
  +p^{-1+3 \delta} \sum_{\ell=0}^{j-1} p^{(j-1-\ell) \delta} p^{\ell \delta}
  = (j+1) p^{j \delta} + j p^{j \delta +2 \delta-1}  \le p^{j \delta} d_2(p^{2j})
\end{split}
\end{equation*} 
as long as $\delta < \frac{1}{2}$.  
\end{proof}

With Lemma \ref{GmultI3} in hand we can proceed with the proof of Proposition \ref{Hid}. 

\begin{proof}[Proof  of Proposition \ref{Hid}]  (i) Throughout the proof of this proposition we let $\sigma =\Re(s)$. 
Observe that by the bound $|c_{\ell}(r)|  \le (\ell,r)$ and Lemma \ref{GIbd}, we can check that the series 
for $H_{\I,\J;\{ a_1 \} , \{ b_1 \} }(s)$
is absolutely convergent 
for $\Re(s) > \frac{1}{2} +2 \delta$ and $\delta < \frac{1}{4}$. Furthermore, 
since $c_{\ell}(r) = \sum_{d \mid (\ell,r)} d \mu(\tfrac{\ell}{d})$ we have
\begin{align*}
 H_{\I,\J;\{ a_1 \} , \{ b_1 \} }(s) & =\sum_{r=1}^{\infty}
      \sum_{\ell=1}^{\infty} \frac{G_{\I}(1-a_1,\ell)G_{\J}(1-b_1,\ell)  }{\ell^{2-a_1-b_1} r^{a_1+b_1+2s}}
      \sum_{d \mid (l,r)} d \mu(\tfrac{\ell}{d}) 
        =  \sum_{\ell=1}^{\infty} \alpha_{\ell} \sum_{r=1}^{\infty} \frac{1}{r^c} 
   \sum_{d \mid l, d \mid r} d \mu(\tfrac{\ell}{d}) 
\end{align*} 
where 
$\alpha_{\ell} =  \frac{G_{\I}(1-a_1,\ell)G_{\J}(1-b_1,\ell)  }{\ell^{2-a_1-b_1}}$ and $c=a_1+b_1+2s$.
Thus 
\begin{align*}
    H_{\I,\J;\{ a_1 \} , \{ b_1 \} }(s) & =\sum_{\ell=1}^{\infty} \alpha_{\ell} \sum_{d \mid \ell} d 
    \mu(\tfrac{\ell}{d}) \sum_{r \ge 1, d \mid r} \frac{1}{r^c} 
  = \sum_{\ell=1}^{\infty} \alpha_{\ell} 
  \sum_{d \mid l}  \frac{d \mu(\tfrac{\ell}{d})}{d^c} \zeta(c) \\
  & = \zeta(c) \sum_{\ell=1}^{\infty} \alpha_{\ell} \sum_{d \mid \ell} d^{1-c} \mu(\tfrac{\ell}{d}) 
   = \zeta(c) \sum_{\ell=1}^{\infty} \alpha_{\ell}  \ell^{1-c} \sum_{d \mid \ell} \frac{\mu(d)}{d^{1-c}}.
 \end{align*}
 For $p$ prime and $j \ge 1$ we have $
    \sum_{d \mid p^j} \frac{\mu(d)}{d^{1-c}}
    = 1- \frac{1}{p^{1-c}}$.  By multiplicativity
 \begin{equation}
 \begin{split}
  \label{ellsum} 
  \sum_{\ell=1}^{\infty} \alpha_{\ell}  \ell^{1-c} \sum_{d \mid \ell} \frac{\mu(d)}{d^{1-c}}
  & = \prod_{p} \Big(
  1 + \sum_{j=1}^{\infty} 
  \frac{G_{\I}(1-a_1,p^j)G_{\J}(1-b_1,p^j)}{(p^j)^{2-a_1-b_1}}
  (p^j)^{1-a_1-b_1-2s} ( 1- p^{a_1+b_1+2s-1})
  \Big)  \\
&   = \prod_{p} \Big( 
  1 + \sum_{j=1}^{\infty} G_{\I}(1-a_1,p^j)G_{\J}(1-b_1,p^j)
        \cdot \frac{  1-p^{a_1+b_1+2s-1} }{ (p^j)^{1+2s}}
  \Big).
\end{split}
\end{equation}

We aim to simplify the expression with the brackets. 
At this point it will be convenient to introduce some notation.  Let 
\begin{equation}
  \label{variables}
U=p^{-1},  \ W = p^{-1-2s}, \ X_i=p^{-a_i},  \ Y_i = p^{-b_i},  \text{ for } i=1,2,3. 
\end{equation}
Observe that we have the bounds 
\begin{align}
  \label{variableboundsa}
  |W| & \le p^{-1-2 \sigma}, \\
  \label{variableboundsb}
   p^{-\delta} \le  |X_i|, |Y_i| & \le p^{\delta} \text{ for }  i=1,2,3.
\end{align}
Also set
\begin{equation}
  \label{Tjfj}
  T_j  = G_{\I}(1-a_1,p^j)G_{\J}(1-b_1,p^j)  \text{ and }  f_j  = \frac{  1-p^{a_1+b_1+2s-1} }{ (p^j)^{1+2s}}
  \text{ for } j \in \mathbb{Z}_{\ge 0}. 
\end{equation}
Observe that by Lemma  \ref{GmultI3} 
\begin{equation}
\begin{split}
  \label{T1}
  T_1 =    G_{\I}(1-a_1,p)G_{\J}(1-b_1,p)
  =  (X_2 + X_3- U X_2 X_3 X_{1}^{-1})(Y_2 + Y_3- U Y_2 Y_3 Y_{1}^{-1}) 
\end{split}
\end{equation}
and from \eqref{variables} and  \eqref{Tjfj} we have 
\begin{equation}
  \label{fjformula}
  f_j =   W^j - X_{1}^{-1} Y_{1}^{-1} U^2 W^{j-1} \text{ for } j \in \mathbb{N}. 
\end{equation}
With these observations in hand, we see that the first two terms within the brackets in \eqref{ellsum} 
sum to
\begin{equation}
\begin{split}
  \label{expansion}
 1+T_1 f_1& = 1 + (X_2 + X_3- U X_2 X_3 X_{1}^{-1}) ( Y_2 + Y_3- U Y_2 Y_3 Y_{1}^{-1}) W (1- X_{1}^{-1} Y_{1}^{-1}  U^2W^{-1})  \\
 & = 1 + (X_2 Y_2 + X_2 Y_3 + Y_2 X_3 + X_3 Y_3)W + 
 b_{20} U^2 +b_{30}U^3 +b_{40} U^4 +b_{11} UW +b_{21} U^2 W
\end{split}
\end{equation}
where the $b_{uv}$ is a polynomial in $X_{1}^{-1}, X_2, X_3, Y_{1}^{-1},Y_2, Y_3$
that is homogeneous of degree $2(u+v)$, so that, by \eqref{variableboundsb}, one has
\begin{equation*}
  \label{buvbounds}
  |b_{uv}| \ll p^{2(u+v) \delta}  
\end{equation*}
for all $u,v$. 
It follows that 
\begin{equation}
\begin{split}
   \label{errbound}
   & b_{20} U^2 +b_{30}U^3 +b_{40} U^4 +b_{11} UW +b_{21} U^2 W 
    \ll \sum_{j=2}^{4}  p^{2 j\delta} p^{-j} + p^{-1-2 \sigma} (p^{4 \delta} p^{-1} + p^{6 \delta} p^{-2}) \\
    & \ll p^{4 \delta -2} ( 1 + p^{2 \delta -1} + p^{4 \delta-2})
    + p^{4 \delta-2-2 \sigma} ( 1+ p^{2 \delta-1}) 
    \ll p^{4 \delta} (p^{-2-2\sigma}+p^{-2}), 
\end{split}
\end{equation}
since $\delta < \frac{1}{2}$.   Combining \eqref{expansion} and \eqref{errbound} it follows that the 
sum of the first two terms within the brackets in \eqref{ellsum} equal
\begin{equation*}
\begin{split}
  \label{first2terms}
 & 1 +  p^{-1-a_2-b_2-2s}+p^{-1-a_2-b_3-2s} + p^{-1-a_3-b_2-2s}+p^{-1-a_3-b_3-2s} 
 +O((p^{-2-2\sigma}+p^{-2})p^{4 \delta}).
\end{split}
\end{equation*}
This leads us to write the sum over $\ell$ in \eqref{ellsum}   as
\begin{align*}
  \zeta(1+a_2+b_2+2s) \zeta(1+a_2+b_3+2s) \zeta(1+a_3+b_2+2s) \zeta(1+a_3+b_3+2s)  
 \mathcal{C}_{\I,\J; \{ a_1 \} , \{ b_1 \}}(s), 
\end{align*} 
valid for $\Re(s) \ge 2\delta$, 
where
\begin{equation*}
    \mathcal{C}_{\I,\J; \{ a_1 \} , \{ b_1 \}}(s)
  = \prod_{p} \mathcal{C}_{\I,\J; \{ a_1 \} , \{ b_1 \}}(p;s) 
\end{equation*}
and 
\begin{equation}
\begin{split}
  \label{CIJa1b1ps}
   \mathcal{C}_{\I,\J; \{ a_1 \} , \{ b_1 \}}(p;s) 
&  =   \Big(
 1 + \sum_{j=1}^{\infty} \frac{G_{\I}(1-a_1,p^j)G_{\J}(1-b_1,p^j)}{(p^j)^{1+2s}}
    (1-p^{a_1+b_1+2s-1})
\Big) \times \\
&   \Big(1- \frac{1}{p^{1+a_2+b_2+2s}} \Big) 
  \Big(1- \frac{1}{p^{1+a_2+b_3+2s}}\Big)
\Big(1- \frac{1}{p^{1+a_3+b_2+2s}}\Big)
\Big(1- \frac{1}{p^{1+a_3+b_3+2s}}\Big).
\end{split}
\end{equation}
Hence
\begin{align*}
   & H_{\I,\J;\{ a_1 \} , \{ b_1 \} }(s) = 
      \Bigg( \prod_{k=2}^{3} \prod_{\ell=2}^{3} 
     \zeta(1+a_k+b_{\ell}+2s) \Bigg)  \zeta(a_1+b_1+2s)  \mathcal{C}_{\I,\J; \{ a_1 \} , \{ b_1 \}}(s).
\end{align*}
The next step is to derive an explicit formula for $\mathcal{C}_{\I,\J; \{ a_1 \} , \{ b_1 \}}(s)$, namely \eqref{lps}. 
By \eqref{GI3} 
It follows that 
\begin{align*}
\frac{G_{\I}(1-a_1,p^j) G_{\J}(1-b_1,p^j)}{(p^j)^{1+2s}}  
 & =
p^{-(1+a_2+b_2+2s) j} \frac{1-p^{-1+a_1-a_3}}{1-p^{a_2-a_3}} \frac{1-p^{-1+b_1-b_3}}{1-p^{b_2-b_3}}  \\
& +  p^{-(1+a_2+b_3+2s) j} \frac{1-p^{-1+a_1-a_3}}{1-p^{a_2-a_3}}\frac{1-p^{-1+b_1-b_2}}{1-p^{b_3-b_2}}  \\
& +  p^{-(1+a_3+b_2+2s) j} \frac{1-p^{-1+a_1-a_2}}{1-p^{a_3-a_2}}\frac{1-p^{-1+b_1-b_3}}{1-p^{b_2-b_3}} \\
& +  p^{-(1+a_3+b_3+2s) j} \frac{1-p^{-1+a_1-a_2}}{1-p^{a_3-a_2}} \frac{1-p^{-1+b_1-b_2}}{1-p^{b_3-b_2}}.   
\end{align*}
Since $
  \sum_{j=1}^{\infty} p^{-j \kappa} = \frac{p^{-\kappa}}{1-p^{-\kappa}}$, 
it follows from this and \eqref{CIJa1b1ps}  that
\begin{equation*}
\begin{split}
&  \mathcal{C}_{\I,\J; \{ a_1 \}, \{ b_1 \}}(p;s)
  =  \\
  &  \Bigg( 1 + \Big(
\frac{p^{-(1+a_2+b_2+2s)}}{1-p^{-(1+a_2+b_2+2s)}}
 \frac{1-p^{-1+a_1-a_3}}{1-p^{a_2-a_3}} \frac{1-p^{-1+b_1-b_3}}{1-p^{b_2-b_3}}  \\
& +  \frac{p^{-(1+a_2+b_3+2s) }}{1-p^{-(1+a_2+b_3+2s) }}
 \frac{1-p^{-1+a_1-a_3}}{1-p^{a_2-a_3}}\frac{1-p^{-1+b_1-b_2}}{1-p^{b_3-b_2}}  \\
& +  \frac{p^{-(1+a_3+b_2+2s) }}{1-p^{-(1+a_3+b_2+2s) }}
 \frac{1-p^{-1+a_1-a_2}}{1-p^{a_3-a_2}}\frac{1-p^{-1+b_1-b_3}}{1-p^{b_2-b_3}} \\
& + \frac{ p^{-(1+a_3+b_3+2s) }}{1-p^{-(1+a_3+b_3+2s) }}
 \frac{1-p^{-1+a_1-a_2}}{1-p^{a_3-a_2}} \frac{1-p^{-1+b_1-b_2}}{1-p^{b_3-b_2}}
\Big)   (1-p^{a_1+b_1+2s-1}) \Bigg) \\
& \times
  \Big(1- \frac{1}{p^{1+a_2+b_2+2s}} \Big)
   \Big(1- \frac{1}{p^{1+a_2+b_3+2s}} \Big)
 \Big(1- \frac{1}{p^{1+a_3+b_2+2s}} \Big)
 \Big(1- \frac{1}{p^{1+a_3+b_3+2s}} \Big)
\end{split}
\end{equation*}
and thus $
  \cC_{\I,\J; \{ a_1 \} , \{ b_1 \}}(s) = \prod_{p}  \mathcal{C}_{\I,\J; \{ a_1 \}, \{ b_1 \}}(p;s)$
where $ \mathcal{C}_{\I,\J; \{ a_1 \}, \{ b_1 \}}(p;s)$ is defined by \eqref{lps} and \eqref{Q}.  \\
(ii) We now establish \eqref{CIJexpansion}.
It is convenient to set
\begin{equation}
   \label{Pi}
    \Pi = (1 - X_2 Y_2 W)(1-X_2 Y_3 W)(1-X_3 Y_2 W)(1-X_3 Y_3 W).
\end{equation}
From \eqref{CIJa1b1ps} it follows that 
\begin{equation}
  \label{CIJTjfjPi}
  \mathcal{C}_{\I,\J; \{ a_1 \} , \{ b_1 \}}(p;s) = \Big( 1+\sum_{j=1}^{\infty} T_j f_j \Big) \Pi. 
\end{equation}
Observe that by Lemma \ref{GmultI3} we have 
\begin{equation}
\begin{split}
  \label{T2}
 T_2  =  & \Big( (X_2^2+X_2 X_3 + X_3^2)-( X_2^2 X_3 X_1^{-1}+X_3^2 X_2 X_1^{-1}  )U  \Big) 
    \cdot \Big( (Y_2^2+Y_2 Y_3 + Y_3^2)-( Y_2^2 Y_3 Y_1^{-1}+Y_3^2 Y_2 Y_1^{-1}  )U  \Big).
\end{split}
\end{equation}
Expanding out  $T_1$ \eqref{T1}, $T_2$ \eqref{T2}, $\Pi$ \eqref{Pi}, and
$f_1$ and $f_2$ \eqref{fjformula}, we find with the help of Maple  \footnote{ Maple file  available upon request.} that 
\begin{equation*}
\begin{split}
  (1+ T_1 f_1 +T_2 f_2) \Pi =  \sum_{(u,w) \in S} c_{uw}  U^{u} W^w  
\end{split}
\end{equation*}
where  $S \subset \mathbb{Z}_{\ge 0} \times \mathbb{Z}_{\ge 0}$ is finite, and $c_{uv}$ is a
polynomial in $X_{1}^{-1}, X_2, X_3, Y_1^{-1}, Y_2, Y_3$ that is homogeneous
of degree $2(u+w)$,
so that, by \eqref{variableboundsb}, one has
\begin{equation}
  \label{cuvbounds}
  |c_{uw}| \ll  p^{2(u+w)\delta},
\end{equation}
for all $(u,w) \in S$. 
 It may be checked that 
\begin{align}
  \label{cuvvalues}
 &  c_{00} = 1, \ c_{10}=c_{01} = 0, \  c_{02}  = -X_2X_3 Y_2 Y_3 = -p^{-a_2-a_3 -b_2-b_3}, \text{ and } \\
 \label{cuvvalues2}
  & c_{uw}=0 \text{ if either } u \ge 5 \text{ or } w \ge 7. 
\end{align}
Thus we have 
\begin{equation}
\begin{split}
  \label{jle2bounds}
    & (1+ T_1 f_1 +T_2 f_2) \Pi  =1 + c_{02}W^2 + \sum_{\substack{u \ge 2 \\ (u,0) \in S}} c_{u0} U^u + 
    \sum_{\substack{u \ge 1 \\ (u,0) \in S}} c_{u1} U^u W + \sum_{\substack{(u,w) \in S \\  w \ge 2, (u,w) \ne (0,2) }} c_{uw} U^u W^w.
\end{split}
\end{equation}
Observe that by \eqref{variables}, \eqref{variableboundsa}, and \eqref{cuvbounds} one has
\begin{equation}
   \label{uwterms}
    c_{uw} U^u W^w \ll p^{2(u+w) \delta-u-(1+2\sigma)w}
    = p^{-(1-2 \delta) u - (1+2 \sigma-2 \delta)w}
    \text{ for } u,w \ge 0. 
\end{equation}
Given that $S$ is finite, that $\sigma > -\frac{1}{2}+\delta$, and that we certainly have $\delta < \frac{1}{2}$, it follows from \eqref{cuvvalues}, \eqref{cuvvalues2}, \eqref{jle2bounds}, \eqref{uwterms},
and \eqref{variables} that
\begin{equation}
\begin{split}
   \label{firstparteq}
    (1+ T_1 f_1 +T_2 f_2) \Pi  & =1 + c_{02}W^2 + O \Big(   p^{-2(1-2 \delta)}   +    
    p^{-(1-2 \delta)-(1+2 \sigma-2\delta)}  
    +p^{-3(1+2\sigma-2 \delta)}  \Big)  \\
    & =1+c_{02} W^2 + O \Big( p^{6 \delta+ \max(-2,-2-2\sigma,-3-6 \sigma)} \Big) \\
     & =1- p^{-a_2-a_3 -b_2-b_3-2-4s} + O \Big( p^{6 \delta +\vartheta(\sigma)} \Big) 
\end{split}
\end{equation}    
where $\vartheta(\sigma)$ is defined by \eqref{thetasigma}.  Finally, we bound the contribution from $j \ge 3$ to 
\eqref{CIJTjfjPi}.  We make use of $|T_j| \le (p^{j \delta} d_2(p^{2j}))^2 =(2j+1)^2 p^{2 j \delta}
$, \eqref{fjformula}, and $|\Pi| \le (1+p^{-(1+2 \sigma-2 \delta)}  )^4 < 2^4  $  (since $\sigma > -\frac{1}{2}+\delta$)  to obtain 
\begin{equation}
\begin{split}
  \label{jg3bounds}
 &  \Big( \sum_{j=3}^{\infty}  T_j f_j  \Big) \Pi   \ll 
  \sum_{j=3}^{\infty}  (j+1)^2 p^{2j \delta} \Big(  \Big( \frac{1}{p^{1+2 \sigma}} \Big)^{j} +   \frac{p^{2 \delta}}{p^2} 
   \Big( \frac{1}{p^{1+2 \sigma}} \Big)^{j-1} \Big)  \\
   & \ll  p^{6 \delta } \Big(  \Big( \frac{1}{p^{1+2 \sigma}} \Big)^{3} +   \frac{p^{2 \delta}}{p^2} 
   \Big( \frac{1}{p^{1+2 \sigma}} \Big)^{2} \Big)  
       \le  p^{8 \delta} \Big(   \frac{1}{p^{3+6 \sigma}}   +  
    \frac{1}{p^{4+4 \sigma}} \Big)  
     \ll p^{8 \delta} p^{-\vartheta(\sigma)}
\end{split}
\end{equation}
for $\sigma \ge -\frac{1}{2}+\delta +\e$. 
Combining \eqref{CIJTjfjPi}, \eqref{firstparteq}, and \eqref{jg3bounds} we establish 
\eqref{CIJexpansion} along with \eqref{thetasigma}.  Note that from \eqref{thetasigma}, it follows that 
if $\sigma=\Re(s) = -\frac{1}{4}+\delta+\e$ with $\e >0$, then $ \mathcal{C}_{\I,\J; \{ a_1 \} , \{ b_1 \}}(p;s) 
= 1 + O(p^{-1-4 \e})$. Therefore $\mathcal{C}_{\I,\J; \{ a_1 \} , \{ b_1 \}}(s) $ is holomorphic and 
absolutely convergent for $\Re(s) > - \frac{1}{4} + \delta$.  Furthermore, we see that the poles listed 
in \eqref{Hpoles} in arise from the zeta factors in \eqref{Hfactorization}.
\end{proof}

\begin{proof}[Proof of Proposition \ref{ACidentities}]
\noindent (i)   In this proof we set $ x_{i}=p^{-a_i}$, $y_i=p^{-b_i}$ for $1 \le i \le 3$, and $u=p^{-1}$. 
We aim to show $
\cA_{\I_{\{a_1 \}}, \J_{ \{ b_1 \} }}(0) =  
  \cC_{\I,\J; \{ a_1 \} , \{ b_1 \} }(0)$. 
 Indeed, this will follow from 
 \begin{equation}
  \label{ApCplocal}
   \cA_{p;\I_{\{ a_1 \}}, \J_{\{ b_1 \}}}(0) =   \mathcal{C}_{p;\I,\J; \{ a_1 \} , \{ b_1 \}}(0)
\end{equation}
for $p$ prime. 
Note that $\I_{\{a_1 \}} = ( -b_1,a_2,a_3)$ and $\J_{ \{ b_1 \} } = ( -a_1,b_2,b_3)$.
We observe that
 $\cA_{\I_{\{a_1 \}}, \J_{ \{ b_1 \} }}(0)$ is obtained from
 $\cA_{\I,\J}(0)$ by the transformation
 $a_1 \to -b_1 \text{ and } b_1 \to -a_1. 
 $ 
 Therefore by \eqref{AIJ} and \eqref{ApIJ}
 \begin{align*}
   \cA_{\I_{\{a_1 \}}, \J_{ \{ b_1 \} }}(0) 
& = \prod_{p}
   P(p^{b_1},p^{-a_2},p^{-a_3},p^{a_1},p^{-b_2},p^{-b_3},p^{-1}) 
 = \prod_{p} P(y_{1}^{-1},x_2,x_3,x_{1}^{-1},y_2,y_3,u).
\end{align*}
On the other hand, by \eqref{CIJa1b1sdefn} and \eqref{lps}
\begin{align*}
    \mathcal{C}_{\I,\J; \{ a_1 \} , \{ b_1 \}}(0)
&  = \prod_{p} Q(p^{-a_2},p^{-a_3},p^{-b_2},p^{-b_3}; p^{-a_1},p^{-b_1};p^{-1},1) 
 = \prod_{p} Q(x_2,x_3,y_2,y_3;x_1,y_1;u,1).
\end{align*}
Therefore we see that \eqref{ApCplocal} is true 
 if $
  P(y_1^{-1},x_2,x_3,x_{1}^{-1},y_2,y_3,u) = Q(x_2,x_3,y_2,y_3;x_1,y_1;u,1)$ is true.
From the definitions \eqref{P} and \eqref{Q} this identity is equivalent to an identity of the form
$P_1(X_1,\ldots, X_7)=Q_1(X_1, \ldots, X_7)$ in which $P_1$ and $Q_1$ are certain polynomials
(the degrees and coefficients are computable).  Its verification is therefore a matter of 
routine (but lengthy) algebraic calculations: we find, by  a Maple calculation\footnote{Maple file available upon request.
  }, that it is indeed a valid identity. 
\\
(ii)  We will derive this from \eqref{ApCplocal}. 
An inspection of \eqref{P} and \eqref{Q} reveals that one has both 
\begin{equation*}
  P(X_1,X_2,X_3, Y_1, Y_2,Y_3; \Lambda^2 U) =
  P(\Lambda X_1, \Lambda X_2, \Lambda X_3, \Lambda Y_1, \Lambda Y_2, \Lambda Y_3; U)
\end{equation*}
and 
\begin{equation*}
   Q(X_2,X_3,Y_2,Y_3;X_1, Y_1; U, \Lambda^2 V) = 
   Q(\Lambda X_2, \Lambda X_3, \Lambda Y_2, \Lambda Y_3; \Lambda X_1, \Lambda Y_1; U,V),
\end{equation*}
whenever $\Lambda \ne 0$. We shall take $\Lambda=p^{\eta}$ where $p$ is an
arbitrary prime and $\eta$ some complex number (to be specified where needed), one deduces
via \eqref{ApIJ} and \eqref{lps}  that 
\begin{equation}
  \label{Aideta}
  \mathcal{A}_{p;\I,\J}(-2 \eta) = \mathcal{A}_{p;\I+\{-\eta\},\J+ \{ -\eta \}}(0),
\end{equation}
\begin{equation}
    \label{Cideta}
     \mathcal{C}_{\I,\J; \{a_1 \}, \{ b_1 \}}(p;- \tfrac{a_1+b_1}{2}) = 
      \mathcal{C}_{\I',\J'; \{a_1-\eta \}, \{ b_1-\eta \}}(p;0)
\end{equation}
(where for sets $U,V$, $U+V = \{ u+v \ : \ u \in U, v \in V \}$).  To deduce part (ii) of the 
Proposition, we set $\eta= \frac{a_1+b_1}{2}$ in \eqref{Aideta} and \eqref{Cideta}.
This gives us
\[
    \mathcal{A}_{p;\I,\J}(-a_1-b_1) = \mathcal{A}_{p;\I',\J'}(0)
    \text{ and } 
      \mathcal{C}_{\I,\J; \{a_1 \}, \{ b_1 \}}(p;-\eta) = 
      \mathcal{C}_{\I+\{-\eta\},\J+\{-\eta\}; \{a_1-\eta \}, \{ b_1-\eta \}}(p;0)
\]
where $\I'= \I + \{ -\eta \}$, $\J' = \J + \{ -\eta \}$, $a_1'=a_1-\eta= \frac{a_1-b_1}{2}$, $b_1'=b_1-\eta=\frac{b_1-a_1}{2}$. By these identities, along with \eqref{AIJ} and \eqref{CIJa1b1sdefn},
part (ii) requires that we establish $\mathcal{A}_{p;\I',\J'}(0)=
\mathcal{C}_{\I',\J'; \{ a_1' \}, \{ b_1' \}}(p;0)$.  Since $a_1'=-b_1'$, it follows that 
$( \I'_{ \{ a_1' \}}, \J'_{ \{ b_1' \}})=(\I', \J')$ and thus we must show 
 $\mathcal{A}_{p; \I'_{ \{ a_1' \}}, \J'_{ \{ b_1' \}}}(0)=
\mathcal{C}_{\I',\J'; \{ a_1' \}, \{ b_1' \}}(p;0)$. However, this is precisely
\eqref{ApCplocal}. \\
\noindent (iii)  For part (iii) we now set $\eta = \frac{a_2+b_2}{2}$.  We then have by \eqref{Cideta}
and \eqref{ApCplocal}
\begin{equation*}
  \mathcal{C}_{\I,\J; \{a_1 \}, \{ b_1 \}}(p;-\eta)= \mathcal{C}_{\I'',\J''; \{a_1'' \}, \{ b_1'' \}}(p;0)
  = \mathcal{A}_{p; \I''_{ \{ a_1'' \}}, \J''_{ \{ b_1'' \}}}(0),
\end{equation*}
where $\I'' = \I+ \{ - \eta \}$, $\J''= \J+ \{ - \eta \}$, $a_1''= a_1 - \eta$, and $b_1''=b_1-\eta$;
\begin{equation*}
   \mathcal{C}_{-\J,-\I; \{ - b_3 \}, \{ -a_3 \}}(p;\eta)
   = \mathcal{C}_{\I^{*}, \J^{*}; \{ a_3^{*} \}, \{ b_3^{*} \}}(p;0) =
   \mathcal{A}_{p; \I^{*}_{\{ a_3^{*} \}}, \J^{*}_{\{ b_3^{*} \}}}(0),
\end{equation*}
where $\I^{*} = (-\J)+ \{ \eta \}$, $\J^{*} = (-\J)+ \{ \eta \}$, $a_3^{*}=-b_3 + \eta$
and $b_3^{*}= - a_3 +\eta$.   By the last two identities, and \eqref{CIJa1b1sdefn}, we see 
that part (iii) follows if $(\I''_{ \{ a_1'' \}}, \J''_{ \{ b_1'' \}})= ( \I^{*}_{\{ a_3^{*} \}}, \J^{*}_{\{ b_3^{*} \}})$. 
However, observe that 
\begin{equation*}
\begin{split}
  (\I''_{ \{ a_1'' \}}, \J''_{ \{ b_1'' \}})& = \Big( \{-(b_1-\eta),a_2-\eta,a_3 -\eta \} ,  \{ -(a_1-\eta),b_2-\eta,
  b_3-\eta \} \Big), \\
  ( \I^{*}_{\{ a_3^{*} \}}, \J^{*}_{\{ b_3^{*} \}}) & = 
  \Big( \{ -b_1+\eta,-b_2+\eta,- (-a_3+\eta)\}, \{ -a_1+\eta,-a_2+\eta,- (-b_3+\eta) \} \Big).
\end{split}
\end{equation*}
Since $\eta = \frac{a_2+b_2}{2}$ we have both $a_2-\eta =-b_2+\eta$ and $b_2-\eta=-a_2+\eta$
and thus these two expressions are equal and the proof is complete. \\
(iv)  This follows from part (iii) by a permutation of variables. 
\end{proof}

\begin{proof}[Proof of Lemma \ref{Hbounds}]
(i)
By Proposition \ref{Hid}, we have for $\Re(s)=2 \delta$
\begin{equation*}
\begin{split}
  H_{\I,\J; \{ a_1\}, \{ b_1 \}}(s) &  \ll |\zeta(a_1+b_1+4 \delta +2iu)| \prod_{\substack{k_1 \ne 1 \\ k_2 \ne 1}} 
  |\zeta(1+a_{k_1}+b_{k_2}+4 \delta+2iu)|  \\
  & \ll  |\zeta( a_1+b_1+ 4 \delta +2iu)| \zeta^4(1+2 \delta)
\end{split}
\end{equation*}
for all real $u$.  However, it follows from \cite[Theorem 1.9]{Iv} that for $s=x+iy$ with $0 \le x \le 0.99$ 
and $y \in \mathbb{R}$
\[
  |\zeta(x+iy)| \ll (1+|y|)^{\frac{1-x}{2}} \log(2+|y|).
\]
Since $2 \delta \le \Re(a_1+b_1)+4 \delta \le 6\delta \le 0.99$ for $\delta \in (0, \tfrac{1}{7})$
and $2u -2 \delta \le 2u+\Im(a_1+b_1) \le 2u+2 \delta$ , it follows that 
\begin{equation*}
\begin{split}
 H_{\I,\J; \{ a_1\}, \{ b_1 \}}(s)   \ll  \delta^{-4} |s|^{\frac{1-2\delta}{2}} \log(1+|s|) \ll |s|^{\frac{1}{2}}. 
\end{split}
\end{equation*}
(ii) Since $\Re(s) \ge \frac{1}{2}-2 \delta$, it follows that $\Re(1+a_{k_1}+b_{k_2}+2s) \ge 2-6 \delta \ge 1.1$ as $\delta \in(0,\frac{1}{7})$ and thus 
$\prod_{\substack{k_1 \ne 1 \\ k_2 \ne 1}} \zeta(1+a_{k_1}+b_{k_2}+2s) \ll 1$. 
In addition, we have $\Im(s) \asymp T$ and thus by \eqref{zetabounds}
\[
  |\zeta(a_1+b_1+2s)| \ll T^{\frac{ 1-(1-6 \delta) }{2}} \log T \ll  T^{4 \delta}.
\]
Finally $ \mathcal{C}_{\I,\J; \{ a_1 \}, \{ b_1 \}}(s)  \ll 1$ in this region, and thus we 
establish the result.
\end{proof}

\section{Proofs of technical lemmas} \label{appendix2}

In this section we present the proofs of Parts (ii) and (iii) of Lemma \ref{Stirling}, 
and the proof of Lemma \ref{fpartials}.  The first of these proofs makes extensive use of 
Stirling's formula. 
\subsection{Proof of Lemma \ref{Stirling}(ii)}
\begin{proof}
Throughout this proof we assume that $|a_i|, |b_i| \le \delta$ for $i=1,2,3$ with $\delta \in (0,\frac{1}{6})$.  
This argument follows closely \cite[pp.390-391]{HB}. 
Let $\log z$ denote the principal branch of the logarithm, so that $-\pi < \Im(\log z) < \pi$ for 
$z \in \C \backslash (-\infty,0]$.  For each fixed $\eta >0$ we have 
\begin{equation}
  \label{Stirlingexpansion}
  \log \Gamma(z) =(z-\tfrac{1}{2}) \log z - z +\frac{1}{2} \log(2 \pi) + O(|z|^{-1})
\end{equation}
in the sector $|\arg z | \le \pi -\eta$. 
Throughout this argument we shall assume that
\begin{equation*}
  \label{alphabetaa}
 \alpha = \frac{1}{2} \Big(\frac{1}{2}+a+it \Big), \ 
 \beta=\frac{s}{2}, \text{ and }
 |a| \le \delta.
\end{equation*}
 We begin by supposing that $|\Im(s)| \le t^{\frac{1}{2}}$. 
Note that \eqref{Stirlingexpansion} implies
\begin{align*}
  \log \Gamma(\alpha+\beta) -\log \Gamma(\alpha) 
  =\beta \log(\alpha) + (\alpha+\beta-\tfrac{1}{2}) \log(1+\beta/\alpha)-\beta +O(t^{-1}). 
\end{align*}
Also
\[ 
   (\alpha+\beta-\tfrac{1}{2}) \log(1+\beta/\alpha)-\beta
   = (\alpha+\beta-\tfrac{1}{2})  \Big( \frac{\beta}{\alpha} + O \Big(  \Big(\frac{|\beta|}{|\alpha|}\Big)^2 \Big) \Big) - \beta
   = -\frac{\beta}{2 \alpha} +O \Big( \frac{|\beta|^2}{|\alpha|} \Big)
\]
and thus $\log \Gamma(\alpha+\beta) -\log \Gamma(\alpha) 
  = \beta \log(\alpha)+O ( \frac{|s|^2+1}{t} )$.   It follows that 
\begin{equation}
   \label{loggamma1}
    \log \Gamma \Big( \frac{1}{2} \Big(\frac{1}{2}+a+it +s \Big) \Big)
    -\log \Gamma \Big( \frac{1}{2} \Big(\frac{1}{2}+a+it \Big)  \Big)= \frac{s}{2} \log (  \tfrac{it}{2} ) + O \Big( \frac{|s|^2+1}{t} \Big). 
\end{equation}
Conjugating the above equation and replacing $a$ by $\overline{a}$, and $s$
by $\overline{s}$, yields
\begin{equation}
     \label{loggammaconjugate}
    \log \Gamma \Big( \frac{1}{2} \Big(\frac{1}{2}+a-it +s \Big) \Big)
    -\log \Gamma \Big( \frac{1}{2} \Big(\frac{1}{2}+a-it \Big)  \Big)= \frac{s}{2} \log (  -\tfrac{it}{2} ) + O \Big( \frac{|s|^2+1}{t} \Big). 
\end{equation}
Taking $a=a_j$ in \eqref{loggamma1} and $a=b_j$ in \eqref{loggammaconjugate} we find 
\begin{equation}
\begin{split}
    & \log \Gamma \Big( \frac{1}{2} \Big(\frac{1}{2}+a_j+it +s \Big) \Big)
    -\log \Gamma \Big( \frac{1}{2} \Big(\frac{1}{2}+a_j+it \Big)  \Big)+
     \log \Gamma \Big( \frac{1}{2} \Big(\frac{1}{2}+b_j-it +s \Big) \Big)
    -\log \Gamma \Big( \frac{1}{2} \Big(\frac{1}{2}+b_j-it \Big)  \Big) \\
    & =  s \log \Big(  \frac{t}{2}   \Big) + O \Big( \frac{|s|^2+1}{t} \Big). 
\end{split}
\end{equation}
Exponentiating and taking the product over $j =1,2,3$ yields 
\[
   g_{\I,\J}(s,t) = \prod_{j=1}^{3} \Big( \frac{t}{2} \Big)^{s} \Big(1 + O \Big( \frac{|s|^2+1}{t} \Big)\Big)
   = \Big( \frac{t}{2} \Big)^{3s} \Big(1 + O \Big( \frac{|s|^2+1}{t} \Big)\Big),
\]
since $0 \le \Re(s) \le A$ and $|\Im(s)| \le t^{\frac{1}{2}}$. 

Next we deal with the case $|\Im(s)| > t^{\frac{1}{2}}$.  In this range,  note that the $O$ term becomes larger
than 1.  Thus it suffices to establish that, for  $s=\sigma+iy$, $0 \le \sigma \le A$, and $|y| > t^{\frac{1}{2}}$
one has
\begin{equation}
  \label{gIJlargebd}
    g_{\I,\J}(s,t)  \ll_{A} y^{2} t^{3 \sigma-1+3\delta}. 
\end{equation}
We shall use repeatedly the Stirling estimate:  for $0 \le x \ll 1$ and $|y| \ge 1$, 
\begin{equation}
  \label{absoluteStirling}
  |\Gamma(x+iy)|  = (2 \pi)^{\frac{1}{2}}  |y|^{x-\frac{1}{2}} e^{-\frac{\pi |y|}{2}} ( 1 + O(|y|^{-1})). 
\end{equation}
Thus if $|y+t| \ge 1$, then, when $a \in \{ a_1, a_2, a_3 \}$,
\begin{align*}
   \Bigg| \frac{\Gamma ( \frac{1}{2} (\frac{1}{2}+a+it +s ) )}{ \Gamma ( \frac{1}{2} (\frac{1}{2}+a+it )  ) } \Bigg| &  \asymp
   \frac{ | \frac{t+a''+y}{2}|^{\frac{a'+\sigma}{2}-\frac{1}{4} }
   e^{-\frac{\pi}{4}|t+y+a''|  } }{|\frac{t+a''}{2}|^{-\frac{1}{4}} e^{-\frac{\pi (t+a'')}{4}  }} 
   \asymp 
    \frac{ |t+a''+y|^{\frac{a'+\sigma}{2}-\frac{1}{4} }
   e^{-\frac{\pi}{4}|t+y+a''|  } }{t^{-\frac{1}{4}} e^{-\frac{\pi t}{4}  }} 
\end{align*}
where $a=a'+i a''$.   Similarly, if $|y-t| \ge 1$ and $b \in \{ b_1, b_2, b_3 \}$, then 
\begin{equation*}
\begin{split}
      \Bigg| \frac{\Gamma ( \frac{1}{2} (\frac{1}{2}+b-it +s ) )}{ \Gamma ( \frac{1}{2} (\frac{1}{2}+b-it ) ) } 
       \Bigg| & \asymp
        \frac{  | b''+y-t |^{\frac{b'+\sigma}{2}-\frac{1}{4} }
   e^{-\frac{\pi}{4}|b''+y-t|  }}{|t|^{-\frac{1}{4}} e^{-\frac{\pi t}{4}  }}.
\end{split}
\end{equation*}
Thus if $|y-t| \ge 1$ and $|y+t| \ge 1$, these combine to give 
\begin{equation*}
\begin{split}
      \Bigg| \frac{\Gamma ( \frac{1}{2} (\frac{1}{2}+a+it +s ) )\Gamma ( \frac{1}{2} (\frac{1}{2}+b-it +s ) )}{ \Gamma ( \frac{1}{2} (\frac{1}{2}+a+it )  ) \Gamma ( \frac{1}{2} (\frac{1}{2}+b-it )  ) } 
       \Bigg| & \ll
        \frac{ |t+a''+y|^{\frac{a'+\sigma}{2}-\frac{1}{4} }  | b''+y-t |^{\frac{b'+\sigma}{2}-\frac{1}{4} }
      e^{-\frac{\pi}{4}( |t+y+a''|+|b''+y-t|)  } }{(|t|^{-\frac{1}{4}} e^{-\frac{\pi t}{4}  })^2} \\
      &  \ll
        \frac{ |t+y|^{\frac{a'+\sigma}{2}-\frac{1}{4} }  | y-t |^{\frac{b'+\sigma}{2}-\frac{1}{4} }
      e^{-\frac{\pi}{4}( |t+y|+|y-t|)  } }{|t|^{-\frac{1}{2}} e^{-\frac{\pi t}{2}  }}.
\end{split}
\end{equation*} 
In the case that $|y-t| \ge 1$ and $|y+t| \ge 1$, we apply this with 
 $a=a_j=a_j'+ia_j''$ and $b=b_j=b_j'+ib_j''$ to find that
\begin{equation}
\begin{split}
 |g_{\I,\J}(s,t)| & \ll 
  \prod_{j=1}^{3}
    \frac{ |y+t|^{\frac{a_j'+\sigma}{2}-\frac{1}{4}} |y-t|^{\frac{b_j'+\sigma}{2}-\frac{1}{4}} e^{-\frac{\pi}{4}( |t+y|+|y-t|)  } }{ t^{-\frac{1}{2}} 
   e^{-\frac{\pi t}{2}} }  \\
 \label{gIJbd2}
   & \ll  \prod_{j=1}^{3}
    \frac{ |y+t|^{\frac{\sigma}{2}-\frac{1}{4}+\frac{\delta}{2}} |y-t|^{\frac{\sigma}{2}-\frac{1}{4}+\frac{\delta}{2}} e^{-\frac{\pi}{4}( |t+y|+|y-t|)  } }{ t^{-\frac{1}{2}} 
   e^{-\frac{\pi t}{2}} }, 
\end{split}
\end{equation}
since we are assuming the bound \eqref{sizerestrictiondelta} for $\I$ and $\J$.
It follows (given that $\delta$ is positive) that if $y \ge t+1$ then
\begin{align*}
     |g_{\I,\J}(s,t)| & \ll (t^{\frac{1}{2}} (y+t)^{\frac{\sigma}{2}-\frac{1}{4}+\frac{\delta}{2}} 
     (y-t)^{\frac{\sigma}{2}-\frac{1}{4}+\frac{\delta}{2}} e^{-\frac{\pi}{2}(y-t)} )^3
     \ll t^{\frac{3}{2}} (r+t)^{\frac{3\sigma}{2}-\frac{3}{4}+\frac{3 \delta}{2}} 
     r^{\frac{3\sigma}{2}-\frac{3}{4}+ \frac{3\delta}{2}} e^{-\frac{3\pi r}{2}},
\end{align*}
where we have used that $y+t \asymp y$, and that $0 \le \sigma \le A$, and have made the variable
change $r=y-t$. By considering separately, the cases $1 \le r \le t$ and $r > t$, 
we obtain that $|g_{\I,\J}(s,t)| \ll (r+t)^2 t^{3\sigma-1}$
which establishes  \eqref{gIJlargebd} in the case $y \ge t+1$. 
For $y \in [\sqrt{t},t-1]$ and $\frac{1}{2} \le \sigma  \le A$, we have by \eqref{gIJbd2}
\begin{align*}
  |g_{\I,\J}(s,t)| & \ll 
   \Big( \frac{ (y+t)^{\frac{\sigma}{2}-\frac{1}{4}+\frac{\delta}{2}} (t-y)^{\frac{\sigma}{2}-\frac{1}{4}+\frac{\delta}{2}}e^{-\frac{\pi t}{2}} }{ t^{-\frac{1}{2}} 
   e^{-\frac{\pi t}{2}} } \Big)^3
   \ll \Big( \frac{t^{\sigma-\frac{1}{2}+\delta}}{t^{-\frac{1}{2}}} \Big)^3 = t^{3(\sigma+\delta)} \ll y^2 t^{3\sigma-1+3 \delta},
\end{align*}
since $y > \sqrt{t}$.  
 For $y \in [\sqrt{t},t-1]$ and $\sigma \in [0,\frac{1}{2})$, we obtain 
\begin{align} 
 \nonumber |g_{\I,\J}(s,t)| & \ll 
   \Big( \frac{ (y+t)^{\frac{\sigma}{2}-\frac{1}{4}+ \frac{\delta}{2}} (t-y)^{\frac{\sigma}{2}-\frac{1}{4}+\frac{\delta}{2}}e^{-\frac{\pi t}{2}} }{ t^{-\frac{1}{2}} 
   e^{-\frac{\pi t}{2}} } \Big)^3 \ll 
   \Big( \frac{t^{\frac{\sigma}{2}-\frac{1}{4}+\frac{\delta}{2}}(t-y)^{\frac{\sigma}{2}-\frac{1}{4}+\frac{\delta}{2}}}{t^{-\frac{1}{2}}} \Big)^3   \\
  \label{lastinequality} & \ll t^{\frac{3\sigma}{2}+\frac{3}{4}+\frac{3 \delta}{2}} (t-y)^{\frac{3 \sigma}{2}-\frac{3}{4}+\frac{3 \delta}{2}}
   \ll y^{2} t^{3\sigma-1+3 \delta}. 
\end{align} 
In the case that $\sigma \in [0,\frac{1}{2})$ this can be checked by considering the function $h(y) =y^2 (t-y)^{\frac{3}{4}-\frac{3\sigma}{2}}$
on the interval $[\sqrt{t},t-1]$.  
Elementary calculus shows that the minimum of $h$ on this interval is $\gg t^{\frac{7}{4}-\frac{3\sigma}{2}}$
and therefore \eqref{lastinequality} follows.  
Now, if $y \in [t-1,t+1]$ then, for $|b| \le \delta$, one has 
$| \frac{\Gamma ( \frac{1}{2} (\frac{1}{2}+b-it +s ) )}{ \Gamma ( \frac{1}{2} (\frac{1}{2}+b-it ) ) }| 
\ll_{A} \frac{1}{ |\Gamma ( \frac{1}{2} (\frac{1}{2}+b-it ) )| }  
\ll t^{\frac{1}{4}+\frac{\delta}{2}} e^{\frac{\pi t}{4}} $ 
(by \eqref{absoluteStirling}) and thus 
\begin{equation*}
\begin{split}
      \Bigg| \frac{\Gamma ( \frac{1}{2} (\frac{1}{2}+a+it +s ) )\Gamma ( \frac{1}{2} (\frac{1}{2}+b-it +s ) )}{ \Gamma ( \frac{1}{2} (\frac{1}{2}+a+it )  ) \Gamma ( \frac{1}{2} (\frac{1}{2}+b-it )  ) } 
       \Bigg| & \ll
        \frac{ | t+a''+y|^{\frac{\sigma}{2}-\frac{1}{4} +\frac{\delta}{2} }
   e^{-\frac{\pi}{2}|\frac{t+y+a''}{2}|    }}{ |t|^{-\frac{1}{4}} e^{-\frac{\pi t}{4}  }} 
   \cdot (t^{\frac{1}{4}+\frac{\delta}{2}} e^{\frac{\pi t}{4}})
   \ll t^{\frac{\sigma}{2} +\frac{1}{4}+\delta}
\end{split}
\end{equation*} 
for $|a|, |b| \le \delta$. 
Therefore $ |g_{\I,\J}(s,t)|  \ll (t^{\frac{\sigma}{2}+\frac{1}{4}+\delta})^3\ll y^2 t^{3 \sigma-1+3 \delta}$, since
$y \in [t-1,t+1]$ and $\sigma \ge 0$.  The cases for $y \le -\sqrt{t}$ are proven in a similar fashion. 
 \end{proof}

\subsection{Partial derivative bounds. Proof of Lemma \ref{Stirling} (iii)}

\begin{proof}
For $1 \le j \le 3$,  let $p_j(s,t)=\frac{\Gamma ( \frac{1}{2} (\frac{1}{2}+a_j+it +s ) )\Gamma ( \frac{1}{2} (\frac{1}{2}+b_j-it +s ) )}{ \Gamma ( \frac{1}{2} (\frac{1}{2}+a_j+it )  ) \Gamma ( \frac{1}{2} (\frac{1}{2}+b_j-it )  )}$, 
$  \theta_j(s,t) = \frac{d}{dt} \log p_j(s,t)$, and 
$\Theta(s,t) = \sum_{j=1}^{3} \theta_j(s,t)$.
Observe that 
\begin{equation}
    \label{gstderivative}
     \frac{d}{dt} g_{\I,\J}(s,t) = g_{\I,\J}(s,t) \Theta(s,t) 
\end{equation}
and more generally, for $i \ge 1$, 
\begin{equation}
   \label{gstithderivative}
   \frac{d^i}{dt^i} g_{\I,\J}(s,t) = \sum_{u+v=i-1} \binom{i-1}{u}   
   \frac{d^u}{dt^u} g_{\I,\J}(s,t) 
    \frac{d^v}{dt^v} \Theta(s,t).
\end{equation}
We shall now demonstrate that 
\begin{equation}
   \label{thetajderivatives}
      \frac{d^v}{dt^v} \theta_j(s,t)   \ll |s|t^{-v-1} \text{ for } j=1,2,3. 
\end{equation}
From this it follows directly that one has 
\begin{equation}
   \label{Thetaderivatives} 
        \frac{d^v}{dt^v} \Theta(s,t)   \ll |s|t^{-v-1}.
\end{equation}
Using these facts, we can prove the Lemma by induction. 
Observe that Lemma \ref{Stirling} (ii),  \eqref{Thetaderivatives} with $v=0$, and \eqref{gstderivative} imply
$\frac{d}{dt} g_{\I,\J}(s,t) \ll |s| T^{3 \varepsilon+3 \delta-1}$.  
This, together with Lemma  \ref{Stirling} (ii), establishes the Lemma in the cases $i=0$ and $i=1$.  
Now assume the inductive hypothesis that, for some $j \in \mathbb{N}$ one has $\frac{d^u}{dt^u} g_{\I,\J}(s,t) \ll |s|^u T^{3 \varepsilon+3 \delta-u}$ for $u \le j-1$. 
Combining this with \eqref{Thetaderivatives}  and \eqref{gstithderivative}, 
we obtain the case $i=j$ of the bound in \eqref{gstiderivatives}, and so (by induction) we find that 
\eqref{gstiderivatives} holds for
for all $i \ge 0$, $\Re(s)=\varepsilon$, and $|\Im(s)| \le \sqrt{T}$.  
Thus to complete our proof of Lemma \ref{Stirling} (iii) we must establish \eqref{thetajderivatives}. 
Note that 
\begin{align*}
  \theta_j(s,t)  = \frac{i}{2} & \Bigg(  
  \frac{  \Gamma' ( \frac{1}{2} (\frac{1}{2}+a_j+it +s ) )}{ \Gamma ( \frac{1}{2} (\frac{1}{2}+a_j+it+s )  )}
  -    \frac{  \Gamma' ( \frac{1}{2} (\frac{1}{2}+a_j+it  ) )}{ \Gamma ( \frac{1}{2} (\frac{1}{2}+a_j+it )  )}
  -  
  \frac{  \Gamma' ( \frac{1}{2} (\frac{1}{2}+b_j-it +s ) )}{ \Gamma ( \frac{1}{2} (\frac{1}{2}+b_j-it+s )  )}
  +    \frac{  \Gamma' ( \frac{1}{2} (\frac{1}{2}+b_j-it  ) )}{ \Gamma ( \frac{1}{2} (\frac{1}{2}+b_j-it )  )}
  \Bigg).
\end{align*}
However, we have the asymptotic expansion 
$\frac{\Gamma'}{\Gamma}(z) = \log z + O ( |z|^{-1} )$
and thus, when $y=\Im(s)$, 
\begin{align*}
    \theta_j(s,t) &  =  \frac{i}{2} \Big(- \log \Big(  \frac{1}{2} \Big(\frac{1}{2}+a_j+it \Big)  \Big)
    + \log \Big(  \frac{1}{2} \Big(\frac{1}{2}+a_j+it+s \Big) \Big)  \\
    & - \log \Big(  \frac{1}{2} \Big(\frac{1}{2}+b_j-it +s \Big)  \Big)
    + \log \Big(  \frac{1}{2} \Big(\frac{1}{2}+b_j-it \Big) \Big) \Big) +
     O(t^{-1})  \\
    & = \frac{i}{2} ( \log(t+y) - \log(t) - \log |y-t| + \log(t))+ O(t^{-1})  \\
   &   =\frac{i}{2} \Big( \log \Big(1+\frac{y}{t} \Big)  -\log \Big( 1- \frac{y}{t} \Big) \Big)+  O(t^{-1}) 
    \ll  \frac{1+|y|}{t}  \ll \frac{|s|}{t}
\end{align*}
since $|y| \le \sqrt{T}$ and $\Re(s) = \e>0$. We now study the derivatives of $\theta_j$. 
We have 
\begin{align*}
  \frac{d^v}{dt^v} \theta_j(s,t) =  \Big( \frac{i}{2} \Big)^{v+1}  & \Big( 
   \Big( \frac{\Gamma'}{\Gamma}\Big)^{(v)} \Big( \frac{1}{2} \Big(\frac{1}{2}+a_j+it +s \Big) \Big)  -    
    \Big( \frac{\Gamma'}{\Gamma}\Big)^{(v)} \Big( \frac{1}{2} \Big(\frac{1}{2}+a_j+it \Big) \Big) 
    \Big)
       \\
     & +  \Big( \Big( \frac{\Gamma'}{\Gamma}\Big)^{(v)}  \Big( \frac{1}{2} \Big(\frac{1}{2}+b_j-it +s \Big) \Big)  -    
    \Big( \frac{\Gamma'}{\Gamma}\Big)^{(v)} \Big( \frac{1}{2} \Big(\frac{1}{2}+b_j-it  \Big) \Big)  \Big)
     \Big( \frac{-i}{2} \Big)^{v+1}
    . 
\end{align*} 
It is known  (see \cite{AS})  that
$
  ( \frac{\Gamma'}{\Gamma} )^{(v)} (z) = \frac{(-1)^{v-1} (v-1)!}{z^v} + O 
  (\frac{1}{|z|^{v+1}} )$ for $\Re(z) > 0$ 
and thus
\begin{equation}
\begin{split}
   \label{dnuthetaj}
    \frac{d^v}{dt^v} \theta_j(s,t) = (-1)^{v-1} (v-1)! \Big( \frac{i}{2} \Big)^{v+1}  & \Bigg( 
    \frac{1}{ (\frac{1}{2} (\frac{1}{2}+a_j+it +s ) )^v} - \frac{1}{( \frac{1}{2} (\frac{1}{2}+a_j+it ) )^v} \\
    & + \frac{1}{( \frac{1}{2} (\frac{1}{2}+b_j-it +s ) )^v} - \frac{1}{( \frac{1}{2} (\frac{1}{2}+b_j-it  ) )^v}
    \Bigg) + O(t^{-v-1}). 
\end{split}
\end{equation}
Note that  $\alpha_j = \frac{\frac{1}{2} +a_j+it }{\frac{1}{2} +a_j+it +s }$ satisfies 
\[
  |\alpha_j^{-1} -1| = \frac{|s|}{|\frac{1}{2}+a_j+it|}
  < \frac{|s|}{t-\delta} \le \frac{ \sqrt{\varepsilon^2+T}}{t-\delta}
  < \frac{c_1}{ \sqrt{T}} < \frac{1}{2},
\]
where $c_1 >0$ and the last inequality holds 
when $T$ is sufficiently large.  This implies that $|\alpha_j^{-1}| > \frac{1}{2}$ 
and also $|1-\alpha_j| < \frac{4|s|}{t} < 1$, for $T$ sufficiently large.  From these inequalities and the binomial theorem
it follows that $\alpha_{j}^{\nu} -1 \ll_{\nu} \frac{|s|}{t}$ for $\nu \in \mathbb{N}$.   
Multiplying this last inequality by $ 2^{\nu}(\frac{1}{2}+a_j+it)^{-\nu} $ and using 
$|\frac{1}{2}+a_j+it| > t/2$, we obtain that the first two terms in the
brackets in  \eqref{dnuthetaj} sum to $\ll_{\nu} \frac{|s|}{t^{\nu+1}}$.  The final two terms in the brackets
may be treated similarly.  We may therefore deduce from  \eqref{dnuthetaj}  that 
$ \frac{d^v}{dt^v} \theta_j(s,t) \ll_{\nu}  \frac{|s|+1}{t^{\nu+1}}$. Since this shows that 
\eqref{Thetaderivatives} holds (given that we have $\Re(s) =\e >0$), the proof is now complete.
\end{proof}
\subsection{Proof of Lemma \ref{fpartials}}
Next we provide a proof of Lemma \ref{fpartials}. 
We aim to establish the following result: 
Let  $M \ll T^{\frac{3}{2}+\varepsilon}$, $N \asymp M$, $0 \ne r \ll \frac{M}{T_0} T^{\varepsilon}$, 
and $(x,y) \in [M,2M] \times [N,2N]$.  
Then 
\begin{equation}
   \label{fpartialsbd}
  x^m y^n f_{r}^{(m,n)}(x,y) \ll  T^{\frac{3}{2} \varepsilon} P^{n}
   \text{ where }
     P = \Big( \frac{T}{T_0} \Big) T^{\varepsilon}.
\end{equation}
\begin{proof} Recall that $f_r$ is defined by \eqref{fMN}.  That is,  $f_r(x,y) = W( \frac{x}{M} )
W ( \frac{y}{N} ) \phi(x,y)$
where
\[
 \phi(x,y) = 
\frac{1}{2 \pi i} \int_{(\varepsilon)} \frac{G(s)}{s} 
\Big( \frac{1}{\pi^3 xy} \Big)^s 
\frac{1}{T} \int_{-\infty}^{\infty} \Big( 
1+\frac{r}{y} \Big)^{-it} g_{\I, \J}(s,t) w(t) dt \, ds
\]
when $x,y>0$ (and is otherwise equal to $0$). 
In order to prove this lemma it suffices to establish
\begin{equation}
    \label{phipartialsbd}
      x^m y^n \phi^{(m,n)}(x,y) \ll  T^{\frac{3}{2} \varepsilon} P^{n} \text{ for } x \asymp M, y \asymp N \text{ where, }
     P = \Big( \frac{T}{T_0} \Big) T^{\varepsilon}.
\end{equation}
Indeed, it follows from \eqref{phipartialsbd} (via two applications of the generalized product rule) that one has:
\begin{equation*}
\begin{split}
  |f_{r}^{(m,n)}(x,y)| & =  \Big| \sum_{i_1+i_2=m} \binom{m}{i_1} W^{(i_1)}\Big( \frac{x}{M} \Big) M^{-i_1}
  \sum_{j_1+j_2=n}  \binom{n}{j_1} W^{(j_1)}\Big( \frac{y}{N} \Big) N^{-j_1} \phi^{(i_2,j_2)}(x,y) \Big| \\
  & \le  \sum_{i_1+i_2=m} 2^m O_{i_1}(1)  \Big( \frac{x}{2} \Big)^{-i_1} 
  \sum_{j_1+j_2=n} 2^n O_{j_1}(1)  \Big( \frac{y}{2} \Big)^{-j_1} \cdot x^{-i_2} y^{-j_2} \cdot
  O_{i_2,j_2}(T^{\frac{3}{2} \varepsilon} P^{j_2}) 
 \\                                                                                                                                                                                                              & =                                                                                                                                                                                                              \Big(  \sum_{i_1=0}^{m} \sum_{j_1=0}^{n} O_{i_1,j_1,m,n}(1) \cdot  P^{-j_1} \Big)
 T^{\frac{3}{2} \varepsilon}  P^{n}    x^{-m} y^{-n}
\end{split}
\end{equation*} 
where we have made use of $W(u)$ being zero when $u \ge 2$.
Since $P \ge 1$ (as follows by \eqref{cond3}, for all sufficiently large $T$),  this bound for $f_{r}^{(m,n)}(x,y)$ implies that in \eqref{fpartialsbd}.
We have now reduced the proof of the lemma to establishing \eqref{phipartialsbd}. 
We begin with 
\begin{equation}
  \label{phipartials}
  \phi^{(m,n)}(x,y) =
  \frac{1}{2 \pi i} \int_{(\varepsilon)} \frac{G(s)}{s} 
\Big( \frac{1}{\pi^3} \Big)^s 
\frac{1}{T} \int_{-\infty}^{\infty} 
   \frac{\partial^m}{\partial x^m} \frac{\partial^n}{\partial y^n}  \Big(x^{-s} y^{-s} \Big( 
1+\frac{r}{y} \Big)^{-it} \Big)  g_{\I,\J}(s,t) \omega(t) dt \, ds. 
\end{equation}
It suffices to compute the partial derivatives in the last inner integral. 
Observe that
\begin{equation}
  \label{Pks}
   \frac{d^k}{dx^k} x^{-s} = P_k(s) x^{-s-k} \text{ where }
  P_k(s) = 
  \begin{cases}
  1 & \text{if } k=0,  \\
  \prod_{j=0}^{k-1}(-s-j) & \text{if } k \in \mathbb{N}. 
  \end{cases}
\end{equation}
It suffices to determine $\frac{d^n}{dy^n}  y^{-s} ( 
1+\frac{r}{y} )^{-it}$ for $n \ge 0$.
It may be shown by a straightforward induction proof that 
\[
   \frac{d^k}{dy^k} \Big( 1 + \frac{r}{y} \Big)^{-it}
   = \Big( 1 + \frac{r}{y} \Big)^{-it} (y+r)^{-k} \sum_{m=1}^{k} 
   \binom{k}{m} P_{k-m}(m) P_m(-it) \Big( \frac{r}{y} \Big)^{m}.
\]
For $k \ge 1$, one gets (when  $y \asymp N \asymp M$, $t \asymp T$, and 
$1 \le |r|  \ll \frac{M}{T_0} T^{\varepsilon} =o(M)$ and $P = ( \frac{T}{T_0} ) T^{\varepsilon}
\ge 1$):
\begin{equation}
\begin{split}
  \label{derivativeoneplusryit}
\Big| \frac{d^k}{dy^k} \Big( 1 + \frac{r}{y} \Big)^{-it} \Big| & 
\le |y(1+o(1)|^{-k} \sum_{m=1}^{k} 2^k ((k-1)!) \cdot O_m (T^m) \cdot 
O_m \Big( \Big( \frac{P}{T} \Big)^m \Big)  \\
& \ll_k y^{-k} P^k \cdot \sum_{m=1}^{k} O_m(1) \ll_{k} \Big( \frac{P}{y} \Big)^k. 
\end{split}
\end{equation}
In the case $k=0$ we have the same estimate, since $|u^{-it}|=1 = (P/y)^{0}$ for $u >0$.  
By the generalized product rule,  \eqref{Pks}, and \eqref{derivativeoneplusryit}
\begin{equation*}
\begin{split}
 \frac{d^n}{dy^n} y^{-s} \Big( 1 + \frac{r}{y} \Big)^{-it}
 & = \sum_{j+k=n} \binom{n}{j} y^{-s-j} P_j(s)   \frac{d^k}{dy^k} \Big( 1 + \frac{r}{y} \Big)^{-it} 
 \ll \sum_{j+k=n}  \binom{n}{j}  N^{-\varepsilon-j} |s|^{j}  \Big( \frac{P}{y} \Big)^k \\
 & = N^{-\varepsilon} \Big( \frac{|s|}{N} + \frac{P}{y}  \Big)^n \ll N^{-n-\varepsilon}
 (|s|+P)^n
\end{split}
\end{equation*}
when $\Re(s) = \varepsilon$, $y \asymp N$, $t \asymp T$,  
$1 \le |r|  \ll \frac{M}{T_0} T^{\varepsilon} =o(M)$ and $P = ( \frac{T}{T_0} ) T^{\varepsilon}
\ge 1$, by \eqref{cond3}. 
By \eqref{Pks} it follows that 
\begin{equation*}
\begin{split}
   \frac{\partial^m}{\partial x^m} \frac{\partial^n}{\partial y^n}  \Big(x^{-s} y^{-s} \Big( 
1+\frac{r}{y} \Big)^{-it} \Big)  & \ll 
x^{-\varepsilon-m} |s|^m \cdot N^{-n-\varepsilon} (P+|s|)^n 
\ll M^{-2 \varepsilon} x^{-m} y^{-n} |s|^m(P+|s|)^n  \\
& \ll M^{-2 \varepsilon} (1+|s|)^{m+n} P^n x^{-m} y^{-n}. 
\end{split}
\end{equation*}
Using this,  along with $\Big|\frac{G(s)}{s}  \Big| \le \frac{|G(s)|}{\varepsilon}$ for $\Re(s)=
\varepsilon$, and 
\[
  |g_{\I,\J}(s,t)| \ll \Big( \frac{t}{2} \Big)^{3 \varepsilon} (1 + O(|s|^2+1))
  \ll (1+|s|)^2  t^{3 \varepsilon} \text{ for } \Re(s) = \varepsilon \text{ and }
  t \asymp T > 1
\]
(which follows trivially from Lemma \ref{Stirling} (ii))
we may bound \eqref{phipartials}.  The three bounds,
and \eqref{cond2}, \eqref{cond3},
 combine to give
\begin{equation*}
\begin{split}
      x^m y^n \phi^{(m,n)}(x,y) &  \ll_{\varepsilon}
      P^n  M^{-2 \varepsilon}
      \int_{(\varepsilon)} |G(s)|  \Big( \frac{1}{T} 
      \int_{-\infty}^{\infty}  (1+|s|)^{m+n+2}  T^{3 \varepsilon} |\omega(t)| dt \Big)   |ds| \\
      & \ll \Big( \frac{T^3}{M^2} \Big)^{\varepsilon} P^n 
      \text{ for } x \asymp M, y \asymp N \text{ and } m,n \in \mathbb{N} \cup \{ 0 \}. 
\end{split}
\end{equation*}
Finally, observe that $M \gg \frac{T_0}{T^{\varepsilon}} |r| \ge \frac{T_0}{T^{\varepsilon}} 
> T^{\frac{3}{4}}$ (by \eqref{cond1}). Hence $\frac{T^3}{M^2} \ll T^{3-\frac{3}{2}}
=T^{\frac{3}{2}}$ and we establish Lemma \ref{fpartials}. 
\end{proof}


\section*{Acknowledgments} 
 I wish to express my deepest gratitude to the referee for  their 
 very  careful and meticulous reading of this manuscript.    
 Most significantly, the referee corrected Proposition 6.1
 and provided alternate proofs of Lemma \ref{fpartials} and Proposition \ref{ACidentities} (ii), (iii).  
 Furthermore, they provided  extensive comments, simplifications, and improvements.  Thank-you to Alia Hamieh,
 Quanli Shen, and Peng-Jie Wong for helpful discussions and for corrections. 
Also, thank-you to Olga Balkanova, Brian Conrey, Aleksandar Ivi\'{c}, Jon Keating, and Kannan Soundararajan for  
comments and references. 

\bibliographystyle{amsplain}

\begin{thebibliography}{99}


\bibitem{AS}
M. Abramowitz and I.A. Stegun,  
\newblock
 Handbook of mathematical functions with formulas, graphs, and mathematical tables.
\newblock 
 National Bureau of Standards Applied Mathematics Series, 55 For sale by the Superintendent of Documents, U.S. Government Printing Office, Washington, D.C. 1964.

\bibitem{Ar}
F. Aryan.
\newblock Binary and quadratic divisor problems.
\newblock Int. J. Number Theory 13 (2017), no. 6, 1457-1471.



\bibitem{BBLR}
Sandro Bettin, H. M. Bui, Xiannan Li, Maksym M.  Radziwi\l\l.
\newblock A quadratic divisor problem and moments of the Riemann zeta-function.
\newblock  J. Eur. Math. Soc. (JEMS) 22 (2020), no. 12, 3953-3980. 
 



\bibitem{Bl}
V. Blomer. 
\newblock
Shifted convolution sums and subconvexity bounds for automorphic $L$-functions.
\newblock  Int. Math. Res. Not. IMRN 2004,  no. 73, 3905-3926. 

\bibitem{BW}
J. Bourgain and N. Watt.
\newblock
Decoupling for perturbed cones and mean square of $|\zeta(\tfrac{1}{2}+it)|$.
\newblock
Int. Math. Res. Not. IMRN 2018, no. 17, 5219-5296.



\bibitem{Br}
J. Bredberg. 
\newblock Large gaps between consecutive zeros, on the critical line, of the Riemann zeta function.
\newblock preprint,
{\tt arxiv.org/abs/1101.3197}. 

\bibitem{CFKRS2003}
J.B. Conrey, D.W. Farmer, J.P. Keating, M.O. Rubinstein, and N.C. Snaith.
\newblock Autocorrelation of random matrix polynomials.  
\newblock Comm. Math. Phys. 237 (2003), no. 3, 365-395. 

\bibitem{CFKRS}
J.B. Conrey, D.W. Farmer, J.P. Keating, M.O. Rubinstein, and N.C. Snaith.
 \newblock  Integral moments of $L$-functions. 
 \newblock Proc. London Math. Soc. (3)  91  (2005),  no. 1, 33--104.

\bibitem{CFKRS2008}
J.B. Conrey, D.W. Farmer, J.P. Keating, M.O. Rubinstein, and N.C. Snaith. 
\newblock Lower order terms in the full moment conjecture for the Riemann zeta function.
 \newblock J. Number Theory 128 (2008), no. 6, 1516-1554.  
  

\bibitem{CGh} 
J.B. Conrey and A. Ghosh.
 \newblock A conjecture for the sixth power moment of the Riemann zeta-function.
  \newblock  Internat. Math. Res. Notices 1998, no. 15,  775-780.



\bibitem{CG}
J. B. Conrey and S. M. Gonek.
\newblock High moments of the Riemann zeta-function. 
\newblock Duke Math. J. 107 (2001), 577-604. 




\bibitem{CK1} J.B. Conrey and J. P. Keating.
 \newblock Moments of zeta and correlations of divisor-sums: I.
  \newblock Philos. Trans. Roy. Soc. A 373 (2015), no. 2040, 20140313, 11 pp.


\bibitem{CK2} J.B. Conrey and J. P. Keating.
 \newblock Moments of zeta and correlations of divisor-sums: II. 
 \newblock  Advances in the theory of numbers, 75–85, Fields Inst. Commun., 77, Fields Inst. Res. Math. Sci., Toronto, ON, 2015.




\bibitem{CK3} J.B. Conrey and J. P. Keating. 
\newblock Moments of zeta and correlations of divisor-sums: III. 
\newblock Indag. Math. (N.S.) 26 (2015), no. 5, 736-747.
  
\bibitem{CK4} 
 J.B. Conrey and J. P. Keating.
\newblock Moments of zeta and correlations of divisor-sums: IV. 
\newblock
Research in Number Theory,  Res. Number Theory 2 (2016), Paper No. 24, 24 pp. 



\bibitem{Di}
A.  Diaconu. 
\newblock On the third moment of $L(\frac{1}{2},\chi_d)$ I: The rational function field case. 
\newblock J. Number Theory 198 (2019), 1-42. 

\bibitem{DW}
A. Diaconu and  I. Whitehead.  
\newblock On the third moment of $L(\frac{1}{2},\chi_d)$ II: The number field case. 
\newblock J. Eur. Math. Soc. (JEMS) 23 (2021), no. 6, 2051–2070. 

\bibitem{DFI}
W. Duke, J.B. Friedlander, and H. Iwaniec.
\newblock A quadratic divisor problem.
 \newblock 
Invent. Math. 115 (1994), no. 2, 209--217. 

\bibitem{GG}
D. A. Goldston and S.M. Gonek. 
 \newblock  Mean value theorems for long Dirichlet polynomials and tails of Dirichlet series. 
 \newblock  Acta Arith.  84  (1998),  no. 2, 155--192. 

\bibitem{Hall}
R.R. Hall. 
\newblock 
Generalized Wirtinger inequalities, random matrix theory, and the zeros of the Riemann zeta-function.
 \newblock 
J. Number Theory 97 (2002), no. 2, 397-409. 

\bibitem{H}
G. Harcos. 
 \newblock An additive problem in the Fourier coefficients of cusp forms.
  \newblock 
Math. Ann. 326 (2003), no. 2, 347-365. 

\bibitem{HL} 
G.H. Hardy and J.E. Littlewood. 
 \newblock  Contributions to the theory of the Riemann zeta function and
the theory of the distribution of primes.
 \newblock Acta. Math. 41 (1918), 119-196.  

\bibitem{Ha}
A. Harper. 
 \newblock 
Sharp conditional bounds for moments of the Riemann zeta function. 
\newblock preprint,
{\tt arxiv.org/abs/1305.4618}. 

\bibitem{HB}
D. R. Heath-Brown.
  \newblock  The fourth power moment of the Riemann zeta function.
   \newblock  
Proc. London Math. Soc. (3)  38  (1979), no. 3, 385--422.

\bibitem{HB2}
D.R. Heath-Brown. 
\newblock Fractional moments of the Riemann zeta-function. 
\newblock J. London Math. Soc. 24
(1981), 65-78.



\bibitem{He}
K. Henriot. 
\newblock Nair-Tenenbaum bounds uniform with respect to the discriminant. 
\newblock  Math Proceedings of the Cambridge Philosophical Society 
153 (2012), no. 3,  405-424. 

\bibitem{HY}
C.P. Hughes and Matthew P. Young. 
\newblock
 The twisted fourth moment of the Riemann zeta function.
  \newblock J. Reine Angew. Math. 641 (2010), 203-236.


\bibitem{In}
A.E. Ingham. 
 \newblock  Mean-value theorems in the theory of the Riemann zeta function. 
 \newblock Proc. Lond. Math.
Soc. {\bf 27} (1926), 273-300.


\bibitem{Iv} 
 A. Ivi\'{c}.
   \newblock  The Riemann zeta-function. Theory and applications. 
   \newblock Reprint of the 1985 original [Wiley, New York; MR0792089], Dover Publications, Inc., Mineola, NY, 2003. 

\bibitem{Iv1} A. Ivi\'{c},   
\newblock  Lectures on mean values of the Riemann zeta function. 
\newblock
Tata Institute of Fundamental 
Research Lectures on Mathematics and Physics, 82. Published for the Tata Institute of Fundamental Research, Bombay; by Springer-Verlag, Berlin, 1991.
Tata Institute Lectures Notes 82, Springer-Verlag, 1991.

\bibitem{Iv2}
A. Ivi\'{c}.
 \newblock  
On the ternary additive divisor problem and the sixth moment of the zeta-function. 
\newblock in Sieve methods, exponential sums, and their applications in number theory (Cardiff, 1995), 205--243,
London Math. Soc. Lecture Note Ser., 237, Cambridge Univ. Press, Cambridge, 1997. 



\bibitem{IM} A. Ivi\'{c} and Y. Motohashi.  
\newblock On the fourth power moment of the Riemann zeta function. 
\newblock  Journal of Number Theory 51, 16-45 (1995).


\bibitem{KS}
J. P. Keating and N. C. Snaith.
 \newblock Random matrix theory and $\zeta(\tfrac{1}{2}+it)$. 
 \newblock Comm. Math. Phys. 214 (2000) 57-89. 

\bibitem{L}
R. S. Lehman. 
 \newblock On the distribution of zeros of the Riemann zeta-function. 
 \newblock
Proc. Lond. Math. Soc., 3(20):303-320, 1970.

\bibitem{MRT}
K. Matom\"aki, M.  Radziwi\l\l, and T. Tao.
\newblock Correlations of the von Mangoldt and higher divisor functions I. Long shift ranges. 
\newblock 
 Proc. Lond. Math. Soc. (3) 118 (2019), no. 2, 284-350.


\bibitem{MRT2}
K. Matom\"aki,  M. Radziwi\l\l, and T. Tao. 
\newblock Correlations of the von Mangoldt and higher divisor functions II. Divisor correlations in short ranges. 
\newblock
Math. Annalen 374 (1-2), pp. 793-840, 2019.
 

\bibitem{MV}
H.L. Montgomery and R.C. Vaughan. 
 \newblock Multiplicative number theory. I. Classical theory.
 \newblock  Cambridge Studies in Advanced Mathematics, 97. Cambridge University Press, Cambridge, 2007

\bibitem{Mo} Y. Motohashi.
  \newblock  Spectral theory of the Riemann zeta-function. 
  \newblock Cambridge Tracts in Mathematics, 127. Cambridge University Press, Cambridge, 1997, 228 pp. 


\bibitem{NT}
N. Ng and M. Thom.
\newblock Bounds and conjectures for additive divisor sums.
 \newblock  Funct. Approx. Comment. Math. 60 (2019), no. 1, 97-142.

 

\bibitem{R} 
M.  Radziwi\l\l.
\newblock The 4.36th Moment of the Riemann Zeta-Function. 
\newblock
 International Mathematics Research Notices, 2012, no. 18, pp 4245-4259.

\bibitem{RS}
M.  Radziwi\l\l \, and K. Soundararajan. 
\newblock Continuous lower bounds for moments of zeta and L-functions. 
\newblock 
Mathematika 59 (2013), no. 1, 119-128. 

\bibitem{Ra1} K. Ramachandra.
 \newblock Some remarks on the mean value of the Riemann zeta-function and other Dirichlet
series, I. 
\newblock Hardy-Ramanujan J.  1 (1978), 1-15.

\bibitem{Ra2} K. Ramachandra. 
\newblock Some remarks on the mean value of the Riemann zeta-function and other Dirichlet
series, II.
\newblock Hardy-Ramanujan J. 3 (1980), 1-25.

\bibitem{Ra3} K. Ramachandra.
 \newblock Some remarks on the mean value of the Riemann zeta-function and other Dirichlet
series, III.
\newblock Ann. Acad. Sci. Fenn. 5 (1980), 145-180.


\bibitem{RuYa}
M. Rubinstein and S. Yamagishi. 
\newblock Computing the moment polynomials of the zeta function. 
\newblock  Mathematics of Computation 84 (2015), 425-454.


\bibitem{S2}
K. Soundararajan.
\newblock Moments of the Riemann zeta function. 
\newblock Ann. of Math. (2) 170 (2009), no. 2, 981-993.

\bibitem{Tao}
T. Tao.
 \newblock Heuristic computation of correlations of higher order divisor functions. 
 \newblock blog post, 
{\tt https://terrytao.wordpress.com/2016/08/31/heuristic-computation} \\
{\tt-of-correlations-of-higher-order-divisor-functions/}.

\bibitem{Ti}
E.C. Titchmarsh.
  \newblock  The theory of the Riemann zeta function, second
edition.
 \newblock Oxford University Press, New York, 1986.


\bibitem{Zav}
N.I.  Zavorotny\u{\i}.
 \newblock On the fourth moment of the Riemann zeta function (Russian).
  \newblock  
Automorphic functions and number theory, Part I, II (Russian), 69-124a, 254, Akad. Nauk SSSR, Dal'nevostochn. Otdel., Vladivostok, 1989. 

\bibitem{Z} 
Q. Zhang. 
\newblock 
On the cubic moment of quadratic Dirichlet L-functions. 
\newblock 
Math. Res. Lett. 12 (2005), no. 2-3, 413-424. 



\end{thebibliography}


\begin{dajauthors}
\begin{authorinfo}[ng]
  Nathan Ng\\
  University of Lethbridge\\
  Lethbridge, Canada\\
  nathan.ng\imageat{}uleth\imagedot{}ca \\
  \url{https://www.cs.uleth.ca/~nathanng/}
\end{authorinfo}
\end{dajauthors}

\end{document}